\documentclass[final,1p,times]{elsarticle}

\usepackage[colorlinks=true]{hyperref}
\usepackage{amsfonts}
\usepackage{mathrsfs}
\usepackage{amsthm,amsmath,amssymb,mathrsfs,amsfonts,graphicx,graphics,latexsym,exscale,cmmib57,dsfont,bbm,amscd,ulem,curves}
\usepackage{color,soul}
\newtheorem*{LCL}{Lebesgue's covering lemma}
\newtheorem*{prI}{Reflection principle~I}
\newtheorem*{prII}{Reflection principle~II}
\newtheorem*{prIII}{Reflection principle~III}

\newtheorem{thm}{Theorem}[section]
\newtheorem{lem}[thm]{Lemma}

\newtheorem{cor}[thm]{Corollary}

\newtheorem{prop}[thm]{Proposition}

\theoremstyle{definition}
\newtheorem*{notes}{Notes}
\newtheorem*{note}{Note}
\newtheorem{exas}[thm]{Examples}
\newtheorem{exa}[thm]{Example}

\newtheorem{remark}[thm]{Remark}
\newtheorem{remarks}[thm]{Remarks}
\newtheorem{defn}[thm]{Definition}
\newtheorem{noteB}{Note}
\newtheorem{note2.4}{Note}
\newtheorem{note2.6}{Note}

\newtheorem{sn}[thm]{Definition}
\newtheorem{step}{Step}
\newtheorem{que}[thm]{Question}

\newtheorem{no}[thm]{Standing notation}
\newtheorem{snAus}[thm]{Notation}
\newtheorem*{no*}{Standing notation}
\newtheorem{sa}[thm]{Standing assumption}
\newcommand{\z}{\boldsymbol{z}}

\journal{DCDS-A}
\begin{document}

\begin{frontmatter}
\title{Minimality, distality and equicontinuity for semigroup actions on compact Hausdorff spaces}

\address{$($Dedicated to the memory of Professor Isaac Namioka$)$}

\author{Joseph Auslander}
\ead{jna@math.umd.edu}
\address{Department of Mathematics, University of Maryland, College Park, MD 20742, U.S.A.}

\author{Xiongping Dai}
\ead{xpdai@nju.edu.cn}
\address{Department of Mathematics, Nanjing University, Nanjing 210093, People's Republic of China}

\begin{abstract}
Let $\pi\colon T\times X\rightarrow X$ with phase map $(t,x)\mapsto tx$, denoted $(\pi,T,X)$, be a \textit{semiflow} on a compact Hausdorff space $X$ with phase semigroup $T$. If each $t\in T$ is onto, $(\pi,T,X)$ is called surjective; and if each $t\in T$ is 1-1 onto $(\pi,T,X)$ is called invertible and in latter case it induces $\pi^{-1}\colon X\times T\rightarrow X$ by $(x,t)\mapsto xt:=t^{-1}x$, denoted $(\pi^{-1},X,T)$. In this paper, we show that $(\pi,T,X)$ is equicontinuous surjective iff it is uniformly distal iff $(\pi^{-1},X,T)$ is equicontinuous surjective. As applications of this theorem, we also consider the minimality, distality, and sensitivity of $(\pi^{-1},X,T)$ if $(\pi,T,X)$ is invertible with these dynamics. We also study the pointwise recurrence and Gottschalk's weak almost periodicity of $\mathbb{Z}$-flow with compact zero-dimensional phase space.
\end{abstract}

\end{frontmatter}

\tableofcontents
\section{Introduction}\label{sec1}
Let $T$ be a \textit{topological semigroup} with identity $e$; that is, $T$ is a $T_2$-space, meanwhile it is a multiplicative semigroup with $te=et=t$ for all $t\in T$ such that the binary operation $(s,t)\mapsto st$ of $T\times T$ to $T$ is continuous.
Let $X$ be a non-empty compact $T_2$-space, unless stated otherwise, in this paper. Given $A\subset X$ by $\textrm{Int}_XA$ and $\textrm{cls}_XA$ we will denote respectively the interior and closure of $A$ relative to the space $X$. We will write $\varDelta_X=\{(x,x)\,|\,x\in X\}$ for the diagonal set of $X\times X$.

We say that $\pi\colon T\times X\rightarrow X, (t,x)\mapsto tx$ is a \textit{semiflow} or \textit{transformation semigroup} \cite{G,EEN} with phase space $X$ and with phase semigroup $T$, denoted $(T,X)$, if the phase map $(t,x)\mapsto tx$ is jointly continuous from $T\times X$ to $X$ such that
$$ex=x\ \forall x\in X\quad \textrm{and}\quad
t(sx)=(ts)x\ \forall s,t\in T, x\in X.$$ 
Here $\pi\colon (t,x)\mapsto tx$ is called the phase map of $(T,X)$.
When $T$ is a topological group, i.e., $G$ is a group such that $(s,t)\mapsto st^{-1}$ of $T\times T$ onto $T$ is continuous, then we shall call $(T,X)$ a \textit{flow} or \textit{transformation group} with the phase group $T$ (cf.~\cite{GH,E69,Fur,Aus,EE}).

Given any integer $k\ge2$, write $X^k=X\times\dotsm\times X$ ($k$-times) and let $(T,X^k)$ stand for the product semiflow with phase map $(t,(x_1,\dotsc,x_k))\mapsto (tx_1,\dotsc,tx_k)$.

\begin{no}\label{no1.1}
Given $(T,X), x\in X$ and subsets $A,U,V$ of $X$, we write
\begin{enumerate}
\item $Tx=\{tx\,|\,t\in T\}$,
\item $TA={\bigcup}_{t\in T}tA={\bigcup}_{x\in A}Tx$,
\item $N_T(x,U)=\{t\in T\,|\,tx\in U\}$,
\item ${N_T(U,V)}=\{t\in T\,|\,U\cap t^{-1}V\not=\emptyset\}$,
\item $t^{-1}x=\{y\in X\,|\,ty=x\}$ for all $t\in T$.
\end{enumerate}
\end{no}

\begin{no}\label{no1.2}
Let $\mathscr{U}_X$ be the compatible symmetric uniform structure of the compact $T_2$-space $X$; then for all $\varepsilon\in\mathscr{U}_X$, $x\in X$, and $B\subset X$, we will write
\begin{enumerate}
\item[1)] $\varepsilon[B]={\bigcup}_{x\in B}\varepsilon[x]$ where $\varepsilon[x]=\{y\in X\,|\,(x,y)\in\varepsilon\}$;
\item[2)] $T(\varepsilon[x],x)={\bigcup}_{t\in T}t(\varepsilon[x]\times\{x\})$; and
\item[3)]  for all $n\ge2$, $\frac{\varepsilon}{n}$ stands for an entourage in $\mathscr{U}_X$ with $(\frac{\varepsilon}{n})^n\subset\varepsilon$. In other words, if $x,z\in X$ are such that
$(x,y_1)\in\frac{\varepsilon}{n},  (y_1,y_2)\in\frac{\varepsilon}{n}, \dotsc, (y_{n-2},y_{n-1})\in\frac{\varepsilon}{n}$, and $(y_{n-1},z)\in\frac{\varepsilon}{n}$
then $(x,z)\in\varepsilon$.
\end{enumerate}
See, e.g.,~\cite{Kel, GH} and \cite[Appendix~II]{Aus}.
\end{no}

\begin{no}\label{sn1.3}
\begin{enumerate}
\item $(T,X)$ is \textit{surjective} if and only if each $t\in T$ is an onto self-map of $X$, i.e., $tX=X$ for all $t\in T$.
\item $(T,X)$ is called \textit{effective} if $t\not=e$ implies $tx\not=x$ for some $x\in X$.

\item $(T,X)$ is \textit{invertible} iff each $t\in T$ is 1-1 onto. In this case, by $\langle T\rangle$ we will denote the smallest group of self-homeomorphisms of $X$ containing $T$, and then $(\langle T\rangle,X)$ is a flow.
\end{enumerate}
\end{no}

However, it should be noted that since $T$ is in general neither a right-syndetic nor a normal subsemigroup of $\langle T\rangle$, $(T,X)$ and $(\langle T\rangle,X)$ do generally not possess the same dynamics. In fact, contrary to what one might hope or expect, the passage from group to the semigroup case is not straightforward in many important cases we will consider here; cf., e.g.,~\cite{EEN}.

Benjamin Weiss had pointed out an example of a minimal action of a group $\langle T\rangle$ with a generating subsemigroup $T$ whose action is not minimal (cf.~\cite[p.~3062]{AA}). We can, however, show that if $T$ is amenable here, Weiss' case does not occur (see Reflection principle~II below).

On the other hand, although $(\langle T\rangle, X)$ is minimal if $(T,X)$ is so; yet a minimal subset of $T$ need not be a minimal set of $\langle T\rangle$. Let us see an explicit example for this as follows.

\begin{exas}\label{ex1.4}
\begin{enumerate}
\item \textit{There exists an invertible semiflow $(T,X)$ such that there are points of $X$ which are minimal for $(T,X)$ but not for $(\langle T\rangle,X)$.}

\begin{proof}
Indeed, let $X=[-1,2]$ with the usual topology and for every $\alpha$ with $0<\alpha<1$, define two self homeomorphisms of $X$ as follows:
\begin{gather*}
f_\alpha\colon X\rightarrow X,\
\begin{cases}x\mapsto x& \textrm{if }-1\le x\le 0,\\
x\mapsto\alpha x& \textrm{if }0\le x\le 1,\\
x\mapsto(2-\alpha)x+2(\alpha-1)& \textrm{if }1\le x\le 2;\end{cases}\quad
\intertext{and}
g_\alpha\colon X\rightarrow X,\
\begin{cases}
x\mapsto(2-\alpha)x+(1-\alpha)& \textrm{if }-1\le x\le 0,\\
x\mapsto 1-\alpha(1-x) & \textrm{if }0\le x\le 1,\\
x\mapsto x & \textrm{if }1\le x\le 2.\end{cases}
\end{gather*}
Now let $T=\langle f_\alpha,g_\alpha\,|\,0<\alpha<1\rangle_+$ be the discrete semigroup generated by $\{f_\alpha\,|\, 0<\alpha<1\}$ and $\{g_\alpha\,|\, 0<\alpha<1\}$. It is easy to see that each $t\in T$ is bijective so that $(T,X)$ is invertible. And $\Lambda=[0,1]$ is minimal for $(T,X)$ so that every point of $\Lambda$ is minimal for $(T,X)$ but not for $(\langle T\rangle,X)$. In fact, since for every $x\in\Lambda$ we have $\textrm{cls}_X\langle T\rangle x=X$, and $-1$ and $2$ are fixed points, which are the only minimal points of $(\langle T\rangle,X)$, thus $x\in\Lambda$ is not minimal for $(\langle T\rangle,X)$.
Here each $t\in T\setminus\{e\}$ restricted to $\Lambda$ is not surjective.
\end{proof}

\item {\it There is an invertible semiflow $(T,X)$ on a compact metric space $X$ such that each orbit $Tx$ is not dense in $X$ but $(\langle T\rangle,X)$ has a residual set of points that have dense orbits.}

\begin{proof}
In fact, $(T,X)$ above is such a semiflow. The orbit of every point $x\in(-1,2)$ is not dense in $X$ for $(T,X)$ but dense for $(\langle T\rangle,X)$.
\end{proof}

\item Notice that there are points $x\in X$ such that $\textrm{Int}_XTx\not=\emptyset$ in 1. above. However, if we only consider the rational $\alpha$ with $0<\alpha<1$, then no point $x\in X$ such that $\textrm{Int}_X\langle T\rangle x\not=\emptyset$.
\end{enumerate}
\end{exas}

In applications of the topological dynamical systems theory, in fact we are often concerned with only semiflows, not flows. For example, for a flow $\pi\colon\mathbb{R}\times M^n\rightarrow M^n, (t,x)\mapsto tx$ on a manifold $M^n$ induced by a vector field, we are usually interested to the dynamics like recurrence and almost periodicity of trajectories as $t\to+\infty$ or $t\to-\infty$, not $|t|\to+\infty$. In this case, we need essentially to consider the invertible semiflow $\pi_+\colon\mathbb{R}_+\times M^n\rightarrow M^n$ (cf., e.g.,~\cite{SS74, SS, SSY, SY}).

Although dynamics on a metric phase space is an important case, yet in many interesting cases we have to face with non-metric phase spaces. For example, the universal dynamics are usually defined on compact $T_2$ non-metrizable phase spaces (cf.~\cite{E60, Ch62, C63-D, E69, G76, Aus, DG, ACD}). The Stone-\v{C}ech compactification $\beta T$ of the phase group or semigroup $T$ is also an important phase space which is compact $T_2$ non-metrizable in general.

In this paper, we shall be mainly concerned with the dynamics\,---\,equicontinuity, distality, minimality, sensitivity, and weak almost periodicity, each of which is one of the most fruitful notions in the abstract theory of topological dynamics.

These dynamics are all independent of the topology on the phase semigroup $T$. Thus henceforth, unless specified otherwise, we will assume in our later discussion:

\begin{sa}
\begin{enumerate}
\item The phase semigroup of any semiflow is a discrete infinite semigroup with identity element $e$. In this case, every compact subset of $T$ is finite.

\item The phase space of any semiflow is always assumed to be a non-empty compact $T_2$-space with the compatible symmetric uniform structure $\mathscr{U}$.
\end{enumerate}
\end{sa}

It turns out that many of our semiflow results can be useful for studying flows; for example, $\S\ref{sec7.2}$ and $\S\ref{uap}$.

\medskip
The authors would like to thank Professors Ethan~Akin Xiangdong~Ye for their much motivating discussion.

\subsection{Basic notions and preliminary lemmas}\label{sec1.1}
Let $(T,X)$ be a semiflow with phase semigroup $T$ and with phase map $\pi\colon(t,x)\mapsto tx$. We then first introduce and unify the most basic and important dynamics notions needed throughout in our later arguments.

\subsubsection{Equicontinuity by $\varepsilon$-$\delta$}\label{sec1.1.1}
\begin{description}
\item[(a)] $(T,X)$ is called \textit{equicontinuous} in case given $\varepsilon\in\mathscr{U}_X$ there exists a $\delta\in\mathscr{U}_X$ such that for all $t\in T$, $t(x,y)\in\varepsilon$ if $x,y\in X$ with $(x,y)\in\delta$.

\newpage
\item[(b)] We say $(T,X)$ is \textit{equicontinuous at $x\in X$}, denoted $x\in\mathrm{Equi}\,(T,X)$, if for all $\varepsilon\in\mathscr{U}_X$ there is a $\delta\in\mathscr{U}_X$ such that $t(\delta[x])\subset\varepsilon[tx]$ for all $t\in T$; or equivalently, $T(\delta[x],x)\subset\varepsilon$.
\end{description}
\begin{note}
These notions may also be defined for semiflows with phase spaces that are non-compact uniform spaces; see \cite[$\S$11]{GH}.
\end{note}

By (\textbf{a}), the equicontinuity of $(T,X)$ is independent of the topology of the phase semigroup $T$.
In addition, since here $X$ is a compact $T_2$-space, thus it holds that:

\begin{LCL}[{cf.~\cite[Theorem~5.27]{Kel}}]
If $\mathscr{V}$ is an open cover of $X$, then there exists a ``Lebesgue index'' $\delta\in\mathscr{U}_X$ such that given $x\in X$, $\delta[x]\subseteq V$ for some $V\in\mathscr{V}$.
\end{LCL}

From Lebesgue's covering lemma above and general topology, we can then easily obtain the following basic characterizations of equicontinuity of semiflows:

\begin{lem}[{cf.~\cite[11.06, 11.09, 11.12, 11.23]{GH} for general function spaces}]\label{lem1.6}
Let $(T,X)$ be a semiflow with phase map $(t,x)\mapsto tx$. Then the following statements are pairwise equivalent:
\begin{enumerate}
\item[$(1)$] $(T,X)$ is equicontinuous.
\item[$(2)$] $\mathrm{Equi}\,(T,X)=X$.
\item[$(3)$] If $\alpha\in\mathscr{U}_X$, then there exists a finite partition $\mathscr{X}$ of $X$ such that $A\in\mathscr{X}$ and $t\in T$ implies $tA\times tA\subseteq\alpha$.
\item[$(4)$] $(T,X)$ is totally bounded with $X^X$ in its space-index uniformity; that is, to each $\alpha\in\mathscr{U}_X$ there corresponds a finite subset $K$ of $T$ such that for every $t\in T$, $(tx,kx)\in\alpha\ \forall x\in X$, for some $k\in K$.
\item[$(5)$] If $\alpha\in\mathscr{U}_X$, then there exists a finite partition $\mathscr{T}$ of $T$ such that $B\in\mathscr{T}$ and $x\in X$ implies $Bx\times Bx\subseteq\alpha$.
\item[$(6)$] The orbit space $\{(tx)_{t\in T}\,|\,x\in X\}$ is totally bounded in $X^T$ provided with its space-index uniformity; i.e., to each $\alpha\in\mathscr{U}_X$ there corresponds a finite subset $K$ of $X$ such that for every $x\in X$, $(tx,tx_k)\in\alpha\ \forall t\in T$, for some $x_k\in K$.
\end{enumerate}
\begin{note}
We will show in $\S\ref{uap}$ that if it is surjective $(T,X)$ is equicontinuous iff it is almost periodic~(cf.~Theorem~\ref{u8.3}). Thus Lemma~\ref{lem1.6} is in fact a generalization of \cite[Theorem~4.38]{GH} from flows to semiflows. However, our proof presented here is self-contained and transparent.
\end{note}
\end{lem}

\begin{proof}
$(1)\Leftrightarrow(2)$: Let $\mathrm{Equi}\,(T,X)=X$. For $\varepsilon\in\mathscr{U}_X$ and $x\in X$, there is a $\delta_x\in\mathscr{U}_X$ with $T(\delta_x[x],x)\in\varepsilon/2$. Since $X$ is compact, by the Lebesgue covering lemma there exists a $\delta\in\mathscr{U}_X$ such that for all $y\in X$, $\delta[y]\subseteq\delta_x[x]$ for some $x\in X$. So by triangle inequality, $T(\delta[y],y)\subset\varepsilon$. Since $\varepsilon$ and $y$ are arbitrary, $(T,X)$ is equicontinuous. $(1)\Rightarrow(2)$ is obvious.

$(1)\Rightarrow(3)$: Let $\alpha\in\mathscr{U}_X$. There is $\delta\in\mathscr{U}_X$ such that $T\delta\subseteq\alpha$. Since $X$ is compact, we can take some $\beta\in\mathscr{U}_X$ such that $\beta[x]\times\beta[x]\subseteq\delta$ for all $x\in X$. Thus $t(\beta[x])\times t(\beta[x])\subseteq\alpha$ for all $t\in T$ and $x\in X$. This implies (3).

$(3)\Rightarrow(4)$: Let $\alpha\in\mathscr{U}_X$ and let $\varepsilon\in\mathscr{U}_X$ such that $\varepsilon[x]\times\varepsilon[x]\subseteq\alpha$ for all $x\in X$. Then there exists a finite partition $\mathscr{X}$ of $X$ such that $tA\times tA\subseteq\frac{\varepsilon}{3}$ for all $t\in T$ and $A\in\mathscr{X}$.
Select a finite subset $X_0$ of $X$ such that $X=\bigcup_{x_0\in X_0}\frac{\varepsilon}{3}[x_0]$. If $A\in\mathscr{X}$ and if $t\in T$, then there exists $x_0\in X_0$ such that $tA\cap\frac{\varepsilon}{3}[x_0]\not=\emptyset$ whence $tA\subseteq\varepsilon[x_0]$. Each $t\in T$ thus determines a mapping $t^*\colon\mathscr{X}\rightarrow X_0$ such that $tA\subseteq\varepsilon[t^*A]$ for all $A\in\mathscr{X}$. Since $X_0^\mathscr{X}$ is finite, there exists a finite subset $K$ of $T$ for which $\{t^*\,|\,t\in T\}=\{k^*\,|\,k\in K\}$. Now for any $t\in T$, there is some $k\in K$ with $t^*=k^*$ such that to each $x\in X$,
$(tx,kx)\in(tA,kA)\subseteq\varepsilon[k^*A]\times\varepsilon[k^*A]\subseteq\alpha$
for some $A\in\mathscr{X}$ with $x\in A$. Thus $(T,X)$ is totally bounded.

$(4)\Rightarrow(1)$: Let $\varepsilon\in\mathscr{U}_X$. Then there is a finite subset $K$ of $T$ such that for every $t\in T$, $(tx,kx)\in\frac{\varepsilon}{3}\ \forall x\in X$, for some $k\in K$. Since for each $t\in T$, $x\mapsto tx$ is uniformly continuous, so there is $\delta\in\mathscr{U}_X$ such that $(x,y)\in\delta$ implies $(kx,ky)\in\frac{\varepsilon}{3}$ for all $k\in K$. Then by triangle inequality, $T\delta\subseteq\varepsilon$. Thus (1) holds.

$(1)\Rightarrow(5)$: We first note that $C(X,X)$ is complete in its space-index uniformity. Then by (4), $E:=\mathrm{cls}_{C(X,X)}T$ is compact. Thus, $(p,x)\mapsto p(x)$ of $E\times X$ to $X$ is continuous whence uniformly continuous. This implies (5).

$(5)\Rightarrow(6)$: Let $\alpha\in\mathscr{U}_X$. By changing the roles of $t$ and $x$ in the above proof of ``$(3)\Rightarrow(4)$'' we can obtain that:
There exists a finite subset $K$ of $X$, for which to every $x\in X$ there corresponds $x_k\in K$ such that $(tx,tx_k)\in\alpha$ for all $t\in T$.
Thus (6) holds.

$(6)\Rightarrow(4)$: Let $\alpha\in\mathscr{U}_X$. By condition (6), we can define a finite partition $\mathscr{X}=\{A_k\,|\,k\in K\}$ of $X$, where $A_k=\{x\in X\,|\,(tx,tx_k)\in\frac{\alpha}{3}\,\forall t\in T\}$, such that if $A\in\mathscr{X}$ then there is some $x_k\in K$ with $(tA,tx_k)\in\frac{\alpha}{3}$ so that $tA\times tA\subseteq\alpha$ for all $t\in T$.

The proof of Lemma~\ref{lem1.6} is thus completed.
\end{proof}

\subsubsection{Minimality}\label{sec1.1.2}
\begin{description}
\item[(c)] A subset $A$ of $X$ is \textit{invariant} if $Tx\subseteq A$ for all $x\in X$, or equivalently, $TA\subseteq A$. It is \textit{negatively invariant} if $t^{-1}x\subseteq A$ for all $x\in A$ and $t\in T$.

\item[(d)] A subset $\Lambda$ of $X$ is referred to as \textit{minimal} if $\Lambda$ is a non-empty, closed, and invariant set containing no proper subsets with those properties.
If $X$ itself is minimal, then we call $(T,X)$ a \textit{minimal semiflow}.
\item[(e)] An $x\in X$ is called a \textit{minimal point} if $\textrm{cls}_X{Tx}$ is a minimal set of $(T,X)$. If every point of $X$ is minimal, then $(T,X)$ is called \textit{pointwise minimal}.
\end{description}

Clearly, the minimality is also independent of the topology of $T$; and $\Lambda$ is minimal if and only if it is exactly the orbit closure of each of its points.

Moreover, if $(T,X)$ is minimal, then $\textrm{cls}_X(T\setminus\{e\})X=X$ but there is no the property that $tX=X\ \forall t\in T$ in general.
Let's see such a counterexample, which is motivated by Furstenberg's \cite[p.~40]{Fur} for the case that $\alpha=1/2$.

\begin{exa}\label{ex1.7}
Let $X=[0,1]$ be the unit interval with the usual topology and for each $\alpha$ with $0<\alpha<1$, define two injective mappings of $X$ into itself as follows:
\begin{gather*}
f_\alpha\colon X\rightarrow X,\ x\mapsto\alpha x\quad{\textrm{and}}\quad
g_\alpha\colon X\rightarrow X,\  x\mapsto 1-\alpha(1-x).
\end{gather*}
Now let $T=\langle f_\alpha,g_\alpha\,|\,0<\alpha<1\rangle_+$ be the discrete free semigroup generated by $\{f_\alpha,g_\alpha\,|\,0<\alpha<1\}$. It is easy to see that
each $t\in T$ is injective and $(T,X)$ is equicontinuous minimal with $\textrm{cls}_X(T\setminus\{e\})X=X$, but each $t\in T\setminus\{e\}$ is not surjective.
\end{exa}

However, we will show that this is actually in the affirmative for some special phase semigroups; see Propositions~\ref{prop3.7B} and \ref{prop3.16B} and Corollary~\ref{cor3.18B}.
In addition minimality will be equivalently described by `almost periodicity'; see Lemma~\ref{lem2.6} in $\S\ref{sec2.2}$.

\subsubsection{Distality, proximity and regional proximity}\label{sec1.1.3}
The concept of ``distality'' has been proved to be a very fruitful one for topological dynamics of flows, giving rise to a rather extensive theory; see \cite{E69, G76, Aus}. In fact it is also important in semiflows.
\begin{description}
\item[(f)] We say that $x\in X$ is \textit{proximal} to $y\in X$, write $(x,y)\in P(T,X)$ or $P(X)$ or $y\in P[x]$, if there exist a net $\{t_n\}$ in $T$ and a point $z\in X$ with $t_n(x,y)\to(z,z)$.
By definition, it holds that
\begin{gather*}P(X)={\bigcap}_{\alpha\in\mathscr{U}_X}{\bigcup}_{t\in T}t^{-1}\alpha.\end{gather*}

\item[(g)] $(T,X)$ is called \textit{distal} if for all $x,y\in X$ with $x\not=y$, one can find some $\alpha\in\mathscr{U}_X$ with $t(x,y)\not\in\alpha$ for every $t\in T$.
\item[(h)] An $x\in X$ is called a \textit{distal point} of $(T,X)$ if there exists no point other than itself in $\textrm{cls}_XTx$ to be proximal to it under $(T,X)$.

\item[(i)] $(T,X)$ is called a \textit{point-distal semiflow} if there exists a point $x\in X$ such that $x$ is a distal point of $(T,X)$ and $Tx$ is dense in $X$ (cf.~Veech~\cite{V70}).
\end{description}

Clearly, if $x_i$ is a distal point of $(T,X_i)$, then $(x_i)$ is a distal point of $(T,\prod X_i)$. If $(T,X)$ is a point-distal \textit{flow} with $x$ a distal point, then each of $Tx$ is a distal point. This is also true in surjective semiflows as follows:

\begin{lem}\label{lem1.8}
Let $(T,X)$ be a point-distal semiflow with a distal point $x\in X$. Then $(T,X)$ is surjective if and only if $tx$ is a distal point of $(T,X)$ for all $t\in T$.
\end{lem}

\begin{proof}
The ``only if'' part. First $s^{-1}(sx)=x$ for all $s\in T$. This implies that $sx$ is distal for $(T,X)$, for all $s\in T$. Indeed, otherwise, there is some $y\in\textrm{cls}_XTsx$ with $y\not=sx$ such that $(y,sx)\in P(X)$; so $(z,x)\in P(X)$ and $z\not=x$, for all $z\in s^{-1}y\not=\emptyset$; this is a contradiction.
We will postpone the proof of the ``if'' part of Lemma~\ref{lem1.8} in $\S\ref{sec6.1}$ using Ellis' semigroup.
\end{proof}

It should be mentioned that a distal point $x$ does not satisfy that every $y\in X\setminus\textrm{cls}_XTx$ is not proximal to $x$ under $(T,X)$ unless $(T,X)$ is pointwise minimal. For example, let $f\colon I\rightarrow I$, $x\mapsto x^2$ where $I=[0,1]$; then $x=0$ is a distal point but $(x,y)\in P(X)$ for all $0<y<1$.
Moreover, it is evident that
\begin{quote}
{\it $(T,X)$ is distal \textit{iff} $P(X)=\varDelta_X$} (cf.~\cite[Lemma~5.12]{E69} for $T$ in groups).
\end{quote}
Thus a minimal semiflow is distal at a point $x$ if and only if the proximal cell $P[x]=\{x\}$.

By using Ellis' semigroup (cf.~\cite{Fur, Aus}), it is a well-known fact that {\it a distal point is a minimal point for any semiflow}. Thus:
    \begin{quote}
    {\it If $(T,X)$ is point-distal it is minimal. Moreover, $(T,X)$ is distal iff every point of $X$ is distal for $(T,X)$.}
    \end{quote}

Using ``IP$^*$-recurrence'' and the ``central set'' notion, Furstenberg characterized distality of a point by product minimality as follows:
\begin{quote}
{\it A point $x$ is distal for $(T,X)$ if and only if for every minimal point $y$ of any semiflow $(T,Y)$, $(x,y)$ is a minimal point of $(T,X\times Y)$} (cf.~Furstenberg \cite[(i) and (iv) of Theorem~9.11]{Fur}).
\end{quote}

In fact, by using a purely topological proof\,---\,maximal set of almost periodic points independently of Furstenberg's theorem and Ellis' semigroup, we can characterize distal point by the product minimality as follows:

\begin{quote}
{\it An $x\in X$ is a distal point of $(T,X)$ if and only if $(x,x^\prime)$ is a minimal point of $(T,X\times X)$ for all minimal point $x^\prime$ of $(T,X)$} (cf.~Theorem~\ref{thm5.1}).
\end{quote}

Comparing with Furstenberg's,  our characterization does not need to utilize other semiflow $(T,Y)$. This theorem will be proved in $\S\ref{sec5}$ in terms of almost periodic points.

Notice that in general the entourage $\alpha$ in $\S\ref{sec1.1.3}$\,(\textbf{g}) depends on the pair $(x,y)$. In view of this, we now introduce the concept of ``uniformly distal'' as follows:

\begin{sn}[{cf.~\cite{SS} for $T=\mathbb{\mathbb{R}}$}]\label{def1.9}
\begin{enumerate}
\item $(T,X)$ is \textit{uniformly distal at a point $x\in X$} if given $\varepsilon\in\mathscr{U}_X$ there exists a $\delta\in\mathscr{U}_X$ such that if $y\in X$ with $y\not\in\varepsilon[x]$ then $ty\not\in\delta[tx]$ for all $t\in T$.
\item We say $(T,X)$ is \textit{uniformly distal} if given $\varepsilon\in\mathscr{U}_X$ there exists a $\delta\in\mathscr{U}_X$ such that if $(x,y)\not\in\varepsilon$ then $(tx,ty)\cap\delta=\emptyset$, for all $x,y\in X$ and $t\in T$.
\end{enumerate}
We will show that when $(T,X)$ is minimal, then $(T,X)$ is uniformly distal at every point of $X$ iff it is uniformly distal; see Lemma~\ref{lem1.10} below.
\begin{note}
This notion may be defined on non-compact uniform spaces by the same means. Moreover, it can be extended to \textit{set-valued} semiflows; that is, $tx$ is a subset of $X$ for $t\in T$ and $x\in X$.
\end{note}
\end{sn}

Clearly, a uniformly distal point is a distal point on any uniform $T_2$ space, because if for $x,y\in X$ with $x\not=y$ there is an $\varepsilon\in\mathscr{U}_X$ such that $(x,y)\not\in\varepsilon$. So uniformly distal is distal.
In fact, it is easy to verify that
\begin{description}
\item[($\star$)] {\it Let $X$ be a uniform space not necessarily compact. Then $\pi\colon T\times X\rightarrow X$ s.t. $(t,x)\mapsto tx$ is equicontinuous iff $\pi^{-1}\colon X\times T\rightarrow X$ s.t. $(x,t)\mapsto xt:=t^{-1}x$ is uniformly distal. 
    
    Thus a flow $(T,X)$ is equicontinuous iff it is uniformly distal.}
\end{description}

\begin{proof}
Let $\pi\colon T\times X\rightarrow X$ be  equicontinuous with $\varepsilon$-$\delta$ as in $\S\ref{sec1.1.1}$\,(\textbf{a}). Then if $(x,y)\not\in\varepsilon$, then $(xt,yt)=(t^{-1}x,t^{-1}y)$ is disjoint with $\delta$ for all $t\in T$. So $X\times T\rightarrow X$ is uniformly distal.

Conversely, assume $\pi^{-1}\colon X\times T\rightarrow X$ is uniformly distal with $\varepsilon$-$\delta$ as in Definition~\ref{def1.9}. It is obvious that $(x,y)\in\delta$ implies $(tx,ty)\in\varepsilon$; since otherwise $(x,y)\in(tx,ty)t\cap\delta=\emptyset$. Thus $T\times X\rightarrow X$ is equicontinuous.

The second part of ($\star$) follows at once from the fact that $T=T^{-1}$ for $T$ is a group in the flow $(T,X)$. The proof is completed.
\end{proof}

Notice that the group structure of $T$ plays a role in both of the ``if'' and ``only if'' parts of the second part of ($\star$). However, since there is no $T=T^{-1}$ for a general semiflow with $T$ not a group, hence according to Example~\ref{ex1.7} ``Equicontinuous $\Leftrightarrow$ Uniformly distal'' is not obvious for semiflows with which we will be mainly concerned. See Theorem~\ref{thm1.14} below.

Let $(T,X)$ be an arbitrary semiflow. Next we will introduce another important relation which is weaker than proximity on $X$.
\begin{description}
\item[(j)] We say that $x\in X$ is \textit{regionally proximal} to $y\in X$, denoted $(x,y)\in Q(T,X)$ or $Q(X)$, if there are nets $\{x_n\}, \{y_n\}$ in $X$ and $\{t_n\}$ in $T$ such that $t_n(x_n,y_n)\to(z,z)$ for some $z\in X$.
\end{description}
Clearly,
\begin{gather*}
Q(X)={\bigcap}_{\alpha\in\mathscr{U}_X}\mathrm{cls}_{X\times X}{\bigcup}_{t\in T}t^{-1}\alpha
\end{gather*}
is a closed symmetric reflexive relation on $X$.
It is clear that $P(X)\subseteq Q(X)$ and so if $Q(X)=\varDelta_X$, then $(T,X)$ is distal.

It is already known that if the proximal cell $P[x]=\{x\}$ then $x$ is a distal point of $(T,X)$; if $P(X)=\varDelta_X$ then $(T,X)$ is distal. \textit{What can we say for $Q[x]=\{x\}$ and $Q(X)=\varDelta_X$?}
We can then obtain the following facts:

\begin{lem}\label{lem1.10}
Let $(T,X)$ be any semiflow with phase map $(t,x)\rightarrow tx$ and $x_0\in X$; then the following statements hold:
\begin{enumerate}
\item[$(1)$] If $Q(T,X)=\varDelta_X$, then $X\times T\rightarrow X$ with phase map $(x,t)\mapsto xt:=t^{-1}x$ is equicontinuous.
\item[$(2)$] $Q(T,X)=\varDelta_X$ iff $(T,X)$ is uniformly distal.
\item[$(3)$] Let $A[x_0]=\{y\in X\,|\,\exists\, t_n\in T\textrm{ and }y_n\to y\textrm{ s.t. }\lim_nt_n(x,y_n)\in\varDelta_X\}$. Then $A[x_0]=\{x_0\}$ iff $(T,X)$ is uniformly distal at $x_0$.
\item[$(4)$] Let $(T,X)$ be minimal. Then $(T,X)$ is uniformly distal iff it is uniformly distal at every point of $X$.
\end{enumerate}
\end{lem}

\begin{proof}
(1). Given $\varepsilon\in\mathscr{U}_X$, by $Q(T,X)\subset\varepsilon$ and by the finite intersection property of $X$, there is some $\delta\in\mathscr{U}_X$ such that $\mathrm{cls}_{X\times X}{\bigcup}_{t\in T}t^{-1}\delta\subseteq\varepsilon$. Thus $\delta t\subseteq\varepsilon$ for all $t\in T$. This shows that $X\times T\rightarrow X$ is equicontinuous.

(2). This statement follows easily from the foregoing ($\star$) and (1) of Lemma~\ref{lem1.10}.

(3). Assume $A[x_0]=\{x_0\}$. If $(T,X)$ were not uniformly distal at $x_0$, then there would exist $\varepsilon\in\mathscr{U}_X$, $x_n\not\in\varepsilon[x_0]$ and $t_n\in T$ such that $\lim_n(t_nx_n,t_nx_0)\in\varDelta_X$. As $X$ is compact, let $x_n\to y$ by passing to a subnet of $\{x_n\}$ if necessary; then $y\in A[x_0]$ with $x_0\not=y$, a contradiction. Conversely, if $(T,X)$ is uniformly distal at $x_0$ and $y\in A[x_0]$, then $y=x_0$. Thus $A[x_0]=\{x_0\}$.

(4). The necessity is obvious. Now suppose $(T,X)$ is uniformly distal at every point of $X$. Then $(T,X)$ is distal so that $P(T,X)=\varDelta_X$ and by (3) of Lemma~\ref{lem1.10}, $A[x]=\{x\}$ for all $x\in X$.
Thus $(T,X)$ is pointwise almost automorphic\footnote{Let $(T,X)$ be invertible. An $x\in X$ is an almost automorphic point~\cite{V77} if $t_nx\to y$ and $t_n^{-1}y\to x^\prime$ implies that $x=x^\prime$. Since $(t_nx,t_n^{-1}y)\to(y,x^\prime)$ implies that $x^\prime\in A[x]$, hence if $A[x]=\{x\}$ then $x$ is almost automorphic.} so that it is equicontinuous (cf.~\cite[Lemma~5.2 and Proposition~5.5]{DX}). Hence $Q(T,X)=P(T,X)=\varDelta_X$. Finally (2) of Lemma~\ref{lem1.10} follows that $(T,X)$ is uniformly distal.

The proof of Lemma~\ref{lem1.10} is thus completed.
\end{proof}

It should be noticed that although $P(X)$ and $Q(X)$ both are reflexive symmetric relations on $X$, yet if $T$ is a non-abelian semigroup they need not be invariant in our semigroup setting. In view of this, even if $P(X)$ and $Q(X)$ are closed equivalence relations on $X$, $(T,X/P)$ and $(T,X/Q)$ do not need to make sense in general semiflows.

Of course there always exist minimal invariant closed equivalence relations $S_d$ and $S_{eq}$ on $X$ containing $P(X)$ and $Q(X)$, respectively, so that $(T,X/S_d)$ and $(T,X/S_{eq})$ are respectively the maximal distal and equicontinuous factors of $(T,X)$.

\subsubsection{Amenability and \textit{C}-semigroup}\label{sec1.1.4}
It is known that the structure of a topological semigroup is closely related to some dynamics of its actions; see e.g. \cite{DG}. We will consider here two kinds of phase semigroups as follows.
\begin{description}
\item[(k)] A discrete semigroup $T$ is called \textit{amenable} if every semiflow $(T,Y)$ with the phase semigroup $T$ permits an invariant Borel probability measure, i.e., there is a Borel probability measure $\mu$ on $Y$ such that $\mu(B)=\mu(t^{-1}B)\ \forall t\in T$ for all Borel subset $B\subseteq Y$ (cf.~\cite{Dix,Day}).

\item[(l)] Let $T$ be a topological semigroup, not necessarily discrete; then $T$ is called a \textit{C-semigroup} if $T\setminus sT$ and $T\setminus Ts$ are relatively compact in $T$ for all $s\in T$ (cf.~\cite{KM}).
\end{description}

Then it is easy to verify that
\begin{quote}
If $(T,X)$ is an invertible semiflow with $T$ a right \textit{C}-semigroup, then $P(X)$ and $Q(X)$ both are invariant reflexive symmetric relations in $X$.
\end{quote}

In particular each abelian semigroup is amenable by the classical Markov-Kakutani fixed-point theorem.
If $T$ is a topological group, then $sT=Ts=T$ for all $s\in T$ so it is a \textit{C}-semigroup. Clearly, $T=(\mathbb{Z}_+,+)$ is a \textit{C}-semigroup. In addition, under the usual non-discrete topology, $T=(\mathbb{R}_+,+)$ is a \textit{C}-semigroup, but not under the discrete topology.

\subsubsection{Ellis enveloping semigroups}\label{sec1.1.5}
Let $X^X$ be the set of all functions from $X$ to itself, continuous or not. In contrast to the space-index uniformity on $X^X$, the topology of pointwise convergence for $X^X$ is defined as follows:
A net $\{f_n\}$ in $X^X$ converges to $f$ if and only if $f_n(x)\to f(x)$ for each $x\in X$ (cf.~\cite{Kel}).
A subbase of this topology is the family of all subsets of the form $\{f\,|\,f(x)\in U\}$, where $x$ is a point of $X$ and $U$ is open in $X$.
Then we recall several notions based on $(T,X)$ as follows:
\begin{description}
\item[(m)] By $E(T,X)$ or simply $E(X)$, we denote the \textit{Ellis semigroup} of $(T,X)$; that is, $E(X)$ is the closure of $T$, precisely $\{\pi^t\colon X\rightarrow X\,|\,t\in T\}$, in $X^X$ in the sense of the pointwise topology (cf.,~e.g., \cite{E69, Fur, Aus}).

\item[(n)] An element $u\in E(X)$ is called an \textit{idempotent} in $E(X)$, denoted $u\in J(E(X))$, if $u^2=u$.

\item[(o)] $I\not=\emptyset$ is called a \textit{minimal left ideal} in $E(X)$ if $E(X)I\subseteq I$ and no proper non-empty subset of $I$ has this property.
\begin{itemize}
\item Since $E(X)$ is a compact \textit{right-topological} semigroup (i.e., $E(X)$ is a semigroup and a compact $T_2$-space with $R_q\colon p\mapsto pq$ continuous, for all $q\in E(X)$), there always exists an idempotent in each minimal left ideal in $E(X)$ (cf.~\cite{E69, Aus}).
\item Moreover, $(x,x^\prime)\in P(X)$ iff $\exists\,p\in E(X)$ with $p(x)=p(x^\prime)$ iff there is a minimal left ideal $I$ in $E(X)$ such that $p(x)=p(y)\ \forall p\in I$.
\end{itemize}
\end{description}

Clearly $E(X)$ associated to $(T,X)$ is independent of the topology of the phase semigroup $T$. When $(T,X)$ is a flow we will consider whether or not $T$ is a topological group under the pointwise topology in $\S\ref{sec9}$.

The proof of the following basic lemma is taken nearly word-for-word from \cite[2. of Proposition~5.16]{E69}. We will postpone the details in $\S\ref{sec6.1}$ following a preliminary result---Lemma~\ref{lem6.3}.

\begin{lem}[{cf.~\cite{E69,Aus} for $T$ in groups}]\label{lem1.11}
Given any semiflow $(T,X)$, $P(X)$ is an equivalence relation on $X$ iff there is only one minimal left ideal in $E(X)$.
\end{lem}

When $(T,X)$ is equicontinuous, $E(X)$ consists of continuous maps of $X$ into itself, that is, $E(X)\subset C(X,X)$. The converse is obviously false. Then
\begin{description}
\item[(p)] $(T,X)$ is called \textit{weakly equicontinuous} if $E(X)\subset C(X,X)$.
\end{description}

This notion will be characterized by means of weakly almost periodic (w.a.p.) functions on $X$ with respect to (w.r.t.) $(T,X)$ viz Proposition~\ref{u8.8} in $\S\ref{uap}$. Particularly by the later Lemma~\ref{lem6.7} and Corollary~\ref{cor6.11}, it follows easily that
\begin{quote}
\textbf{Theorem} (cf.~Ellis~\cite[Theorem~3]{E59})\textbf{.}\,~{\it A distal semiflow is equicontinuous iff it is weakly equicontinuous.}
\end{quote}
In fact, Ellis's theorem will be generalized from distal semiflows to point-distal ones; see (3) of Corollary~\ref{cor6.8} in $\S\ref{sec6.1}$.

\subsection{Main theorems}
Although its proof is very easy (cf., e.g.,~\cite{E69,Aus,F63}), yet it is a very useful important fact in topological dynamics that
\begin{quote}
{\textbf{Theorem} (Gottschalk-Ellis)\textbf{.}~\it If $(T,X)$ is an equicontinuous flow, then it is distal} (cf.~\cite[Proposition~4.4 and Corollary~5.4]{E69}).
\end{quote}

In fact, an equicontinuous flow is uniformly distal by $\S\ref{sec1.1.3}\,(\star)$. We note that the group structure of $T$ plays a role in its various proofs available in the literature (cf.~\cite{E69,Aus,F63}). Moreover, if $T$ is only a semigroup, the above important result need not be true.
For instance, Example~\ref{ex1.7} is equicontinuous but not distal with $P(X)=X\times X\not=\varDelta_X$.

Let us see a more simple counterexample with abelian phase semigroup, which together with Examples~\ref{ex1.4} shows that some dynamics in semiflows are very different with that in the flows.

\begin{exa}\label{ex1.12}
Let $X=\{a,b,c\}$ be a discrete space and we let
$f\colon X\rightarrow X$ by $a\mapsto b\mapsto c\mapsto c$.
Then the cascade $(f,X)$ with phase semigroup $\mathbb{Z}_+$ is equicontinuous but it is not distal.
\end{exa}

Here $T$ is abelian and $(T,X)$ is not minimal. Of course, if $(T,X)$ is minimal equicontinuous with $T$ an abelian (or more general, amenable) semigroup, then we shall show it is uniformly distal (cf.~Corollary~\ref{cor3.3B}).

In both of Examples~\ref{ex1.7} and \ref{ex1.12}, each $t\in T$ is not onto. In view of this, recently Ethan Akin and Xiangdong Ye have independently suggested (in personal communications) the following:

\begin{quote}
\textbf{Akin-Ye Assertion.}\ ~\textit{If $(T,X)$ is an equicontinuous surjective semiflow on a compact metric space with $T$ an abelian semigroup, then $(T,X)$ is distal.}
\end{quote}

In fact, by constructing an equivalent isometric metric $d_T$ on $X$, Akin's \cite[(d) of Proposition~2.4]{Aki} implies that $(T,X)$ is distal if $X$ is a compact metric space with phase semigroup $T=\mathbb{Z}_+$.
In $\S\ref{sec2}$, using several different approaches, we shall be able to prove the Akin--Ye assertion without the metric condition on the phase space $X$ and with no the abelian hypothesis on the phase semigroup $T$.

Precisely speaking we shall prove the following theorem, which consists of Theorem~\ref{thm2.1}, Lemma~\ref{lem6.7}, (1) of Proposition~\ref{prop6.10}, and Theorem~\ref{thm6.12}.

\begin{thm}\label{thm1.13}
Let $(T,X)$ be a semiflow; then the following statements are satisfied:
\begin{enumerate}
\item[$(1)$] If $(T,X)$ is equicontinuous surjective, then it is distal.
\item[$(2)$] If $(T,X)$ is distal, then it is invertible.
\item[$(3)$] If $(T,X)$ is distal, then so is $(\langle T\rangle, X)$.
\item[$(4)$] If $(T,X)$ is invertible equicontinuous, then $(\langle T\rangle, X)$ is an equicontinuous flow.
\end{enumerate}
\begin{note}
By $\S\ref{sec1.1.3}$\,(\textbf{g}), $(T,X)$ being distal means that for each $t\in T$, $x\mapsto tx$ is 1-1. However, (2) of Theorem~\ref{thm1.13} tell us more. So (2) is of interest in semiflows.

Since $\langle T\rangle$ is in general much more bigger than $T$, (3) and (4) of Theorem~\ref{thm1.13} are non-trivial. They are useful in our later applications.
\end{note}
\end{thm}

By (1) and (2) of Theorem~\ref{thm1.13}, an equicontinuous semiflow is surjective iff it is distal (cf.~Corollary~\ref{cor6.8}). However, `surjective' is naturally satisfied in many interesting cases such as `homogeneous' semiflow (cf.~Proposition~\ref{prop3.2B}), minimal semiflow with `amenable' phase semigroup (cf.~Proposition~\ref{prop3.7B}), and $\ell$-recurrent semiflow with phase semigroup $T$ a \textit{C}-semigroup (cf.~Corollary~\ref{cor3.17B}).

Recall that `equicontinuity' asserts that if two initial points $x,y$ are sufficiently close, then their orbits $Tx$ and $Ty$ are synchronously close. And `uniform distality' asserts that if two initial points $x,y$ are sufficiently far away, then their orbits $Tx$ and $Ty$ are synchronously far away.

Thus, ``equicontinuity'' looks very different with ``uniform distality''. However, as a result of Theorem~\ref{thm1.13} we can obtain that they are in fact equivalent to each other.

\begin{thm}\label{thm1.14}
A semiflow $(T,X)$ is uniformly distal if and only if it is equicontinuous surjective.
\end{thm}

\begin{proof}
(1) The ``if'' part: By Theorem~\ref{thm1.13}, $(T,X)$ is invertible and further $(\langle T\rangle, X)$ is an equicontinuous flow. Then by $\S\ref{sec1.1.3}\,(\star)$, it follows that $(T,X)$ is uniformly distal.

(2) The ``only if'' part: First by Theorem~\ref{thm1.13}, $(T,X)$ is invertible. Then by Definition~\ref{def1.9}, it follows that given $\varepsilon\in\mathscr{U}_X$, there is a $\delta\in\mathscr{U}_X$ such that if $(x,y)\in\delta$ then $(t^{-1}x,t^{-1}y)\in\varepsilon$ for all $t\in T$ and $x,y\in X$. Thus $T^{-1}$ acts equicontinuously on $X$. Whence by (4) of Theorem~\ref{thm1.13}, $(\langle T^{-1}\rangle,X)$ is equicontinuous so that $(T,X)$ is invertible equicontinuous. This proves Theorem~\ref{thm1.14}.
\end{proof}

Let $(T,X)$ be a semiflow and $M$ a closed invariant subset of $X$. If $(T,X)$ is equicontinuous, then so is $(T,M)$. However, if $(T,X)$ is merely surjective (even though invertible), $(T,M)$ need not be surjective. For instance, 1. of Examples~\ref{ex1.4}. We can here construct a more simple example as follows. Let $f\colon [0,1]\rightarrow[0,1], x\mapsto x^2$ and $M=[0,1/2]$; then $f(M)\subsetneqq M$.

As a consequence of Theorem~\ref{thm1.14}, we can easily obtain the following, which is evident if $T$ is a group because $tM\subseteq M\ \forall t\in T$ with $T=T^{-1}$ implies that $M\subseteq t^{-1}M\subseteq M\ \forall t\in T$ so that $tM=M\ \forall t\in T$. However, for a semigroup $T$, it becomes non-trivial.

\begin{cor}\label{cor1.15}
Let $M$ be a closed invariant subset of a semiflow $(T,X)$. If $(T,X)$ is equicontinuous surjective, then $(T,M)$ is also equicontinuous surjective.
\end{cor}

\begin{proof}
By Theorem~\ref{thm1.14}, $(T,X)$ is uniformly distal so that $(T,M)$ is also uniformly distal. Then $(T,M)$ is equicontinuous surjective by Theorem~\ref{thm1.14} again.
\end{proof}

If $(T,X)$ is a flow with $Q(X)=\varDelta_X$, then by Lemma~\ref{lem1.10} it is equicontinuous.
By using (1) of Theorem~\ref{thm1.13}, `pointwise almost automorphy' and `Veech's relation' $V(X)$, Dai and Xiao in \cite{DX} have proved the following fact (Corollary~\ref{cor1.16}). However, we now can simply prove it by only using (1) of Lemma~\ref{lem1.10} and Theorem~\ref{thm1.13} as follows.

\begin{cor}[{cf.~\cite[Theorem~5.4]{DX}}]\label{cor1.16}
A semiflow $(T,X)$ is equicontinuous surjective if and only if $Q(X)=\varDelta_X$.
\end{cor}

\begin{proof}
Let $(T,X)$ be equicontinuous surjective. Then by equicontinuity, $P(X)=Q(X)$. Now by (1) of Theorem~\ref{thm1.13}, we see $Q(X)=\varDelta_X$. Conversely, assume $Q(X)=\varDelta_X$ and then $P(X)=\varDelta_X$. Thus $(T,X)$ is distal and so invertible by (2) of Theorem~\ref{thm1.13}. Then by (1) of Lemma~\ref{lem1.10} and (4) of Theorem~\ref{thm1.13}, $(T,X)$ is equicontinuous surjective. This proves Corollary~\ref{cor1.16}.
\end{proof}

As be mentioned before, if using no (2) and (4) of Theorem~\ref{thm1.13}, then we need a long zigzag proof for this result as in \cite{DX}. As more applications of Theorem~\ref{thm1.13}, we will present other equivalent conditions for ``equicontinuous surjective'' later on.

In 1963 Furstenberg proved that ``if $X$ is simply connected non-trivial, then $X$ does not admit of a minimal distal flow for any locally compact abelian group $T$'' (cf.~\cite[Theorem~11.1]{F63}). It is natural to ask if there admits of a minimal distal semiflow on $X$ or not. In fact, by (3) of Theorem~\ref{thm1.13} there exists no minimal distal semiflow too.

\begin{thm}\label{thm1.17}
If $X$ is simply connected non-trivial, then $X$ does not admit of a minimal distal semiflow for any locally compact abelian semigroup $T$.
\end{thm}

\begin{proof}
Assume $(T,X)$ is a minimal distal semiflow with $T$ a locally compact abelian semigroup. Then under the discrete topology of $\langle T\rangle$, $(\langle T\rangle,X)$ is a minimal distal flow by (3) of Theorem~\ref{thm1.13}. This then contradicts Furstenberg's theorem.
\end{proof}

In particular, the $n$-sphere, $n\ge2$, cannot support a minimal distal semiflow of any abelian semigroup $T$.

\subsection{Applications}
There has already been some applications in the recent work \cite{DX}\footnote{\cite{DX} is based on the first version of the present paper; see arXiv: 1708.00996v1 [math.DS].} and in the proofs of Theorems~\ref{thm1.14} and \ref{thm1.17} and Corollary~\ref{cor1.16}. Next we shall give some other applications of Theorem~\ref{thm1.13} here.

First, let $\sigma\colon\varSigma_k^+\rightarrow\varSigma_k^+$ be the shift map of the symbolic space $\varSigma_k^+=\{0,1,\dotsc,k-1\}^{\mathbb{Z}_+}$, where $k\ge2$. Suppose $(\sigma,X)$ is a subsystem of $(\sigma,\varSigma_k^+)$ such that $X$ is any infinite, closed, $\sigma$-invariant subset of $\varSigma_k^+$. Then by Theorem~\ref{thm1.13}, it follows that:
\begin{description}
\item[(\P)]{\it If $(\sigma,X)$ is equicontinuous, then it is not surjective and so not distal. If $(\sigma,X)$ is minimal, then $(\sigma,X)$ is not equicontinuous.}
\end{description}

\begin{proof}
Assume $(\sigma,X)$ is equicontinuous. If $(\sigma,X)$ is surjective, then by Theorem~\ref{thm1.13} it is distal. However, there necessarily exists a pair of proximal points for $(\sigma,X)$ (cf.~\cite[p.~158]{Fur}). Finally suppose $(\sigma,X)$ is minimal. Then $(\sigma,X)$ is surjective. Thus $(\sigma,X)$ is not equicontinuous.
\end{proof}

We now introduce a notion for our convenience.

\begin{sn}\label{sn1.18}
If $(T,X)$ is an invertible semiflow, it defines $\pi^{-1}\colon X\times T\rightarrow X$ with phase map $(x,t)\mapsto xt=t^{-1}x$, denoted $(X,T)$.
Here $(X,T)$ will be called the \textit{reflection} of $(T,X)$, which is also thought of as the `history' of $(T,X)$.
\end{sn}

It should be noted that the phase semigroup of the reflection $(X,T)$ is $T$, not $T^{-1}$, with the \textit{discrete} topology in general.

\begin{remarks}\label{rem1.19}
\begin{enumerate}
\item Let $G$ be a non-discrete topological group. If $(G,X)$ is a flow and if $T$ is a subsemigroup of $G$, then $(X,T)$ is of course a semiflow where $T$ with the non-discrete topology inherited from $G$.

\item If $(T,X)$ is invertible with $T$ a locally compact \textit{C}-semigroup and if $t_n\to t$ in $T$ implies that
$t_n^{-1}x\to t^{-1}x$ for all $x\in X$, then $(x,t)\mapsto xt$ is jointly continuous by Corollary~\ref{cor9.17} in $\S\ref{sec9}$ so $(X,T)$ is a semiflow where $T$ with the non-discrete topology.
\end{enumerate}
\end{remarks}

If $(T,X)$ is a minimal cascade corresponding to a $\mathbb{Z}_+$-action, then for every minimal set $X_0$ of its reflection $(X,T)$ we have $X_0T=X_0$ and furthermore $X_0=TX_0$; so $X_0=X$. This indicates that $(X,T)$ is also minimal. However, for an invertible semiflow $(T,X)$ with $T\not=\mathbb{Z}_+$, `$X_0T=X_0$' need not imply $X_0=TX_0$.
Moreover, $(T,X)$ and $(X,T)$ do not share the same dynamics in general. For example, a recurrent/transitive point of a cascade $(T,X)$ need not be a recurrent/transitive point of its reflection $(X,T)$.
In addition, in 1. of Examples~\ref{ex1.4}, every point $x$ of $\Lambda$ is minimal for $(T,X)$ but not minimal for its reflection $(X,T)$; for otherwise, $x$ is minimal for $(\langle T\rangle,X)$.

Nevertheless, as applications of Theorem~\ref{thm1.13}, we will consider in $\S\ref{sec6}$ the minimality, distality, and equicontinuity dynamics of $(X,T)$ as $(T,X)$ itself possesses these dynamics. We will mainly show the following three reflection principles.

\begin{prI}[{cf.~Propositions~\ref{prop6.1} and \ref{prop6.10}}]
Let $\pi\colon T\times X\rightarrow X$ be an invertible semiflow with its reflection $\pi^{-1}\colon X\times T\rightarrow X$. Then:
\begin{enumerate}
\item $(T,X)$ is equicontinuous iff so is $(X,T)$.
\item $(T,X)$ is minimal distal iff so is $(X,T)$.
\item $(T,X)$ is distal iff so is $(X,T)$.
\end{enumerate}
\end{prI}

Reflection principle I may be utilized for proving the ``only if'' part of Theorem~\ref{thm1.14} above. Moreover, it will be useful for us to show Furstenberg's structure theorem of distal minimal \textit{semiflows} (Theorem~\ref{thm3.14A}) in $\S\ref{sec3.2A}$.

Here we are going to give another application. Let $(T,X)$ and $(T,Y)$ be two semiflows and $\varphi\colon X\rightarrow Y$ a continuous surjective map. If $\varphi(tx)=t\varphi(x)$ for all $t\in T$ and $x\in X$, then $(T,Y)$ is called a \textit{factor} of $(T,X)$ and $\varphi$ an \textit{epimorphism} from $(T,X)$ to $(T,Y)$, denoted $\varphi\colon(T,X)\rightarrow(T,Y)$.

\begin{sn}\label{sn1.20}
We will say that $(T,X)$ is a \textit{relatively equicontinuous} (or an \textit{almost periodic}) extension of $(T,Y)$ in case there is an epimorphism $\varphi\colon(T,X)\rightarrow(T,Y)$ such that for every $\varepsilon\in\mathscr{U}_X$ there exists $\delta\in\mathscr{U}_X$ satisfying that if $(x,x^\prime)\in\delta$ with $\varphi(x)=\varphi(x^\prime)$ then $(tx,tx^\prime)\in\varepsilon\ \forall t\in T$ (cf.~\cite[p.~95]{Aus}). If $Y$ is one-pointed, then this reduces to $\S\ref{sec1.1.1}\,(\textbf{a})$.
\end{sn}

Theorem~\ref{thm1.21} below is actually a relativized version of (1) of Theorem~\ref{thm1.13} before. Its special case that $(T,X)$ is a skew-product semiflow driven by $(T,Y)$ with $T=\mathbb{R}_+$ has been proven in \cite{SS74, SS} and \cite{SSY} by using Ellis' enveloping semigroup.

\begin{thm}[{cf.~\cite[Proposition~2.1]{F63} for $T$ in groups}]\label{thm1.21}
If an invertible semiflow $(T,X)$ is a relatively equicontinuous extension of a distal semiflow $(T,Y)$, then $(T,X)$ is distal.
\end{thm}

\begin{proof}
Let $\varphi\colon(T,X)\rightarrow(T,Y)$ be a relatively equicontinuous epimorphism.
First by Theorem~\ref{thm1.13}, it follows that $(T,Y)$ is invertible and then $(Y,T)$ is distal by Reflection principle I.
Note that $\varphi\colon X\rightarrow Y$ is also a homomorphism from $(X,T)$ onto $(Y,T)$. We will show that $(X,T)$ is distal.

For that, let $x,x^\prime\in X$ with $(x,x^\prime)\in P(X,T)$. Then by distality of $(Y,T)$, $\varphi(x)=\varphi(x^\prime)$. Taking a net $\{t_n\}$ in $T$ with $t_n^{-1}(x,x^\prime)\to(z,z)$ for some $z\in X$, by the relative equicontinuity of $(T,X)$ and $\varphi(t_n^{-1}x)=\varphi(t_n^{-1}x^\prime)$, it follows that $x=x^\prime$
and so $(X,T)$ is distal. Again using Reflection principle~I, $(T,X)$ is distal. This proves Theorem~\ref{thm1.21}.
\end{proof}

The above proof implies that if $\varphi\colon(T,X)\rightarrow(T,Y)$ is a relatively equicontinuous epimorphism of flows, $\varphi$ is of distal type (cf., e.g.,~\cite[p.~95]{Aus}). It would be interested to know if this holds for any invertible semiflow or not.

Next for invertible semiflows of amenable semigroups we can obtain the following Reflection principle II.

\begin{prII}[cf.~Theorem~\ref{thm5.4}, Theorem~\ref{thm6.19} and Proposition~\ref{prop6.20}]
Let $(T,X)$ be invertible with $T$ an amenable semigroup and $x\in X$. Then:
\begin{enumerate}
\item $x$ is a minimal point for $(T,X)$ iff so is it for $(X,T)$. Moreover, if $x$ is a minimal point of $(T,X)$, then $\mathrm{cls}_XTx=\mathrm{cls}_XxT$. Hence if $(\langle T\rangle,X)$ is minimal so is $(T,X)$.
\item $x$ is a distal point of $(T,X)$ iff $x$ is a distal point of $(X,T)$.
\end{enumerate}
\end{prII}

It should be mentioned that in light of Examples~\ref{ex1.4} and Weiss's example the amenability of the phase semigroup $T$ is very important for the statement of Reflection principle II.

\begin{sn}
$(T,X)$ is said to be \textit{individually distal} if for each $t\in T$ the cascade $(t,X)$ is a distal system.

It is clear that an individually distal semiflow is invertible and a distal one is individually distal.
\end{sn}

\begin{cor}\label{cor1.23}
Let $(T,X)$ be an individually distal semiflow. If $(\langle T\rangle,X)$ is minimal, then $(T,X)$ is minimal.
\end{cor}

\begin{proof}
Let $x_0\in X$ be arbitrary and $X_0=\overline{Tx_0}$. Then $X_0$ is an invariant closed set of $(T,X)$. We shall show $\langle T\rangle x_0\subset X_0$ consequently $X_0=X$. For this, let $t\in\langle T\rangle$ be such that $t=s_n^{-1}t_n\dotsm s_1^{-1}t_1$ where $s_1,\dotsc,s_n, t_1,\dotsc,t_n\in T$. Since $t_1x_0\in X_0$ and $t_1x_0$ is a minimal point of $(s_1,X)$, by 1. of Reflection principle~II it follows that $s_1^{-1}t_1x_0\in\overline{\{s_1^n(t_1x_0)\,|\,n=0,1,2,\dotsc\}}\subseteq X_0$. Repeating this with $s_1^{-1}t_1x_0$ in place of $x_0$ we see $s_2^{-1}t_2s_1^{-1}t_1x_0\in X_0$. By induction, $tx_0\in X_0$. Since $t\in\langle T\rangle$ be arbitrary, $\langle T\rangle x_0\subset X_0$ and so $(T,X)$ is minimal.
\end{proof}

As a result of our Reflection principle II, we can generalize Theorem~\ref{thm1.21} in amenable semigroups as follows:

\begin{thm}\label{thm1.24}
Let $(T,X)$ and $(T,Y)$ be two minimal invertible semiflows with $T$ an amenable semigroup. If $\varphi\colon (T,X)\rightarrow(T,Y)$ is a relatively equicontinuous extension and $(T,Y)$ is point-distal, then $(T,X)$ is a point-distal semiflow. In fact, if $y\in Y$ is a distal point, then each point of $\varphi^{-1}(y)$ is distal for $(T,X)$.
\end{thm}

\begin{proof}
Let $\varphi\colon X\rightarrow Y$ be a relatively equicontinuous extension.
Let $y\in Y$, with $\varphi(x)=y$ for some $x\in X$, be a distal point of $(T,Y)$, then $y$ is also a distal point of $(Y,T)$ by Reflection principle II so that if $(x,x^\prime)\in P(X,T)$ then $\varphi(x^\prime)=y$, i.e., $x^\prime\in\varphi^{-1}(y)$. Hence as in the proof of Theorem~\ref{thm1.21}, we can easily show that $x^\prime=x$. Thus $x$ is a distal point of $(X,T)$ and then of $(X,T)$.
\end{proof}

In addition, using Reflection principle II in the classical case that $T=\mathbb{R}_+$ and $\langle T\rangle=\mathbb{R}$ we can easily obtain the following:
\begin{quote}
If $\mathbb{R}\times X\rightarrow X,\ (t,x)\mapsto tx$ is a \textit{flow} such that $\mathbb{R}_+\times X\rightarrow X,\ (t,x)\mapsto tx$ has a distal point $x\in X$, then $x$ is a distal point of $\mathbb{R}\times X\rightarrow X$.

In particular, if $\mathbb{R}_+\times X\rightarrow X$ is a distal semiflow, then $\mathbb{R}\times X\rightarrow X$ is a distal flow (cf.~Sacker-Sell~\cite{SS}).
\end{quote}

The following is another consequence of Reflection principle II, where the point is that $T$ need not be right-syndetic in $\langle T\rangle$.

\begin{thm}[{cf.~\cite{SY} for $T=\mathbb{R}_+$}]\label{thm1.25}
Let $(T,X)$ be invertible with $T$ an amenable semigroup such that $P(\langle T\rangle,X)$ is an equivalence relation on $X$. Then
$P(\langle T\rangle,X)=P(T,X)=P(X,T)$.
\end{thm}

\begin{proof}
We only prove $P(\langle T\rangle,X)=P(T,X)$. Clearly, $P(T,X)\subseteq P(\langle T\rangle,X)$. To show the converse inclusion, let $(x,x^\prime)\in P(\langle T\rangle,X)$. Let $\mathbb{I}$ be the unique minimal left ideal in $E(\langle T\rangle,X)$ by Lemma~\ref{lem1.11}. Then there exists some $u\in\mathbb{I}$ with $u(x)=u(x^\prime)$. We will show that $u\in E(T,X)$.

For that, we need to consider the natural flow $\pi_*\colon\langle T\rangle\times E(\langle T\rangle,X)\rightarrow E(\langle T\rangle,X)$, which has only one minimal subset $\mathbb{I}$. Clearly $E(T,X)\subseteq E(\langle T\rangle,X)$ and $\pi_*\colon T\times E(T,X)\rightarrow E(T,X)$ is compatible with $(\langle T\rangle, E(\langle T\rangle,X))$. We can choose a minimal subset $K\subseteq E(T,X)$ with respect to $(T,E(T,X))$. Then by 1 of Reflection principle II, $K$ is also $T^{-1}$-invariant so that $K$ is $\langle T\rangle$-invariant and then $\langle T\rangle$-minimal. Thus $K=\mathbb{I}$. This implies $u\in E(T,X)$ and thus $(x,x^\prime)\in P(T,X)$.

The proof of Theorem~\ref{thm1.25} is therefore complete.
\end{proof}

Notice that in general, $T\cup T^{-1}\varsubsetneq\langle T\rangle$ for an invertible semiflow $(T,X)$. 2. of Reflection principle~II says that an $x\in X$ is a distal point of $(T,X)$ iff $x$ is a distal point of $(X,T)$. However,
\textit{is $x$ a distal point of the induced flow $(\langle T\rangle,X)$?}
It is exactly the localization problem of (3) of Theorem~\ref{thm1.13}. As a consequence of Theorem~\ref{thm1.25}, the answer to this question is YES in the setting of Theorem~\ref{thm1.25}.

\begin{quote}
{\it In the same situation of Theorem~\ref{thm1.25}, an $x\in X$ is a distal point of $(T,X)$ iff $x$ is a distal point of $(\langle T\rangle,X)$.}
\end{quote}

In fact, this still holds without the condition that $P(\langle T\rangle,X)$ is an equivalence relation on $X$; see 1. of Theorem~\ref{thm5.4} by purely topological methods.

We note that if $T=\mathbb{R}_+$ in the proof of Theorem~\ref{thm1.25} and so $\langle T\rangle=\mathbb{R}$, then $\mathbb{I}$ is contained in the $\omega$-limit set $\omega(e)$ of $e\in T$ under $(\langle T\rangle,E(\langle T\rangle,X))$. Since $e=\textit{id}_X\in E(T,X)$ and $\omega(e)\subseteq E(T,X)$ by $\langle T\rangle=\mathbb{R}$, then $\mathbb{I}\subseteq E(T,X)$. This is actually the idea of Yi's proof in the $\mathbb{R}$-action case in \cite[p.~7]{SY}. Clearly Yi's idea does not work for our Theorem~\ref{thm1.25} here.

In addition, when $\langle T\rangle$ is abelian and if $E(\langle T\rangle,X)\subset C(X,X)$, then by Theorem~\ref{thm6.6} it follows that $P(\langle T\rangle,X)$ is an equivalence relation on $X$. Moreover, if $P(\langle T\rangle,X)$ is closed in $X\times X$, then it is an equivalence relation (cf.~\cite[Corollary~6.11]{Aus}).

Next we consider invertible semiflows with \textit{C}-semigroups as our phase semigroups instead of amenable semigroups.

\begin{prIII}[cf.~Theorem~\ref{thm6.31}]
Let $(T,X)$ be invertible with $T$ a \textit{C}-semigroup not necessarily discrete. Then $(T,X)$ is minimal iff so is $(X,T)$.
\end{prIII}

Comparing with Reflection principle~II, we pose the following

\begin{que}
Let $(T,X)$ be invertible with $T$ a \textit{C}-semigroup and $x\in X$. Then, does it hold that $x$ is minimal for $(T,X)$ iff so is $x$ for $(X,T)$?
\end{que}

As other applications of our reflection principles, in $\S\ref{sec7}$ we will study the sensitivity of semiflows and their reflections and consider the $\mathbb{Z}$-actions on zero-dimensional phase spaces; see Theorem~\ref{thm7.12}.

For more applications of Theorem~\ref{thm1.13} we will provide some equivalence conditions for a flow to be almost periodic in $\S\ref{uap}$. See, for example, Theorem~\ref{u8.3} and Theorem~\ref{u8.10}.
\subsection{Connected phase semigroup}\label{sec1.4A}
Let $(T,X)$ be an invertible semiflow with $X$ compact metrizable and with $T$ a connected semigroup.
We now will present a necessary condition for $M\subseteq X$ being minimal.

Recall that a \textit{Cantor-manifold} is defined to be a compact metrizable space $M$ of positive finite dimension $n$ such that $M$ is not disconnected by a subset of dimension $\le n-2$. We will need a classical result.

\begin{lem}[{Hurewicz-Wallman; cf.~\cite[Lemma~2.17]{GH}}]\label{lem1.27}
Let $M$ be a compact metrizable space of positive finite dimension. Then there exists a subset $C$ of $X$ such that $C$ is a Cantor-manifold with $\dim C=\dim M$.
\end{lem}

\begin{prop}\label{prop1.28}
Let $(T,X)$ be an invertible semiflow such that $T$ is a connected semigroup and $X$ a compact metrizable space with $1\le\dim X<\infty$. If $M$ is a $T$-minimal subset of $X$, then either $\dim M=0$ or $M$ is a Cantor-manifold.
\end{prop}

\begin{proof}
Let $n=\dim M\ge1$ and assume $M$ is not a Cantor-manifold. Then there exist closed proper subsets $A,B$ of $M$ such that $M=A\cup B$ and $\dim A\cap B\le n-2$. By Lemma~\ref{lem1.27} there exists a Cantor-manifold $C\subset M$ with $\dim C=n$. Set $T_A=\{t\in T\,|\,tC\subseteq A\}$ and $T_B=\{t\in T\,|\,tC\subseteq B\}$. Clearly, $T_A$ and $T_B$ are disjoint closed subsets of $T$ for $tC$ is a Cantor-manifold. Now for $t\in T$, since $tC$ is an $n$-dimensional Cantor-manifold and $tC=(tC\cap A)\cup(tC\cap B)$, hence either $tC\subseteq A$ or $tC\subseteq B$ and so either $t\in T_A$ or $t\in T_B$. Thus $T=T_A\cup T_B$. Since $T$ is connected, then either $T=T_A$ or $T=T_B$ and so either $\textrm{cls}_X{TC}\subseteq A$ or $\textrm{cls}_X{TC}\subseteq B$. Thus $\textrm{cls}_X{TC}\not=M$ and this contradicts the minimality of $(T,M)$. The proof is completed.
\end{proof}

The foregoing proposition is a generalization of \cite[Theorem~2.18]{GH}. The connected of $T$ plays a role here. Since $T$ is connected, $\dim M=0$ implies $M=Tx=\{x\}$ for some $x\in X$ in Proposition~\ref{prop1.28}.

\section{Distality of equicontinuous surjective semiflows}\label{sec2}
Recall that a semiflow $(T,X)$ with phase map $(t,x)\mapsto tx$ is surjective if and only if each $t\in T$ is an onto map of $X$ (cf.~\ref{sn1.3}); it is equicontinuous if and only if given $\mathscr{U}_X$ there exists a $\delta\in\mathscr{U}_X$ such that $t\delta\subseteq\varepsilon$ for all $t\in T$ (cf.~$\S\ref{sec1.1.1}$\,(\textbf{a})); and it is distal if and only if no diagonal pair is proximal (cf.~$\S\ref{sec1.1.3}$\,(\textbf{g})).

This section will be mainly devoted to proving (1) of Theorem~\ref{thm1.13} using three different approaches, which asserts that every equicontinuous surjective semiflow is distal. That is the following Theorem~\ref{thm2.1}.

\begin{thm}\label{thm2.1}
If $(T,X)$ is an equicontinuous surjective semiflow, then $(T,X)$ is distal.
\end{thm}

Therefore by this theorem, it follows that a surjective semiflow is distal if it satisfies one of conditions (2) $\thicksim$ (6) of Lemma~\ref{lem1.6}.

Our new approaches (Proofs (I), (II) and (III) below) introduced in proving Theorem~\ref{thm2.1} are all certainly valid for flows with phase groups.

\subsection{Using pointwise recurrence of transition maps}
For Proof (I) of Theorem~\ref{thm2.1}, in preparation we first recall a classical notion.
Let $f$ be a continuous self-map of $X$. A point $x\in X$ is said to be (forwardly) \textit{recurrent} if there is a net $\{n_\epsilon\}$ in $\mathbb{Z}_+$ with $n_\epsilon\to+\infty$ such that $f^{n_\epsilon}(x)\to x$. Further $(f,X)$ is called \textit{pointwise recurrent} if each point of $X$ is recurrent for $(f,X)$.
Then the following is easily seen by definition.

\begin{lem}\label{lem2.2}
If $x\in X$ is a recurrent point of $(f,X)$, then $x\in f(X)$.
\end{lem}

\begin{proof}
Let $x\in X$ be recurrent for $(f,X)$. Then there is a net $\{n_i\}$ in $\mathbb{Z}_+$ with $n_i\to+\infty$ and $f^{n_i}(x)\to x$ so that $f(f^{n_i-1}(x))\to x$. Since $X$ is compact $T_2$, there is a subnet $\{j_i\}$ of $\{n_i-1\}$ such that $f^{j_i}(x)\to y\in X$ and thus $f(y)=x$. Thus $x\in f(X)$.
\end{proof}

Thus if $(f,X)$ is pointwise recurrent, then $f$ is surjective (cf.~\cite[Lemma~3.1]{AAB}). The following simple observation is very useful for Theorem~\ref{thm2.1}.

\begin{lem}\label{lem2.3}
Suppose that $f\colon X\rightarrow X$ is equicontinuous surjective. Then every point of $X$ is recurrent for $(f,X)$.
\end{lem}

\begin{proof}
Let $x\in X$ and define $\{x_n\}$ inductively by
$f(x_1)=x$, $f(x_2)=x_1$, $\dotsc$, $f(x_n)=x_{n-1}$, $\dotsc$.
Let $\varepsilon\in\mathscr{U}_X$ and let $\delta$ correspond to $\varepsilon$ in the definition of equicontinuity. Let $n>0$ and $s>0$ be integers such that $(x_n,x_{n+s})\in\delta$, so $(f^{n+s}(x_n),f^{n+s}(x_{n+s}))\in\varepsilon$. Then $(x,f^s(x))\in\varepsilon$ and thus $x$ is (forwardly) recurrent for $(f,X)$.
\end{proof}

Now we can prove Theorem~\ref{thm2.1} by using the pointwise recurrence as follows:

\begin{proof}[\textbf{Proof (I) of Theorem~\ref{thm2.1}}]
For $t\in T$, $(t,X\times X)$ is equicontinuous surjective and thus by Lemma~\ref{lem2.3} it is pointwise recurrent.
Suppose $(y,y^\prime)\in P(T,X)$ with $y\not=y^\prime$. Let $\varepsilon\in\mathscr{U}_X$ such that $(y,y^\prime)\not\in\varepsilon$. Let $\delta\in\mathscr{U}_X$ correspond to $\varepsilon/3$ as in $\S\ref{sec1.1.1}$\,(\textbf{a}).
Since $y$ is proximal to $y^\prime$, we now can take $\tau\in T$ such that $(\tau y, \tau y^\prime)\in\delta$, so $(\tau^ny,\tau^ny^\prime)\in\varepsilon/3$ for all $n>0$. Then there cannot be $n_i$ with $\tau^{n_i}(y,y^\prime)\to(y,y^\prime)$. But this contradicts the pointwise recurrence.
The proof of Theorem~\ref{thm2.1} is thus completed.
\end{proof}

\subsection{Using almost periodicity}\label{sec2.2}
Let $(T,X)$ be a semiflow with $T$ a topological semigroup not necessarily discrete here. We will first recall the concept of ``almost periodicity'' due to Gottschalk.

\begin{sn}[\cite{G, Fur}]\label{sn2.4}
\begin{enumerate}
\item[(i)] A subset $A$ of $T$ is said to be \textit{right-thick} in $T$ if for all compact subset $K$ of $T$ one can find some $s\in T$ such that $Ks\subseteq A$.

\item[(ii)] A subset $A$ of $T$ is called \textit{right-syndetic} in $T$ if there is a compact subset $K$ of $T$ with $Kt\cap A\not=\emptyset$ for every $t\in T$.
We could, of course, have defined the \textit{left-syndetic set} in $T$.

\item[(iii)] A point $x\in X$ is called \textit{almost periodic} (a.p.) if
$N_T(x,U)$ is right-syndetic in $T$ for all neighborhood $U$ of $x$ in $X$; that is, there exists a compact set $K\subseteq T$ with $Ktx\cap U\not=\emptyset$ for all $t\in T$.

\item[(iv)] If every point of $X$ is a.p. for $(T,X)$, then $(T,X)$ is called \textit{pointwise almost periodic}.
\end{enumerate}
Here ``right-thick set'' is weaker than the notion---\textit{replete set} \cite[Definition~3.37]{GH} that requires containing some bilateral translate $Ks\cup sK$ of each compact subset $K$ of $T$.

Given $k\in T$, let $L_k\colon T\rightarrow T, t\mapsto kt$ be the left translation mapping of $T$. Then for subsets $K,A$ of $T$, we simply write
\begin{gather*}
K^{-1}A\index{$K^{-1}A$}={\bigcup}_{k\in K}L_k^{-1}[A],\quad \textrm{where }L_k^{-1}[A]=\{t\in T\,|\,kt\in A\}.
\end{gather*}
Since here $T$ is only a semigroup, $K^{-1}A$ is possibly empty. If $e\in K$ then $A\subseteq K^{-1}A$.
Thus a subset $A$ of $T$ being right-syndetic in $T$ can be equivalently described as follows:
\begin{quote}
{\it $A$ is right-syndetic in $T$ if and only if there exists a compact subset $K$ of $T$ with $T=K^{-1}A$}.
\end{quote}
\end{sn}

\begin{note2.4}
In fact, the a.p. point may be defined for any topological space $X$~\cite{G}. Strengthening the topology on $T$ strengthens the notion of a.p. point, the strongest type of a.p. occurring when $T$ is provided with the discrete topology.
Moreover, for $(T,X)$ in flows with any phase spaces $X$, whether or not $x$ is an a.p. point does not depend on the topology on $T$ (cf.~\cite{Dai}), while this is uncertain for $(T,X)$ in general semiflows. See Lemma~\ref{lem2.6} below for $X$ in compact spaces.
\end{note2.4}

\begin{note2.4}
In some literature, an a.p. point is defined as $A$ is ``right-syndetic'' in the sense that there is a compact subset $K$ of $T$ with $T=KA$.
However, ``$T=K^{-1}N_T(x,U)$'' in (iii) of Definition~\ref{sn2.4} is not permitted to be replaced by ``$T=KN_T(x,U)$'' in semiflows here unless $T=\mathbb{Z}_+$ or $\mathbb{R}_+$; see \cite[Proposition~4.8]{CD} for a counterexample which says that there is a semiflow on a compact metric space such that it has an a.p. point in the sense of (iii) of Definition~\ref{sn2.4} but has no ``almost periodic'' points in the latter sense.
\end{note2.4}

The following two equivalent conditions will be very useful for our later arguments involving almost periodicity.

\begin{lem}[{cf.~\cite{Fur} for $T=\mathbb{Z}_+$}]\label{lem2.5}
A subset $S$ of $T$ is right-syndetic in $T$ if and only if $S\cap R\not=\emptyset$
for each right-thick set $R$ in $T$.
\end{lem}

\begin{proof}
Let $S$ be right-syndetic in $T$ and let $K$ a compact subset of $T$ defined by right-syndeticity of $S$. Then for each right-thick set $R$ in $T$, there is some $t_0\in T$ with $Kt_0\subseteq R$. Since $(Kt_0)\cap S\not=\emptyset$, hence $R\cap S\not=\emptyset$.
Conversely, let $S\cap R\not=\emptyset$ for all right-thick set $R$ in $T$. If $S$ is not right-syndetic, then for each compact subset $K$ of $T$ there is $t_K\in T$ such that $Kt_K\cap S=\emptyset$. Set $R=\bigcup_{K\in\mathscr{K}}Kt_K$ where $\mathscr{K}$ is the set of all non-empty compact subsets of $T$. Clearly $R$ is right-thick in $T$, but $S\cap R=\emptyset$, a contradiction. This proves Lemma~\ref{lem2.5}.
\end{proof}

It should be noticed that when $S$ is a \textbf{left}-syndetic subset of $T$; then for a \textbf{right}-thick subset $R$ of $T$, $S\cap R$ need not be non-empty.

Since our phase space $X$ is a compact $T_2$-space, every orbit closure contains a minimal set by Zorn's lemma. So it contains an a.p. point by the following basic result.

\begin{lem}[{cf.~\cite{GH} and \cite[Proposition~2.5]{E69} for $T$ in groups; \cite{G,Fur,CD}}]\label{lem2.6}
Let $(T,X)$ be a semiflow where $T$ not necessarily discrete; then $x\in X$ is a.p. iff $\textrm{cls}_XTx$ is a compact minimal subset of $X$.
\end{lem}

\begin{note2.6}
Although the notion of a right-syndetic set and hence that of being an a.p. point depends upon the topology on $T$, yet this lemma shows that whether or not $x$ is an a.p. point does not depend on the topology on the phase semigroup $T$.
\end{note2.6}

\begin{note2.6}
The result of Lemma~\ref{lem2.6} is false for general semiflows with non-compact phase spaces. See (2) of Remarks~\ref{rem3.12B} for a counterexample. However, if $(T,X)$ is a semiflow with $T=\mathbb{Z}_+$ or $\mathbb{R}_+$, each provided with its natural topology, $X$ locally compact $T_2$, and $x\in X$, then $x$ is an a.p. point iff $\textrm{cls}_XTx$ is minimal iff $\textrm{cls}_XTx$ is minimal compact (cf.~\cite[Theorem~2.12]{CD}).
\end{note2.6}

\begin{note2.6}\label{n2.6.3}
If $(T,X)$ is \textit{invertible} with $T$ a \textit{discrete} semigroup and $X$ locally compact $T_2$, then $x\in X$ is an a.p. point iff $\textrm{cls}_XTx$ is a compact minimal subset of $X$.
\end{note2.6}

\begin{proof}[Proof of Lemma~\ref{lem2.6}]
Let $x$ be an a.p. point of $(T,X)$; and if $\Lambda$ is a minimal subset of $\textrm{cls}_X{Tx}$ with $x\not\in\Lambda$ there are neighborhoods $U$ of $x$ and $V$ of $\Lambda$ such that $U\cap V=\emptyset$. For every compact subset $K$ of $T$ and $y_0\in\Lambda$, there is a $\delta\in\mathscr{U}_X$ so small that $K(\delta[y_0])\subset V$. Since $t_0x\in\delta[y_0]$ for some $t_0\in T$, then $Kt_0x\subset V$ so $Kt_0\subset N_T(x,V)$. Thus $N_T(x,V)$ is right-thick in $T$. But $N_T(x,U)$ is right-syndetic in $T$, we conclude a contradiction $N_T(x,V)\cap N_T(x,U)\not=\emptyset$.

Conversely, let $\textrm{cls}_XTx$ be compact minimal and let $U$ be an open neighborhood of $x$. Since $Ty$ is dense in $\textrm{cls}_XTx$ for all $y\in\textrm{cls}_XTx$, hence $\{t^{-1}U\,|\,t\in T\}$ is an open cover of the compact subspace $\textrm{cls}_XTx$. Thus one can find a finite subset $K$ of $T$ such that $\textrm{cls}_X{Tx}\subseteq \bigcup_{k\in K}k^{-1}U$; thus for any $t\in T$, $tx\in k^{-1}U$ and so $ktx\in U$ for some $k\in K$; this implies that $N_T(x,U)$ is right-syndetic in $T$; therefore $x$ is an a.p. point of $(T,X)$.
\end{proof}

\begin{proof}[Proof of Note~\ref{n2.6.3}]
The sufficiency follows from Lemma~\ref{lem2.6}. Now let $x$ be an a.p. point. Since $X$ is regular, $\textrm{cls}_XTx$ is a minimal subset of $X$. The rest is to show $\textrm{cls}_XTx$ compact. For this, let $U$ be a compact neighborhood of $x$. Then there exists a finite subset $K$ of $T$ such that $Kt\cap N_T(x,U)\not=\emptyset$ for all $t\in T$. Thus $Tx\subset\bigcup_{k\in K}k^{-1}U$ and so $\textrm{cls}_XTx\subseteq\bigcup_{k\in K}k^{-1}U$ is compact.
\end{proof}

\begin{lem}\label{lem2.7}
If $(x,y)\in P(T,X)$ with $x\not=y$, then $(x,y)$ is not an a.p. point of $(T,X\times X)$.
\end{lem}

\begin{proof}
By the joint continuity of $tx$, if $(x,y)\in P(X)$, then for every $\varepsilon\in\mathscr{U}_X$, $\{t\in T\,|\,t(x,y)\in\varepsilon\}$ is a right-thick subset of $T$. Thus we can obtain the conclusion.
\end{proof}

The following lemma is a generalization as well as strengthening of Lemma~\ref{lem2.3}. See \cite[Lemma~2.3]{Aus} for $T$ in groups.

\begin{lem}\label{lem2.8}
If $(T,X)$ is equicontinuous surjective, then every point of $X$ is an a.p. point of $(T,X)$.
\end{lem}

\begin{proof}
Let $x\in X$ and let $M$ be a minimal subset of $\textrm{cls}_X{Tx}$. If $x\not\in M$, then there is an $\varepsilon\in\mathscr{U}_X$ with $x\not\in\varepsilon[M]$. Let $tx$ be arbitrarily close to some $y\in M$. Since $x$ is a recurrent point for $(t,X)$ by Lemma~\ref{lem2.3}, there is a net $\{n_k\}$ in $\mathbb{N}$ with $n_k\to\infty$ and $t^{n_k}x\to x$. Then by equicontinuity, it follows that $t^{n_k}x$ is arbitrarily close to $t^{n_k-1}y\in M$ and so $x$ is arbitrarily close to $M$, contradicting $x\not\in\varepsilon[M]$. Thus $x\in M$ and so every point of $X$ is a.p. by Lemma~\ref{lem2.6}.
\end{proof}

Note that in view of Example~\ref{ex1.12} the `surjective' condition is important for Lemma~\ref{lem2.8}. However, it is not a necessary condition for almost periodicity; for instance, Example~\ref{ex1.7}.

Now, based on Lemma~\ref{lem2.7} and Lemma~\ref{lem2.8}, we can present another concise proof of Theorem~\ref{thm2.1} as follows.

\begin{proof}[\textbf{Proof (II) of Theorem~\ref{thm2.1}}]
Let $(T,X)$ be equicontinuous surjective. Then $(T,X\times X)$ is equicontinuous surjective and so by Lemma~\ref{lem2.8}, $(T,X\times X)$ is pointwise a.p.. Thus $(T,X)$ is distal by Lemma~\ref{lem2.7}.
\end{proof}

Although ``distality $\Rightarrow$ almost periodicity'' may be localized (cf.~Theorem~\ref{thm5.1} in $\S\ref{sec5}$), yet it is interesting to notice that ``equiconinuous + surjective $\Rightarrow$ distal'' (Theorem~\ref{thm2.1}) and ``equiconinuous + surjective $\Rightarrow$ almost periodic'' (Lemma~\ref{lem2.8}) can not be localized.
In fact we can easily construct a counterexample on the unit interval $I=[0,1]$ with the usual topology as follows.

\begin{exa}\label{exa2.9}
Let $f\colon I\rightarrow I$ be defined by $x\mapsto x^2$. Then $0$ and $1$ are the only recurrent (fixed) points of $(f,I)$. Moreover, $(f,I)$, as a flow with phase group $\mathbb{Z}$, is equicontinuous at each $x\in(0,1)$ but $x\in(0,1)$ is neither an a.p. point and nor a distal point of $(f,I)$.
\end{exa}

\subsection{Using Ellis' enveloping semigroup}
Based on Ellis' semigroup (cf.~$\S\ref{sec1.1.5}$), the following another short proof of Theorem~\ref{thm2.1} without using the pointwise recurrence of an equicontinuous surjection is the other important idea of this paper.

\begin{proof}[\textbf{Proof (III) of Theorem~\ref{thm2.1}}]
Since $(T,X)$ is equicontinuous, then $(p,x)\mapsto p(x)$ of $E(X)\times X$ to $X$ is jointly continuous and hence the topology of pointwise convergence coincides with the compact-open topology for $E(X)$ (cf.~\cite[Theorem~7.15]{Kel}). It follows easily from equicontinuity and surjectivity of each $t\in T$ that all $p\in E(X)$ are surjective. Indeed, let $T\ni t_n\to p\in E(X)$ and $p(X)\not=X$; then there is an $\varepsilon\in\mathscr{U}_X$ so small that $U=\varepsilon[p(X)]\not=X$. Since $p(X)\subset U$ and $t_n\to p$ in the sense of compact-open topology, $t_n(X)\subseteq U$ as $n$ sufficiently large. This contradicts that $tX=X$ for all $t\in T$.
Now for every idempotent $u$ in $E(X)$, since $u(u(x))=u(x)$ for all $x\in X$ and $u(X)=X$, thus $u=\textit{id}_X$. So if $(x,y)\in P(X)$, then $\{p\,|\,p(x)=p(y)\}$ is a non-empty closed subsemigroup of $E(X)$ and so there is an idempotent $u$ in $E(X)$ with $u(x)=u(y)$ and so $x=y$. This proves Theorem~\ref{thm2.1}.
\end{proof}

\section{When is a semiflow surjective?}\label{sec3A}
In light of Examples~\ref{ex1.7} and \ref{ex1.12}, the ``surjective'' condition is essential for our assertion of Theorem~\ref{thm2.1}. In this section we will now introduce some sufficient conditions for that `each $t\in T$ is surjective' for a semiflow $(T,X)$ with some special phase semigroups $T$.

\subsection{Homogeneity condition}
Let $(T,X)$ be a semiflow. Since $X$ is compact by our convention, each $(t,X)$ must have a.p. points by Lemma~\ref{lem2.6} and so it has (forwardly) recurrent points. This point is very useful for us to justify the surjectiveness of a semiflow by the so-called ``homogeneity'' condition as follows.

\begin{sn}[\cite{Fur}]\label{sn3.1}
We say that $(T,X)$ is \textit{homogeneous} if there exists a minimal semiflow $(G,X)$ such that $tgx=gtx$ for all $t\in T,g\in G$ and $x\in X$. Here we do not require $(G,X)$ to be a flow.
\end{sn}

\begin{prop}\label{prop3.2B}
Let $(T,X)$ be a homogeneous semiflow. Then
$(T,X)$ is surjective. Hence if $(T,X)$ is in addition equicontinuous it is distal.
\end{prop}

\begin{proof}
Let $t\in T$. Since $(T,X)$ is homogeneous, then the (forwardly) recurrent points are dense in $X$ for $(t,X)$. Because if $x$ is recurrent for $(t,X)$ it is such that $x\in\textrm{cls}_{X^X}{\{t^nx\,|\,n\ge1\}}\subseteq tX$ and $tX$ is closed,
it follows that $t$ is a self-surjection of $X$ for each $t\in T$. Thus $(T,X)$ surjective, and then it is distal by Theorem~\ref{thm2.1} if it is equicontinuous.
\end{proof}

Particularly, if $(T,X)$ is minimal with $T$ abelian, then it is homogeneous and thus $(T,X)$ is surjective by Proposition~\ref{prop3.2B}. Here we will present a more simple independent proof for this as follows.

\begin{cor}\label{cor3.3B}
Let $(T,X)$ be a minimal semiflow with $T$ abelian. Then $(T,X)$ is surjective and hence if $(T,X)$ is in addition equicontinuous it is (uniformly) distal.
\end{cor}

\begin{proof}
Let $Z=tX$ for all $t\in T$. Then $Z$ is closed and since $T$ abelian $Z$ is $T$-invariant. Thus $Z=X$. This completes the proof by Theorem~\ref{thm2.1} (and Theorem~\ref{thm1.14}).
\end{proof}

It should be noticed here that in view of Example~\ref{ex1.7} the abelian condition in Corollary~\ref{cor3.3B}, which guarantees the homogeneity, is essential. This result will be generalized to amenable semigroups by Proposition~\ref{prop3.7B} in $\S\ref{sec3.2B}$, using ergodic theory.

Given any integer $d\ge1$, as a consequence of Proposition~\ref{prop3.2B} and Theorem~\ref{thm2.1}, the following corollary seems to be non-trivial because it is beyond Ellis' joint continuity theorem.

\begin{cor}\label{cor3.4B}
Let $\mathbb{R}_+^d\times X\rightarrow X$ be a separately continuous semiflow, where $(\mathbb{R}_+^d,+)$ is under the usual Euclidean topology.
If $(\mathbb{R}_+^d,X)$ is minimal equicontinuous, then it is distal.
\end{cor}

\begin{proof}
Write $T=\mathbb{R}_+^d$, which is an additive abelian semigroup. First, under the discrete topology of $T$, $(T,X)$ is a minimal semiflow. Then by Corollary~\ref{cor3.3B}, it follows that for each $t\in T$, $x\mapsto tx$ is a continuous surjection of $X$. Therefore, $(\mathbb{R}_+^d,X)$ is distal by Theorem~\ref{thm2.1}.
\end{proof}

Let $\mathrm{Aut}_T(X)$ be the group of automorphisms of $(T,X)$; i.e., $\varphi\in\mathrm{Aut}_T(X)$ iff $\varphi\colon X\rightarrow X$ is 1-1 onto continuous such that $\varphi t=t\varphi$ for all $t\in T$. Then $\mathrm{Aut}_T(X)$ is called \textit{algebraically transitive}~\cite{Aus} iff $\mathrm{Aut}_T(X)x=X$ for some $x\in X$.

Thus by Proposition~\ref{prop3.2B} with $G=\mathrm{Aut}_T(X)$, it follows that if $(T,X)$ is equicontinuous and $\mathrm{Aut}_T(X)$ is algebraically transitive, then $(T,X)$ is distal.

\subsection{Amenable semigroup condition}\label{sec3.2B}
More general than the case of abelian phase semigroup, now we will consider amenable one (cf.~$\S\ref{sec1.1.4}$\,(\textbf{k})).

\begin{snAus}
Let $\mu$ be a Borel probability measure on the compact $T_2$-space $X$. Then:
\begin{enumerate}
\item $\mu$ is called \textit{quasi-regular} if it is ``outer-regular'' for all Borel subsets of $X$ (i.e. for all Borel set $B$ and $\varepsilon>0$ one can find an open set $U$ with $B\subseteq U$ and $\mu(U\setminus B)<\varepsilon$) and each open subset of $X$ is ``inner regular'' for $\mu$ (i.e. for any open set $U$ and $\varepsilon>0$ one can find a compact set $K$ with $K\subset U$ and $\mu(U\setminus K)<\varepsilon$).

\item By $\textrm{supp}\,(\mu)$ we mean the \textit{support} of the Borel probability measure $\mu$ in $X$; i.e., $x\in\textrm{supp}\,(\mu)$ iff every open neighborhood of $x$ has positive $\mu$-measure. Every point of $\textrm{supp}\,(\mu)$ is also called a density point of $\mu$.
\end{enumerate}
\end{snAus}

\begin{lem}\label{lem3.6B}
Let $\mu$ be an invariant quasi-regular Borel probability measure of $(T,X)$; then $\mathrm{supp}\,(\mu)$ is a closed set of $\mu$-measure $1$ such that $t[\mathrm{supp}\,(\mu)]=\mathrm{supp}\,(\mu)$ for all $t\in T$.
\end{lem}

\begin{proof}
Set $S=\textrm{supp}\,(\mu)$. By definition, it easily follows that $S$ is closed; and moreover, $S$ is of $\mu$-measure $1$. Otherwise, by the quasi-regularity of $\mu$ there exists a compact subset $K$ of $X$ with $K\cap S=\emptyset$ such that $\mu(K)>0$; then $K$ contains at least one point of $S$. For, if not, then there is an open neighborhood $V_x$ of any $x\in K$ with $\mu(V_x)=0$ and so by the compactness of $K$, $\mu(K)=0$ contradicting $\mu(K)>0$.

Now given $t\in T$, since $tS$ is a Borel set and $\mu(tS)=1$, we can easily obtain that $tS=S$. Indeed, $tS\supseteq S$ is obvious. (If $S\setminus tS\not=\emptyset$, then $X\setminus tS$ is an open set containing points of $S$ so that $\mu(X\setminus tS)>0$, a contradiction to $\mu(tS)=1$.)
Next assume $tS\supsetneqq S$ and then we can choose some $y\in tS-S$ and $x\in tS$ such that $tx=y$. Now we can pick an open set $U$ with $y\in U\subset X-S$. Hence $0=\mu(U)=\mu(t^{-1}U)$. This contradicts that $x\in t^{-1}U, x\in S$ and $t^{-1}U$ is an open set.

This thus completes the proof of Lemma~\ref{lem3.6B}.
\end{proof}

Now we can easily conclude the following by Theorem~\ref{thm2.1} together with Lemma~\ref{lem3.6B}, which generalizes Corollary~\ref{cor3.3B}.

\begin{prop}\label{prop3.7B}
Let $(T,X)$ be a semiflow with $T$ an amenable semigroup and with a dense set of a.p. points. Then
$(T,X)$ is surjective; and hence if $(T,X)$ is in addition equicontinuous it is distal.
\end{prop}

\begin{note}
In view of Lemma~\ref{lem3.6B}, the statement of Proposition~\ref{prop3.7B} is still true if $(T,X)$ is only a general minimal semiflow admitting an invariant Borel probability measure.
\end{note}

\begin{proof}
Let $x\in X$ be an a.p. point of $(T,X)$ and write $X_x=\textrm{cls}_XTx$. Then $(T,X_x)$ is a minimal
subsemiflow of $(T,X)$. Since $T$ is amenable by hypothesis, hence by amenability and Riesz's theorem there exists an invariant \textit{quasi-regular} Borel probability measure $\mu$ for $(T,X_x)$. Moreover, since $(T,X_x)$ is minimal and $\textrm{supp}\,(\mu)\subseteq X_x$ is $T$-invariant, thus $\textrm{supp}\,(\mu)=X_x$. Then by Lemma~\ref{lem3.6B}, it follows that each $t\in T$ restricted to $X_x$ is a surjection of $X_x$. Thus $x\in tX$ for all $t\in T$.
This shows that $tX=X$ for all $t\in T$, because a.p. points are dense in $X$ and $tX$ is closed.
Finally by Theorem~\ref{thm2.1}, it follows that $(T,X)$ is distal, if it is equicontinuous. This therefore proves Proposition~\ref{prop3.7B}.
\end{proof}

\subsection{\textit{C}-semigroup condition and $\ell$-recurrence}
It was already known that if $x$ is a recurrent point of a continuous self-map $f$ of $X$ then $x\in f(X)$ (by Lemma~\ref{lem2.2}). However, even for a minimal semiflow $(T,X)$, $X\not=tX$ in general; see Example~\ref{ex1.7}. Now we will generalize Lemma~\ref{lem2.2} to semiflows with a kind of special phase semigroups.

\begin{sn}[\cite{KM}]\label{sn3.8B}
Let $T$ be a topological semigroup, which is not necessarily discrete. Then:
\begin{enumerate}
\item $T$ is called a \textit{right C-semigroup} if $Ts$ is relatively co-compact in $T$, i.e., $\textrm{cls}_T(T\setminus{Ts})$ is compact, for each $s\in T$.

\item We could define \textit{left C-semigroup} in a similar way.
\end{enumerate}

When $T$ is right and left \textit{C}-semigroup, it is called a \textit{C-semigroup} as in Definition~(\textbf{l}) in $\S\ref{sec1.1.4}$.
For example, let $T=[1,\infty)$ with $e=1$; then $T$ is a multiplicative \textit{C}-semigroup under the usual topology.
\end{sn}

Next we need the notion---recurrent point---for a semiflow with general phase semigroup beyond $T=\mathbb{Z}_+$.

\begin{sn}\label{sn3.9B}
Let $(T,X)$ be any semiflow, where $T$ is a non-compact topological semigroup with $e\in T$, not necessarily discrete. By $\mathfrak{N}_{\textit{cpt},e}$ we will denote the family of all compact neighborhoods of $e$. Then:
\begin{enumerate}
\item $T$ is called \textit{locally compact} if $e$ has a compact neighborhood in $T$, i.e., $\mathfrak{N}_{\textit{cpt},e}\not=\emptyset$.

\item Given $x\in X$, $y\in X$ is called a \textit{limit point} of $x$, denoted by $y\in\ell_T(x)$, if $y\in\bigcap_{K\in\mathfrak{N}_{\textit{cpt},e}}\textrm{cls}_XK^cx$, where $K^c$ is the complement of $K$ in $T$.

\item If $x\in\ell_T(x)$, then $x$ is called an \textit{$\ell$-recurrent point} of $(T,X)$; if every point of $X$ is $\ell$-recurrent, then $(T,X)$ is called \textit{pointwise $\ell$-recurrent}. See \cite[Definition~2.11]{DT}.
\end{enumerate}

Of course, even if $T=\mathbb{Z}$, an $\ell$-recurrent point need not be an a.p. point for a general semiflow. For instance, every point of $X$ is $\ell$-recurrent for $(\langle T\rangle,X)$ in 1 of Examples~\ref{ex1.4}, but it is not a.p. except the two ends $-1$ and $2$ of $X$.
\end{sn}

\begin{lem}\label{lem3.10B}
Let $(T,X)$ be a semiflow with a locally compact phase semigroup $T$ and $x\in X$. Then $x$ is an $\ell$-recurrent point of $(T,X)$ if and only if there is a net $\{t_n\,|\,n\in D\}$ in $T$ such that
\begin{enumerate}
\item[$(1)$] $t_nx\to x$,
\item[$(2)$] for every $K\in\mathfrak{N}_{\textit{cpt},e}$, there is some $n_K\in D$ with $t_n\in K^c$ for all $n\ge n_K$.
\end{enumerate}
\begin{note}
If a net $\{t_n\}$ in $T$ satisfies condition (2), then we shall say $t_n\to\infty$.
\end{note}
\end{lem}

\begin{proof}
The sufficiency is obvious; so we only need to prove the necessity. For this, assume $x$ is an $\ell$-recurrent point of $(T,X)$.

Let $\mathscr{U}_x$ be the neighborhoods filter of $x$. Define a binary relation $\ge$ on $\mathscr{U}_x\times\mathfrak{N}_{\textit{cpt},e}$ as follows:
$(U,K)\ge(U^\prime,K^\prime)\Leftrightarrow U\subseteq U^\prime\textrm{ and }K\supseteq K^\prime$. Then,
\begin{enumerate}
\item[(a)] if $(U,K)\ge(U^\prime,K^\prime)$ and $(U^\prime,K^\prime)\ge(U^{\prime\prime},K^{\prime\prime})$, then $(U,K)\ge(U^{\prime\prime},K^{\prime\prime})$;
\item[(b)] if $(U,K)\in\mathscr{U}_x\times\mathfrak{N}_{\textit{cpt},e}$, then $(U,K)\ge(U,K)$;
\item[(c)] if $(U,K),(U^\prime,K^\prime)\in\mathscr{U}_x\times\mathfrak{N}_{\textit{cpt},e}$, then there is $(U^{\prime\prime},K^{\prime\prime})\in\mathscr{U}_x\times\mathfrak{N}_{\textit{cpt},e}$ such that $(U^{\prime\prime},K^{\prime\prime})\ge(U,K)$ and $(U^{\prime\prime},K^{\prime\prime})\ge(U^{\prime},K^{\prime})$.
\end{enumerate}
Thus $(\mathscr{U}_x\times\mathfrak{N}_{\textit{cpt},e},\ge)$ is a directed set. Now for every $(U,K)\in\mathscr{U}_x\times\mathfrak{N}_{\textit{cpt},e}$, we can take some $t_{U,K}\in T$ such that
$t_{U,K}x\in U$ and $t_{U,K}\in K^c$. It is easy to see that $\{t_{U,K}\}$ is a net in $T$ satisfies conditions (1) and (2). This proves Lemma~\ref{lem3.10B}.
\end{proof}

\begin{remarks}\label{rem3.11B}
Suppose $(T,X)$ is a semiflow where $T$ is a locally compact non-compact topological semigroup.
\begin{enumerate}
\item[(a)] An $x$ of $X$ is not necessarily an $\ell$-recurrent point if there is only an infinite sequence $\{t_n\}$ in $T$ with $t_nx\to x$.

For example, let $X=\mathbb{R}\cup\{\infty\}$ be the one-point compactification of the $1$-dimensional Euclidean space $\mathbb{R}$ (so $X$ is homeomorphic with the unit circle) and let $T=(\mathbb{R},+)$ with the usual topology.
Define a flow on $X$ with the phase group $T$ as follows:
\begin{gather*}
T\times X\rightarrow X,\quad (t,x)\mapsto t+x.
\end{gather*}
If $t_n\to0$ in $T$, then $t_nx\to x$ for each $x\in X$. But $\ell_T(x)=\{\infty\}$ for all $x\in X$.

\item[(b)] When $T$ is a group, an a.p. point is always an $\ell$-recurrent point. Thus any minimal flow is pointwise $\ell$-recurrent.
\begin{proof}
If $A$ is a right-syndetic subset of $T$, then $A$ is never contained in any $K\in\mathfrak{N}_{\textit{cpt},e}$ for $T$ is non-compact.
\end{proof}

\item[(c)] More generally than the above (b), let $T$ be such that each right-syndetic set is not relatively compact in $T$. Then every a.p. point is $\ell$-recurrent for $(T,X)$.
\end{enumerate}
\end{remarks}

\begin{remarks}\label{rem3.12B}
The almost periodicity is a strong form of recurrence, yet (b) of Remark~\ref{rem3.11B} is false in general if $T$ is \textit{not} a group, even for semiflows on compact metric spaces with no isolated points. Let's construct such an example as follows.
\begin{enumerate}
\item[(1)] Let $Y$ be a locally compact, non-compact, Polish space with no isolated points like $Y=\mathbb{R}^d$; and let $T=\{e\}\cup Y$, where $e=id_Y\colon y\mapsto y$ is the identity self-map of $Y$ and for every $t\in T$ with $t\not=e$ let $t\colon y\mapsto t$ of $Y$ into $Y$ be the constant map. Then $T$ is a locally compact, non-compact, $\sigma$-compact (in fact separable), and non-abelian topological subsemigroup of $C(Y,Y)$ under the topology defined by the way: for every net $\{t_n\}$ in $T$,
    \begin{equation*}
     t_n\to t\textrm{ in }T \Leftrightarrow t_ny\to ty\ \forall y\in Y.
     \end{equation*}
     In this case, $e$ is an isolated point of $T$ and $T\setminus\{e\}$ is homeomorphic with $Y$, i.e., $t_n\to t$ in $T$ iff $t_n\to t$ in $Y$.

\item[(2)] We now consider the naturally induced semiflow on $Y$ with the phase semigroup $T$ as follows:
\begin{equation*}
T\times Y\rightarrow Y,\ (t,y)\mapsto ty\quad \textrm{where }\ ty=\begin{cases}y &\textrm{ if }t=e,\\ t &\textrm{ if }t\not=e.\end{cases}
\end{equation*}
Given $y_0\in Y$ set $S_{y_0}=\{t\in T\,|\,ty_0=y_0\}=\{e,t_{y_0}\}$ where $t_{y_0}y=y_0\ \forall y\in Y$. Clearly $S_{y_0}$ is a right-syndetic subsemigroup of $T$ by (ii) of Definition~\ref{sn2.4} so every point of $Y$ is a \textit{periodic point} of $(T,Y)$. However $Ty=\textrm{cls}_YTy=Y$, for all $y\in Y$, is not compact.

\item[{}] (Notice that it is a well-known fact that
\begin{quote}
{\it Let $(G,X)$ be a flow with $G$ a locally compact separable group and $X$ a locally compact $T_2$-space, and $x\in X$. Then:}
\begin{enumerate}
\item {\it If $x$ is an a.p. point, $\mathrm{cls}_XGx$ is compact} (cf.~\cite[Proposition~2.5]{E69} and \cite[Lemma~1.6]{Aus});
\item {\it $x$ is periodic if and only if $Gx$ is compact} (cf.~\cite[Theorem~1.5]{Aus}).
\end{enumerate}
\end{quote}
But $(T,Y)$ shows that these statements need not be true in general semiflows.)

\item[(3)] Further based on (1) and (2), define $X=Y\cup\{\infty\}$ to be the one-point compactification of $Y$.
We now consider the naturally induced semiflow on $X$
\begin{equation*}
T\times X\rightarrow X,\ (t,x)\mapsto tx\quad \textrm{where }\ tx=\begin{cases}x &\textrm{ if }t=e,\\t &\textrm{ if }t\not=e.\end{cases}
\end{equation*}
For every $x\in X$ and all neighborhood $U$ of $x$, $N_T(x,U)$ is right-syndetic in $T$. (In fact, take $K\subseteq U$ a compact subset of $T$ and let $t\in T$, then $Ktx\subset U$ so $Kt\cap N_T(x,U)\not=\emptyset$. Thus $N_T(x,U)$ is right-syndetic, which is left-thick but not right-thick.) Clearly, $Tx=Y$ dense in $X$ for all $x\in Y$ and $T\infty=X$. This shows that
    \begin{itemize}
    \item {\it $(T,X)$ is minimal, pointwise a.p. and equicontinuous, but not distal.}
    \end{itemize}
    Nevertheless,
     \begin{itemize}
    \item {\it $x\not\in\ell_T(x)\ \forall x\not=\infty$; in fact, $\ell_T(x)=\{\infty\}$ for all $x\in X$. That is, $\infty$ is the unique $\ell$-recurrent point of $(T,X)$.}
    \begin{proof}
     For every $x\in X, y\in Y$, and taking a compact neighborhood $K$ of $y$ in $Y$, $\{e\}\cup K\in\mathfrak{N}_{\textit{cpt},e}$ such that $y\not\in\textrm{cls}_X(\{e\}\cup K)^cx$ so $y\not\in\ell_T(x)$. Moreover, $\ell_T(x)=\{\infty\}$ is obvious.
     \end{proof}
    \end{itemize}
\item[{}] (Note here that $T$ is neither an amenable semigroup nor a \textit{C}-semigroup.)
\end{enumerate}
\end{remarks}

\begin{remarks}\label{rem3.13B}
Let $T$ be a locally compact, $\sigma$-compact, and non-compact topological semigroup with an increasing sequence $\{K_n\}$ of compact neighborhoods of $e$ such that $T=\bigcup_nK_n$ and let $(T,X)$ be a semiflow. Then:
\begin{enumerate}
\item[(1)] $\ell_T(x)=\bigcap_n\textrm{cls}_XK_n^cx$ for all $x\in X$. Thus, if $X$ is a metric space, then $y\in\ell_T(x)$ if and only if $\exists\,t_n\in K_n^c$ with $t_nx\to y$ as $n\to\infty$.
\item[(2)] If $s^{-1}K$ is relatively compact in $T$ for all $s\in T$ and $K\in\mathfrak{N}_{\textit{cpt},e}$, then $\ell_T(x)$ is invariant for $(T,X)$ with $X$ a metric space. Thus $\ell_T(x)$, for $x\in X$, is an invariant closed non-empty set if $X$ is a compact metric space.
\begin{proof}
Indeed, for all $y\in\ell_T(x)$ and $s\in T$, let $t_n\in K_n^c$ with $t_nx\to y$. Then $st_nx\to sy$. For every compact subset $K$ of $T$, there is some $n_0>0$ such that $st_n\not\in K$ as $n>n_0$. This shows that we can select out a subsequence $\{\tau_n\}$ from $\{st_n\}$ with $\tau_n\in K_n^c$ such that $\tau_nx\to sy$. Thus $\ell_T(x)$ is invariant for all $x\in X$.
\end{proof}
\end{enumerate}

We notice that the classical topological semigroups $T=\mathbb{R}_+^d$ and $\mathbb{Z}_+^d$ both are locally compact non-compact $\sigma$-compact.
\end{remarks}

\begin{remarks}\label{rem3.14B}
Let $(T,X)$ be a semiflow on a uniform $T_2$-space $(X,\mathscr{U}_X)$ not necessarily compact with phase semigroup $T$. Then:
\begin{enumerate}
\item[(a)] A point $x\in X$ is called \textit{Birkhoff recurrent} if for every $\varepsilon\in\mathscr{U}_X$ one can find a compact subset $K$ of $T$ such that $Tx\subseteq\varepsilon[Ktx]\ \forall t\in T$ or equivalently $\overline{Tx}\subseteq\varepsilon[Ky]\ \forall y\in \overline{Tx}$; see \cite[Definition~V7.05]{NS} for $T=\mathbb{R}$ and \cite[Definition~3.1]{CD} for $T$ in general topological semigroups.

\item[(b)] By (a), a Birkhoff recurrent point must be an a.p. point. In fact it has been proved that
\begin{quote}
{\it If $(T,X)$ is a semiflow with $X$ a compact $T_2$ space, then a point $x$ of $X$ is a.p. if and only if it is Birkhoff recurrent} (\cite[Theorem~4.1]{CD}).
\end{quote}
Whenever $T$ is a \textbf{group} and $X$ is a \textit{locally compact $T_2$} space instead of a compact $T_2$ space, this statement still holds (cf.~\cite[Corollary~4.2]{CD}). In view of this, the following question is natural:
\begin{quote}
{\it Does the statement of \cite[Theorem~4.1]{CD} still hold if $(T,X)$ is a \textbf{semiflow} on a locally compact $T_2$ space $X$?} (cf.~\cite[Question~4.9]{CD})
\end{quote}
\item[(c)] Now in the same situation of (2) of Remark~\ref{rem3.12B}, $Y$ is a locally compact, non-compact, Polish space. If $y\in Y$ were Birkhoff recurrent for $(T,Y)$, then $\textrm{cls}_YTy=Y$ would be compact by \cite[Lemma~3.4]{CD}. Therefore, every point of $Y$ is a.p. but not Birkhoff recurrent. \textit{This thus gives us a negative solution to} \cite[Question~4.9]{CD}.
\end{enumerate}
\end{remarks}

\begin{remark}\label{rem3.15B}
Let $(T,X)$ be a semiflow with $T$ a locally compact non-compact semigroup and $x\in X$. If $Tx$ is dense in $X$ such that $\textrm{Int}_XTx=\emptyset$, then $X=\ell_T(x)$; particularly, $x$ is $\ell$-recurrent.
\begin{proof}
Given $y\in X$, let $U$ be an arbitrary neighborhood of $y$ and $K\in\mathfrak{N}_{\textit{cpt},e}$. Then $U\nsubseteq Kx$; otherwise, $\textrm{Int}_XTx\not=\emptyset$. Then $tx\in U$ for some $t\in K^c$. Thus $y\in\ell_T(x)$.
\end{proof}
\end{remark}

Now we can generalize Lemma~\ref{lem2.2} from the special case $T=\mathbb{Z}_+$ to every left \textit{C}-semigroup (cf.~2. of Definition~\ref{sn3.8B}) as follows:

\begin{prop}\label{prop3.16B}
Let $(T,X)$ be a semiflow with $T$ a locally compact, non-compact, left \textit{C}-semigroup and $x\in X$. If $y\in\ell_T(x)$, then $y\in t\mathrm{cls}_XTx$ for every $t\in T$. Hence $\ell_T(x)\subseteq tX$ for all $t\in T$.
\end{prop}

\begin{proof}
Let $t\in T$. Since $T\setminus tT$ is relatively compact in $T$ and $y$ is a limit point of $x$ (cf.~2. of Definition~\ref{sn3.9B}), there is a net $\{t_n\}$ in $T$ with $tt_nx\to y$. Take $t_nx\to z\in\textrm{cls}_XTx$ (passing to a subnet if necessary). Thus $tz=y$ so $y\in t\textrm{cls}_XTx$. This proves Proposition~\ref{prop3.16B}.
\end{proof}

The following is a simple consequence of Proposition~\ref{prop3.16B}, which generalizes \cite[Lemma~3.1]{AAB} from $T=\mathbb{Z}_+$ to a general left \textit{C}-semigroup.

\begin{cor}\label{cor3.17B}
Let $(T,X)$ be a semiflow with $T$ a locally compact, non-compact, left \textit{C}-semigroup. If $(T,X)$ is pointwise $\ell$-recurrent, i.e., $x\in\ell_T(x)\ \forall x\in X$, then $(T,X)$ is surjective.
\end{cor}

Note that an $\ell$-recurrent point need not be a minimal point. So Corollary~\ref{cor3.17B} is comparable with Proposition~\ref{prop3.7B}. Moreover, (3) of Remark~\ref{rem3.12B} shows that the left \textit{C}-semigroup condition is essential for Corollary~\ref{cor3.17B}, since $\infty\not\in tX$ for all $t\in T, t\not=e$.

\begin{cor}\label{cor3.18B}
Let $(T,X)$ be a semiflow with $T$ a locally compact, non-compact, left \textit{C}-semigroup and $x\in X$. Then:
\begin{enumerate}
\item[$(1)$] If $Tx$ is dense in $X$ with
$\mathrm{Int}_XTx=\emptyset$, then $X=tX$ for all $t\in T$.

\item[$(2)$] If $(T,X)$ is equicontinuous and $\mathrm{cls}_XTx=X$ with
$\mathrm{Int}_XTx=\emptyset$, then $(T,X)$ is a minimal surjective semiflow.
\end{enumerate}
\end{cor}

\begin{proof}
(1) Let $Tx$ be dense in $X$ with
$\mathrm{Int}_XTx=\emptyset$. By Remark~\ref{rem3.15B}, $\ell_T(x)=X$. Then the assertion (1) follows at once from Proposition~\ref{prop3.16B}.

(2) Based on (1) it follows that $(T,X)$ is surjective. Then by Theorem~\ref{thm2.1}, $(T,X)$ is distal and so minimal. This proves Corollary~\ref{cor3.18B}.
\end{proof}

In view of Example~\ref{ex1.12}, the condition ``$\mathrm{Int}_XTx=\emptyset$'' is important for the assertions of Corollary~\ref{cor3.18B}.
\section{Inheritance theorems}\label{sec4A}
It is a well-known fact that for every flow $(T,X)$ and for all right-syndetic subgroup $S$ of $T$, $(T,X)$ is distal if and only if $(S,X)$ is distal (cf.~\cite[Proposition~5.14]{E69}). In fact, this kind of inheritance theorem also holds for semiflows with phase semigroups as follows:

\begin{prop}[Inheritance theorem]\label{prop4.1}
Let $(T,X)$ be a semiflow with phase semigroup $T$ not necessarily discrete, and let $S$ be a right-syndetic subsemigroup of $T$. Then:
\begin{enumerate}
\item[$(1)$] $P(T,X)=P(S,X)$;
\item[$(2)$] $(T,X)$ is distal if and only if $(S,X)$ is distal;
\item[$(3)$] $(T,X)$ is invertible if and only if so is $(S,X)$;
\item[$(4)$] If $(T,X)$ is invertible, then $Q(T,X)=Q(S,X)$.
\item[$(5)$] $(T,X)$ is equicontinuous surjective if and only if so is $(S,X)$;

\end{enumerate}
\end{prop}

\begin{note}
When $T$ is a topological group, see~\cite[Lemma~5.13]{E69} for (1) of Proposition~\ref{prop4.1}, \cite[Proposition~5.14]{E69} for (2) of Proposition~\ref{prop4.1}, \cite[Lemma~4.16]{E69} for (4) of Proposition~\ref{prop4.1}, and \cite[Proposition~4.17]{E69} for (5) of Proposition~\ref{prop4.1}. Here (3), (4) and Theorem~\ref{thm2.1} are useful for proving (5).
\end{note}

\begin{proof}
(1) Evidently $P(S,X)\subseteq P(T,X)$. On the other hand, let $(x,y)\in P(T,X)$ and let $\alpha\in\mathscr{U}_X$, then $A_\alpha:=\{t\in T\,|\,t(x,y)\in\alpha\}$ is a right-thick set of $T$. Since $S$ is right-syndetic in $T$, thus $S\cap A_\alpha\not=\emptyset$. This shows $(x,y)\in P(S,X)$. Thus $P(S,X)=P(T,X)$.

(2) Since ``distal $\Leftrightarrow$ $P=\varDelta_X$'' for every semiflow on $X$, then (2) follows at once from (1).

(3) Let $(S,X)$ be invertible. Since $S$ is right-syndetic in $T$, there is a compact subset $K$ of $T$ such that for any $t\in T$, there are $k\in K$ and $s\in S$ with $kt=s$. Let $K^\prime=\{k\in K\,|\,\exists t\in T\textit{ s.t. }kt\in S\}$; then for any $t\in T$, there is some $k^\prime\in K^\prime$ with $k^\prime t=s\in S$. This implies that every $t\in T$ is an injection of $X$ and each $k^\prime\in K^\prime$ is a surjection of $X$.
Thus each $k\in K^\prime$ is a homeomorphism of $X$ so that each $t\in T$ is a homeomorphism of $X$.

(4) Clearly $Q(T,X)\supseteq Q(S,X)$. Let $K$ be a compact subset of $T$ such that for any $t\in T$, there are $k_t\in K$ and $s_t\in S$ such that $k_tt=s_t$. Given any $\alpha\in\mathscr{U}_X$, there is some $\beta\in\mathscr{U}_X$ with $K\beta\subseteq\alpha$. Then $t^{-1}\beta=s_t^{-1}k_t\beta\subseteq s_t^{-1}\alpha$ so that $T^{-1}\beta\subseteq S^{-1}\alpha$. This shows that $Q(T,X)\subseteq Q(S,X)$.

(5) The necessity holds obviously. Now suppose $(S,X)$ is equicontinuous surjective. Then by Theorem~\ref{thm2.1}, $(S,X)$ is invertible and so is $(T,X)$ by (3). Thus by (4), $Q(T,X)=Q(S,X)$. Then $Q(T,X)=\varDelta_X$ by Theorem~\ref{thm1.14} and Lemma~\ref{lem1.10}. Thus $(T,X)$ is equicontinuous surjective by Corollary~\ref{cor1.16}.

The proof of Proposition~\ref{prop4.1} is thus completed.
\end{proof}

Note that in (5) of Proposition~\ref{prop4.1}, since the right-syndetic subsemigroup $S$ need not be dense in $T$, this statement is thus non-trivial. Moreover according to Theorem~\ref{thm1.14} and the later Theorem~\ref{u8.3}, it can be equivalently illustrated as follows:
\begin{enumerate}
\item[$(5)^\prime$] {\it $(T,X)$ is uniformly distal if and only if so is $(S,X)$.}
\end{enumerate}
or
\begin{enumerate}
\item[$(5)^{\prime\prime}$] {\it $E(T,X)$ is a group in $C(X,X)$ if and only if so is $E(S,X)$.}
\end{enumerate}

Next we can obtain a simple consequence of Proposition~\ref{prop4.1}. The following is, more or less, motivated by Clay's \cite[Theorem~9]{C63}.

\begin{prop}\label{prop4.2}
Let $(T,X)$ be a semiflow with $T$ an abelian semigroup not necessarily discrete, and let $S$ be a right-syndetic subsemigroup of $T$. If there are a point $x$ such that $Tx$ is dense in $X$ and a fixed point $p$ (i.e. $Tp=\{p\}$), then $Q(T,X)=Q(S,X)=X\times X$.
\end{prop}

\begin{proof}
We first show that $Tx\times Tx\subset P(T,X)$. In fact, for all $t,s\in T$ and $\alpha\in\mathscr{U}_X$, we can find some $\tau\in T$ such that $\{t,s\}\tau x\subset\frac{\alpha}{3}[p]$. Then $\tau(tx,sx)\in\alpha$. This implies that $(tx,sx)\in P(T,X)$. Thus $Tx\times Tx\subseteq P(T,X)=P(S,X)$ by (1) of Proposition~\ref{prop4.1}. Further by $\textrm{cls}_{X\times X}{P}\subseteq Q$, it follows that $Q(T,X)=Q(S,X)=X\times X$.
This proves Proposition~\ref{prop4.2}.
\end{proof}

Now, in Theorem~\ref{thm2.1}, the condition that $(T,X)$ is surjective may be superficially relaxed by using Proposition~\ref{prop4.1} as follows:

\begin{cor}\label{cor4.3}
Let $(T,X)$ be a semiflow such that $S=\{t\in T\,|\,t\textrm{ is a self-surjection of } X\}$ is right-syndetic in $T$. If $(T,X)$ is equicontinuous, then it is distal and hence it is invertible.
\end{cor}

\begin{proof}
Clearly $S$ is a right-syndetic subsemigroup of $T$. Thus by Theorem~\ref{thm2.1}, $(S,X)$ is distal. So $(T,X)$ is distal by (2) of Proposition~\ref{prop4.1}.
\end{proof}

Finally, we note that the compactness of the phase space $X$ is important for Theorem~\ref{thm2.1}. Otherwise, the statement is false; see \cite[Example~3.7]{DX}.
\section{Distality of points by product almost periodicity}\label{sec5}
It is well known that $(T,X)$ is distal iff $(T,X\times X)$ is pointwise a.p. (cf.~\cite[Proposition~2.5]{DX}; also see 1. and 3. of \cite[Proposition~5.9]{E69} for flows). In fact, by a purely topological proof, we can obtain the following characterization of distal points, which implies that every distal point is an a.p. point.
Here `distal point' and `a.p. point' are as in $\S\ref{sec1.1.3}$\,(\textbf{h}) and Definition~\ref{sn2.4}\,(iii) respectively.

\begin{thm}\label{thm5.1}
Let $(T,X)$ be a semiflow and $x\in X$. Then $x$ is a distal point of $(T,X)$ if and only if $(x,y)$ is an a.p. point of $(T,X\times X)$ for all a.p. point $y$ of $(T,X)$.
\end{thm}

\begin{proof}
(1) Necessity: Let $y\in X$ be any a.p. point of $(T,X)$. By Zorn's lemma, there exists a maximal subset $A$ of $X$ with $y\in A$ such that for all $a_1,\dotsc,a_k$ in $A$, $(a_1,\dotsc,a_k)$ is a.p. for $(T,X^k)$, for all $k\ge1$. Now for $\z=(z_a)_{a\in A}\in X^A$ with $\z_a=a\ \forall a\in A$, we can take an a.p. point $(\z^\prime,x^\prime)$ in $\textrm{cls}_{X^A\times X}{T(\z,x)}$ for $(T,X^A\times X)$. Since $\z$ is a.p. for $(T,X^A)$, then there is a net $\{t_n\}$ in $T$ with $t_n(\z^\prime,x^\prime)\to(\z,x^*)$ and $(\z,x^*)$ is also a.p. for $(T,X^A\times X)$. So $x^*\in A$ by maximality of $A$. Further we can select a net $\{s_n\}$ in $T$ such that $s_n(\z,x)\to(\z,x^*)$ and then $s_n(x^*,x)\to(x^*,x^*)$ with $x^*\in\textrm{cls}_XTx$. Thus $x=x^*\in A$ by distality of $(T,X)$ at $x$. Then $x,y\in A$. Therefore by definition of $A$, $(x,y)$ is a.p. for $(T,X\times X)$.

(2) Sufficiency: Since $X$ is compact, by Zorn's lemma we can choose a $y_0\in X$ which is a.p. for $(T,X)$. So $x$ is a.p. for $(T,X)$ and further every $y\in\textrm{cls}_XTx$ is a.p.. Thus, for all $y\in\textrm{cls}_XTx$, $(x,y)$ is a.p.. This implies that $x$ must be distal (by Lemma~\ref{lem2.7}).

The proof of Theorem~\ref{thm5.1} is thus completed.
\end{proof}

It should be noticed that by using IP$^*$-recurrence of a distal point and his central sets of $\mathbb{Z}_+$, Furstenberg's \cite[(i) $\Leftrightarrow$ (iv) in Theorem~9.11]{Fur} says that $x\in X$ is distal for $(T,X)$ iff for any $(T,Z)$, $(x,z)$ is a.p. for $(T,X\times Z)$ for all a.p. point $z\in Z$, in the special case $T=\mathbb{Z}_+$ with $X$ a compact metric space (cf.~\cite[Theorem~4]{DL} for general semiflows on compact $T_2$-spaces).

\begin{sn}
We say that $(T,X)$ satisfies the \textit{Bronstein condition} if the set of a.p. points of $(T,X\times X)$ is dense in $X\times X$.
\end{sn}

The Bronstein condition is a very important one in topological dynamics; see, e.g., \cite{V77}. Then as a consequence of Theorem~\ref{thm5.1} and Lemma~\ref{lem1.8}, we can easily obtain the following result, which says that the point-distal (cf.~$\S\ref{sec1.1.3}$\,(\textbf{i})) implies the Bronstein condition.

\begin{cor}\label{cor5.3}
If $(T,X)$ is a point-distal surjective semiflow, then $(T,X)$ satisfies the Bronstein condition.
\end{cor}

\begin{proof}
Since $(T,X)$ is surjective point-distal, then by Lemma~\ref{lem1.8} it follows that the distal points are dense in $X$. Then by Theorem~\ref{thm5.1}, for all distal point $x\in X$ and every $y\in X$, $(x,y)$ is a.p. for $(T,X\times X)$. This proves Corollary~\ref{cor5.3}.
\end{proof}

If $(T,X)$ is invertible point-distal with $T$ an amenable semigroup, then we shall show later on that its reflection $(X,T)$ is point-distal (cf.~Proposition~\ref{prop6.20}). Here, based on Theorem~\ref{thm5.1}, we can first prove that $(\langle T\rangle,X)$ is point-distal.

\begin{thm}\label{thm5.4}
Let $(T,X)$ be invertible with $T$ an amenable semigroup and $x\in X$. Then:
\begin{enumerate}
\item $x$ is a distal point of $(T,X)$ iff $x$ is a distal point of $(\langle T\rangle,X)$.
\item $(T,X)$ is point-distal iff $(\langle T\rangle,X)$ is a point-distal flow.
\item $(X,T)$ is point-distal iff $(\langle T\rangle,X)$ is a point-distal flow.
\end{enumerate}
\begin{note}
In fact the sufficiency of 1. and 2. does not need the amenability of $T$.
\end{note}
\end{thm}

\begin{proof}
(1). Clearly if $x$ is a distal point of $(\langle T\rangle,X)$, then it is a distal point of $(T,X)$. Conversely, let $x$ be a distal point of $(T,X)$; we will show $x$ is also a distal point of $(\langle T\rangle,X)$. Given $y\in\textrm{cls}_XTx$, by Theorem~\ref{thm5.1}, $(y,x)$ is an a.p. point of $(T,X\times X)$. Then $W=\textrm{cls}_{X\times X}T(y,x)$ is a minimal subset of $(T,X\times X)$ by Lemma~\ref{lem2.6}. Since $T$ is amenable, by Proposition~\ref{prop3.7B}, it follows that $(\langle T\rangle,W)$ is a minimal subflow of $(\langle T\rangle,X\times X)$ and so $\textrm{cls}_X\langle T\rangle x=\textrm{cls}_XTx$. Thus by Lemma~\ref{lem2.6} again, $(y,x)$ is an a.p. point of $(\langle T\rangle,X\times X)$. Using Theorem~\ref{thm5.1} again, $x$ is a distal point of $(\langle T\rangle,X)$.

(2). In view of 1. of Theorem~\ref{thm5.4}, we only need to show that if $(\langle T\rangle,X)$ is minimal, then $(T,X)$ is minimal. In fact, when $\Lambda$ is a minimal subset of $(T,X)$, by a slight modification of the proof of Corollary~\ref{cor1.23} we can see $\Lambda=X$.

(3). This follows from (2). Thus the proof of Theorem~\ref{thm5.4} is completed.
\end{proof}

\begin{cor}
Let $(T,X)$ be an invertible semiflow with $T$ amenable and $x_0\in X$. If $x_0$ is a distal point of $(T,X)$ and $\overline{\langle T\rangle x_0}=X$, then $(T,X)$ is point-distal.
\end{cor}

\begin{proof}
First by Theorem~\ref{thm5.4}, $x_0$ is a distal point of $(\langle T\rangle,X)$; so $(\langle T\rangle,X)$ is point-distal and moreover $\langle T\rangle x_0$ consists of distal points.
\end{proof}

Next we will present another application of Theorem~\ref{thm5.4}. In 1970 \cite{V70} Veech proved the following theorem:
\begin{quote}
{\it If $(T,X)$ is a point-distal flow on a non-trivial compact metric space $X$, then $(T,X)$ has a non-trivial equicontinuous factor} (cf.~\cite[Theorem~6.1]{V70})
\end{quote}

Next by 2. of Theorem~\ref{thm5.4} and Veech's theorem we can easily obtain an invertible semiflow version of Veech's theorem as follows:

\begin{cor}\label{cor5.6}
Let $(T,X)$ be point-distal invertible with $T$ an amenable semigroup and with $X$ a non-trivial compact metric space. Then $(T,X)$ has a non-trivial equicontinuous (invertible) factor.
\end{cor}

\begin{que}\label{q5.7}
{\it Let $(T,X)$ be a point-distal invertible semiflow with $T$ an any phase semigroup and with $X$ a non-trivial compact $T_2$ space. Does it have a non-trivial equicontinuous factor?} This is also open in point-distal flows (cf.~Veech \cite[p.~802]{V77}).
\end{que}

\begin{que}\label{q5.8}
Let $T$ be a locally compact non-compact topological semigroup, $(T,X)$ a semiflow and $x\in X$. If $(x,y)$ is an $\ell$-recurrent point of $(T,X\times Y)$ in the sense of Definition~\ref{sn3.9B}.3 for every $\ell$-recurrent point $y$ of any $(T,Y)$, is $x$ a distal point of $(T,X)$? (See \cite[(i) $\Leftrightarrow$ (iii) of Theorem~9.11]{Fur} for $T=\mathbb{Z}_+$.)
\end{que}

Our later Theorem~\ref{u8.12} will provide us with sufficient and necessary conditions for any point-distal flow to have non-trivial equicontinuous factors.
\section{Dynamics of reflections of invertible semiflows}\label{sec6}
This section will be mainly devoted to proving (2), (3) and (4) of Theorem~\ref{thm1.13} and our Reflection principles I, II and III using Theorem~\ref{thm2.1}. As applications of our reflection principles, we will prove Furstenberg's structure theorem of minimal distal semiflows in $\S\ref{sec3.2A}$ and we shall consider minimal non-sensitive invertible semiflows in $\S\ref{sec7}$.

Recall that a semiflow $(T,X)$ is invertible iff each $t\in T$ is bijective; and in this case, $\langle T\rangle$ denotes the group generated by $T$. Then $(\langle T\rangle,X)$ is a flow on $X$. However since $T$ is neither a right-syndetic nor a normal subsemigroup of $\langle T\rangle$ in general, the dynamics properties of $(\langle T\rangle,X)$ can not be naturally inherited to $(T,X)$ in many cases.

When $(T,X)$ is invertible, $(X,T)$ denotes its reflection or `history' defined as in Definition~\ref{sn1.18}. If $(T,X)$ had certain dynamical property $\mathfrak{P}$ in the past, i.e., $(X,T)$ has $\mathfrak{P}$, then does $(T,X)$ have $\mathfrak{P}$? This kind of dynamics is called satisfying ``reflection principle'' here.
\subsection{Distality and equicontinuity of reflections}\label{sec6.1}
Theorem~\ref{thm2.1} implies the following, for which we will present a direct proof with no uses of Ellis' joint continuity theorem (Theorem~\ref{thm9.8} in $\S\ref{sec9}$) and Ellis's semigroup (cf.~$\S\ref{sec1.1.5}$).

\begin{prop}\label{prop6.1}
Let $(T,X)$ be an invertible semiflow; then $(T,X)$ is equicontinuous if and only if so is $(X,T)$.
\end{prop}

\begin{proof}
By symmetry we only prove the ``only if'' part and so assume $(T,X)$ is equicontinuous. To be contrary, suppose that $(X,T)$ is not equicontinuous at some point $x\in X$. Then there are $x_i, x_i^\prime$ with $x_i\to x$ and $x_i^\prime\to x$ in $X$ and $t_i\in T$ such that
$(x_it_i,x_i^\prime t_i)=t_i^{-1}(x_i,x_i^\prime)\to(z,z^\prime)$ where $z\not=z^\prime$.
This shows that $(z,z^\prime)\in Q(T,X)$; i.e., $z$ is regionally proximal to $z^\prime$ for $(T,X)$ (cf.~Definition~(j) in $\S$\ref{sec1.1.3}). Then it follows easily from the equicontinuity of $(T,X\times X)$ that $(z,z^\prime)$ is a proximal pair of $(T,X)$, contradicting $(T,X)$ distal by Theorem~\ref{thm2.1}. Thus $(X,T)$ must be equicontinuous. This proves Proposition~\ref{prop6.1}.
\end{proof}

\begin{sn}
Let $E$ be a multiplicative semigroup. Then:
\begin{enumerate}
\item A \textit{left ideal} in $E$ is a non-empty subset $I$ such that $EI\subseteq I$.
\item A \textit{minimal left ideal} in $E$ is one which does not properly contain a left ideal.

\item Let $J(I)$ denote the set of idempotents in a left ideal $I$; i.e., $u\in J(I)$ iff $u^2=u$ and $u\in I$.
    \end{enumerate}
\end{sn}

This is more general than $\S\ref{sec1.1.5}$\,(\textbf{o}); yet we will be mainly interested to the special case $E=E(X)$ associated to a semiflow $(T,X)$.
Since $E$ is a compact $T_2$ right-topological semigroup in this case, hence $J(I)\not=\emptyset$ for all minimal left ideal in $E$ (by \cite[Lemma~6.6]{Aus}).

We will need a purely algebraic lemma for us to characterize the distality of any semiflows (Lemma~\ref{lem6.7} and Theorem~\ref{thm6.22} below).

\begin{lem}[{cf.~\cite[Lemmas~6.1, 6.2 and 6.3]{Aus}}]\label{lem6.3}
Let $E$ be any semigroup and let $I,I^\prime$ be two minimal left ideals in $E$ with $J(I)\not=\emptyset$. Then:
\begin{enumerate}
\item[$(1)$] $Ip=I$ for all $p\in I$.
\item[$(2)$] $pu=p$ for all $u\in J(I), p\in I$.
\item[$(3)$] If $u\in J(I)$ and $p\in I$ with $up=u$, then $p\in J(I)$.
\item[$(4)$] If $u\in J(I)$ then $uI$ is a group with the neutral element $u$.
\item[$(5)$] If $p\in I$ then there is a unique $u\in J(I)$ with $up=p$.
\item[$(6)$] Let $u,v\in J(I)$ and let $p\in uI$. Then there is an $r\in I$ with $rp=v$ and $pr=u$.
\item[$(7)$] $I=\bigcup_{u\in J(I)}uI$.
\item[$(8)$] If $u,v\in J(I)$ with $u\not=v$, then $uI\cap vI=\emptyset$.
\item[$(9)$] Suppose $p,q,r\in I$ satisfy $qp=rp$. Then $q=r$.
\item[$(10)$] If $u\in J(I)$, then there is a unique $v\in J(I^\prime)$ such that $uv=v$ and $vu=u$, denoted $u\sim v$.
\end{enumerate}
\end{lem}

\begin{note}
(1) of Lemma~\ref{lem6.3} implies that each minimal left ideal $I$ of $E(T,X)$ is a closed subset of $E(T,X)$, since $I=Ip=E(T,X)p$ for any $p\in I$ and $E(T,X)$ is compact and $q\mapsto qp$ is continuous. Thus $I$ is a minimal left ideal of $E(T,X)$ iff it is a minimal subset of the induced semiflow
$T\times E(T,X)\rightarrow E(T,X)$, $(t,p)\mapsto t\circ p$.
Here we will mainly need (2), (4), (7), (9), and (10) of Lemma~\ref{lem6.3} in our later arguments.
\end{note}

\begin{proof}[\textbf{Proof of Lemma~\ref{lem1.11}}]
Let $I$ be the only minimal left ideal in $E(X)$ and $(x,y), (y,z)\in P(X)$. Then $p(x)=p(y)$ and $p(y)=p(z)$ for all $p\in I$. So $(x,z)\in P(X)$.

For the ``only if'' part, let $P(X)$ be transitive, $I,K$ minimal left ideals in $E(X)$, and $u\in J(I)$ and $v\in J(K)$ with $uv=v, vu=u$ by (10) of Lemma~\ref{lem6.3}. Let $x\in X$. Then $(x,ux)\in P(X)$ and $(x,vx)\in P(X)$ implies $(ux,vx)\in P(X)$. But $v(ux,vx)=(ux,vx)$ implies that $(ux,vx)$ is an a.p. point of $(T,X\times X)$. Hence $ux=vx$ by Lemma~\ref{lem2.7} and $u=v$ so that $I=K$.
\end{proof}

Following $\S\ref{sec1.1.3}$\,(\textbf{h}), an $x\in X$ is a distal point of $(T,X)$ if and only if it is proximal only to itself in $\textrm{cls}_XTx$.

\begin{lem}[{cf.~Veech~\cite{V70} for $T$ in groups}]\label{lem6.4}
Let $(T,X)$ be a semiflow and $x\in X$. Then $x$ is a distal point of $(T,X)$ iff $x=u(x)$ for all $u\in J(E(X))$.
\end{lem}

\begin{proof}
Let $x$ be a distal point of $(T,X)$ and $u\in J(E(X))$. Since $(x,u(x))\in P(X)$ is a.p. (by Theorem~\ref{thm5.1}), hence $x=u(x)$.

Conversely, assume $x=u(x)$ for all $u\in J(E(X))$ and let $y\in\textrm{cls}_XTx$ such that $(x,y)\in P(X)$. There is a minimal left ideal $I$ such that $p(x)=p(y)\ \forall p\in I$. Since $x\in Ix$ and so $y\in Ix$, it follows that $Iy=Ix$ so $y\in Iy$. Then there is $u\in J(I)$ with $y=u(y)=u(x)=x$. This shows that $x$ is a distal point of $(T,X)$.
\end{proof}

As a consequence of the statements of Lemmas~\ref{lem6.3} and \ref{lem6.4}, the following (2) of Theorem~\ref{thm6.5} is more or less motivated by \cite[Proposition~2.1]{V70}, which is useful for proving the ``if'' part of Lemma~\ref{lem1.8}.

\begin{thm}\label{thm6.5}
Let $(T,X)$ be a semiflow with Ellis' semigroup $E(X)$ and $x\in X$. Then:
\begin{enumerate}
\item[$(1)$] For every minimal left idea $I$ in $E(X)$, $pI\cap J(I)\not=\emptyset$ for all $p\in E(X)$.
\item[$(2)$] $x$ is a distal point of $(T,X)$ iff $x\in p(X)$ for all $p\in E(X)$ iff $x\in u(X)$ for all $u\in J(E(X))$.
\end{enumerate}
\end{thm}

\begin{proof}
(1) Let $p\in E(X)$ and $I$ a minimal left ideal in $E(X)$. Then $pI\subseteq I$ and further by (7) and (4) of Lemma~\ref{lem6.3} there are $q\in I, \delta\in I$, and $v\in J(I)$ such that $pq\delta=v$. Since $q\delta\in I$, hence $pI\cap J(I)\not=\emptyset$.

(2) Assume $x$ is a distal point; then $x=v(x)$ for every $v\in J(E(X))$ by Lemma~\ref{lem6.4}. Thus $x\in p(X)$ for all $p\in E(X)$ by (1). Conversely, suppose that $x\in v(X)$ for all $v\in J(E(X))$. Let $u\in J(E(X))$ be arbitrary. Then there exists $y\in X$ such that $u(y)=x$. So $u(x)=u^2(y)=u(y)=x$. Thus by Lemma~\ref{lem6.4}, $x$ is a distal point of $(T,X)$.

The proof of Theorem~\ref{thm6.5} is thus complete.
\end{proof}

\begin{proof}[\textbf{Proof of the ``if'' part of Lemma~\ref{lem1.8}}]
Let the set of distal points of $(T,X)$ be dense in $X$ and $t\in T$. Since every distal point belongs to $tX$ by Theorem~\ref{thm6.5} and $tX$ is a closed set, hence $tX=X$. The proof of Lemma~\ref{lem1.8} is thus complete.
\end{proof}

Recall that if $E(T,X)\subset C(X,X)$ then $(T,X)$ is called ``weakly equicontinuous'' by $\S\ref{sec1.1.5}\,(\textbf{p})$. So the following says that the proximal relation is an equivalence relation for every weakly equicontinuous semiflow with abelian phase semigroup.

\begin{thm}\label{thm6.6}
Let $(T,X)$ be a semiflow with $T$ an abelian semigroup and $J(E(X))\subset C(X,X)$. Then the following two statements hold:
\begin{enumerate}
\item[$(1)$] There exists a unique minimal left ideal $I$ in $E(X)$ and moreover $I$ contains a unique idempotent $u$. Hence $P(X)$ is an equivalence relation on $X$.
\item[$(2)$] If $x\in X$ is an a.p. point, then it is a distal point. Hence if there is a dense set of a.p. points, then $(T,X)$ is distal.
\end{enumerate}

\begin{note}
In fact, if $E(X)\subset C(X,X)$ and there is a dense set of a.p. points, then $(T,X)$ is not only distal but also equicontinuous by (c) of Lemma~\ref{lem6.7} and Ellis's joint continuity theorem (cf.~Theorem~\ref{thm9.8} in $\S\ref{sec9}$).
\end{note}
\end{thm}

\begin{proof}
(1) Let $I_1$ and $I_2$ be two minimal left ideals in $E(X)$. Then by (10) of Lemma~\ref{lem6.3}, it follows that there are idempotents $u\in I_1$ and $v\in I_2$ such that $uv=v$. Thus there is a net $\{t_n\}$ in $T$ with $t_n\to v$ in $E(X)$ such that
$I_1\ni\lim t_nu=\lim ut_n=uv=v$.
Then $I_1\cap I_2\not=\emptyset$ and thus $I_1=I_2$. This shows that there is only one minimal left ideal $I$ in $E(X)$. Thus $P(X)$ is an equivalence relation on $X$ by Lemma~\ref{lem1.11}.

Let $u,v\in J(I)$. Then by (2) of Lemma~\ref{lem6.3}, $uv=u=uu$. By the above argument, we can see $vu=uv=uu$ and so by (9) of Lemma~\ref{lem6.3} $u=v$.

(2) Let $x$ be an a.p. point of $(T,X)$. Then by $x\in Ix$ where $I$ is as in (1), $x=ux$. Thus by Lemma~\ref{lem6.4}, $x$ is a distal point of $(T,X)$. Because $u\in C(X,X)$ by weak equicontinuity, if the set of a.p. points of $(T,X)$ is dense in $X$ then $u=\textit{id}_X$. Thus $(T,X)$ is pointwise distal by Lemma~\ref{lem6.4}, and so it is distal.
\end{proof}

Ellis' classical characterization of distality states that $(T,X)$ is a distal flow if and only if $E(T,X)$ is a group (cf.~\cite[Theorem~1]{E58}, \cite[Proposition~5.3]{E69} and \cite[Theorem~5.6]{Aus}). Another important consequence of Lemma~\ref{lem6.3} is the following semiflow version of Ellis' characterization, which has already played an important role in \cite{DX}.

\begin{lem}\label{lem6.7}
Let $(T,X)$ be a semiflow, where $T$ is a discrete semigroup $($but not necessarily $e\in T$$)$. Then the following statements are pairwise equivalent:
\begin{enumerate}
\item[$(\mathrm{a})$] $(T,X)$ is a distal semiflow.
\item[$(\mathrm{b})$] $E(X)$ is a minimal left ideal in itself with $\textit{id}_X\in E(X)$.
\item[$(\mathrm{c})$] E(X) is a group with the neutral element $\textit{id}_X$.
\end{enumerate}
Here $\textit{id}_X$ denotes the identity map of $X$.
\end{lem}

\begin{notes}
\begin{enumerate}
\item Condition (b) implies that $(T,X)$ is pointwise minimal, because $E(X)x$ is a minimal set of $(T,X)$ and $x\in E(X)x$ for $\textit{id}_X\in E(X)$.
\item Weaker than `rigidity' and `uniform rigidity' of a cascade (cf.~\cite{GM, AAB}), we say a semiflow $(T,X)$ is \textit{pseudo-rigid} if one can find a net $\{t_n\,|\,n\in D\}$ in $T$ with $t_n\to\infty$ in the sense of (2) of Lemma~\ref{lem3.10B} such that $t_nx\to x$ for all $x\in X$, i.e., $t_n\to \textit{id}_X$ in $E(X)$; moreover, if $t_n\to\textit{id}_X$ uniformly (i.e., given $\varepsilon\in\mathscr{U}_X$ there exists an $n_0\in D$ such that $(t_nx,x)\in\varepsilon\ \forall x\in X$ for all $n\ge n_0$), then $(T,X)$ is called \textit{uniformly pseudo-rigid}.
    Thus, if $(f,X)$ is distal it is pseudo-rigid. In fact, we can obtain the following more general result:
    \begin{quote}
    \textit{If $(T,X)$ is a distal effective semiflow, then it is either pseudo-rigid and so pointwise $\ell$-recurrent or $(t,X)$ is uniformly pseudo-rigid for each $t\in T$ with $t\not=e$.}
    \end{quote}
    Indeed, let $e\not=t\in T$. Since $(t,X)$ is distal and effective, by Lemma~\ref{lem6.7} there is a net $\{t_n\,|\,n\in D\}$ in $T$ with $e\not=t_n\to id_X$ in $E(X)$. If $t_n\to\infty$ in the sense of (2) of Lemma~\ref{lem3.10B}, then $(T,X)$ is pseudo-rigid. Now assume $t_n\not\to\infty$ in the sense of (2) of Lemma~\ref{lem3.10B}. Then by passing to a subnet of $\{t_n\}$ if necessary, there exists a $K\in\mathfrak{N}_{\textit{cpt},e}$ such that $t_n\in K$ for all $n\in D$. Thus $t_n\to e$ in the topological semigroup $T$. Since $K,X$ both are compact and $K\times X\rightarrow X$ is jointly continuous, hence $t_nx\to x$ uniformly for $x\in X$.
\end{enumerate}
\end{notes}

\begin{proof}
Condition $(\mathrm{a})\Rightarrow(\mathrm{b})$: Let $I$ be a minimal left ideal in $E(X)$ and $u\in J(I)$. Then by Lemma~\ref{lem6.4}, $x=u(x)$ for all $x\in X$. Thus $u=\textit{id}_X$ and further $E(X)$ is a minimal left ideal with the unique idempotent $\textit{id}_X\in E(X)$.

Condition $(\mathrm{b})\Rightarrow(\mathrm{c})$: $E(X)$ is a group by (4) of Lemma~\ref{lem6.3} with $u=\textit{id}_X\in E(X)$.

Condition $(\mathrm{c})\Rightarrow(\mathrm{a})$: Suppose $(x,y)\in P(X)$. Then $p(x)=p(y)$ for some $p\in E(X)$ so $x=y$ by $p^{-1}p=\textit{id}_X$, since $E(X)$ is a group with $e=\textit{id}_X\in E(X)$.
Thus $(\mathrm{a})$ holds.

The proof of Lemma~\ref{lem6.7} is thus completed.
\end{proof}

The most important part of Lemma~\ref{lem6.7} is ``(a) $\Rightarrow$ (c)'' which implies (2) of Theorem~\ref{thm1.13}. Now we will present an independent direct proof for it without using Lemma~\ref{lem6.3}.

\begin{proof}[\textbf{Another proof of ``$\mathrm{(a)}\Rightarrow\mathrm{(c)}$'' of Lemma~\ref{lem6.7}}]
Note that `distal' implies `pointwise a.p.' (by Theorem~\ref{thm5.1}). Since $(T,X^X)$ is distal, $E:=E(T,X)$ which is the orbit closure of $\textit{id}_X$ is minimal. Now for every $p\in E$, since $Ep$ is closed $T$-invariant, $Ep=E$. This easily follows that every $p\in E$ has an inverse.
\end{proof}

This algebraic characterization of distality is very useful. Notice that if $e\not\in T$ and $(T,X)$ is distal, then either $\textit{id}_X$ is a pointwise limit point of $T$ in $E(X)$ or $tx=x\ \forall x\in X$ for some $t\in T$ by (b) of Lemma~\ref{lem6.7}.

Now by Lemma~\ref{lem6.7} or by the fact that every distal map is pointwise recurrent, we can obtain the following, (2) of which is just (2) of Theorem~\ref{thm1.13}.

\begin{cor}\label{cor6.8}
Let $(T,X)$ be a semiflow with phase semigroup $T$. Then:
\begin{enumerate}
\item[$(1)$] If $(T,X)$ is distal, then it is invertible and admits an invariant Borel probability measure.
\item[$(2)$] If $(T,X)$ is equicontinuous, then it is distal if and only if it is surjective.
\item[$(3)$] If $(T,X)$ is point-distal surjective with $E(X)\subset C(X,X)$, then $(T,X)$ is equicontinuous.
\end{enumerate}
\end{cor}

\begin{proof}
(1) Let $(T,X)$ be distal; then by Lemma~\ref{lem6.7}, $E(X)$ is a group with $e=\textit{id}_X$. Let $\textrm{Homeo}\,(X)$ be the group of all self-homeomorphisms of $X$. Then $T\subset\textrm{Homeo}\,(X)$ and $E(X)=\textrm{cls}_{X^X}\langle T\rangle$. Thus by Furstenberg's structure theorem of distal minimal flows~\cite{F63} (cf.~Theorem~\ref{thm3.14A} below), it follows that $(\langle T\rangle,X)$ and so $(T,X)$ admit invariant Borel probability measures.

(2) This follows easily from Lemma~\ref{lem6.7} and Theorem~\ref{thm2.1}.

(3) Let $x$ be a distal point with $\textrm{cls}_XTx=X$. By Lemma~\ref{lem1.8}, each point of $Tx$ is distal for $(T,X)$. Given any $u\in J(E(X))$, $sx=usx$ for all $s\in T$. Since $u\in C(X,X)$ and $Tx$ is dense, so $u=\textit{id}_X$. Thus $E(X)$ is a group by (4) of Lemma~\ref{lem6.3} and so $(T,X)$ is minimal distal by Lemma~\ref{lem6.7}. This implies that $(E(X),X)$ is an equicontinuous flow. Thus $(T,X)$ is equicontinuous.

The proof of Corollary~\ref{cor6.8} is thus completed.
\end{proof}

It is interesting that a distal map is always surjective, while an equicontinuous map is not by Examples~\ref{ex1.7} and~\ref{ex1.12} in $\S\ref{sec1}$.
Also this indicates that distal is the more natural concept. However, under locally (weakly) almost periodic condition, the equicontinuous is equivalent to the distal in flows (cf.~\cite{E69,AM}).

\begin{lem}\label{lem6.9}
If $(T,X)$ is minimal invertible such that for each $t\in T$, $(t^{-1},X)$ is rigid, that is, $\textit{id}_X\in \mathrm{cls}_{X^X}^{}\{t^{-n}\,|\,n=1,2,\dotsc\}$, then $(X,T)$ is minimal.
\end{lem}

\begin{proof}
Let $X_0$ be a minimal set of $(X,T)$ by Zorn's lemma, and let $t\in T, t\not=e$ be any given.
Then there exists a net $\{n_k\}$ in $\mathbb{N}$ with $t^{-n_k}\to \textit{id}_X$ in $X^X$ under the pointwise topology. Thus for every point $x_0\in X_0$,
$t^{-n_k}x_0\to x_0$ and so $t(t^{-n_k}x_0)=t^{-n_k+1}x_0\to tx_0$.
Since $-n_k+1\le 0$, then $t^{-n_k+1}x_0\in X_0$ and so $tx_0\in X_0$.
Hence $TX_0\subseteq X_0$ and then $X_0=X$ for $(T,X)$ is minimal.
\end{proof}

As another result of Lemma~\ref{lem6.7}, we can then obtain using algebraic approaches the following simple observation for distal semiflows, which are (3) of Theorem~\ref{thm1.13} and 2 and 3 of Reflection principle~I.

\begin{prop}\label{prop6.10}
Let $(T,X)$ be a semiflow with phase semigroup $T$. Then:
\begin{enumerate}
\item[$(1)$] If $(T,X)$ is distal, then so is $(\langle T\rangle,X)$.
\item[$(2)$] If $(T,X)$ is minimal distal, then $(X,T)$ and $(\langle T\rangle,X)$ both are minimal distal.
\end{enumerate}
\end{prop}

\begin{proof}
(1) Since $E(T,X)$ is a group with $e=\textit{id}_X$ by Lemma~\ref{lem6.7}, then $(T,X)$ is invertible and $T^{-1}\subseteq E(X)$. So $E(X,T)\subseteq E(X)$. If $p(x)=p(y)$ for some $p\in E(X,T)$ then by distality of $(T,X)$ we see $x=y$. Therefore, $(X,T)$ is distal. Moreover, since $\langle T\rangle\subseteq E(X)$, thus $(\langle T\rangle,X)$ is distal by Lemma~\ref{lem6.7} again.

(2) By (1), we only need prove the minimality of $(X,T)$. To this end, let $t\in T$. Since the distal cascade $(t^{-1},X)$ induces a distal semiflow $f\colon(n,x)\mapsto t^{-n}x$ of $\mathbb{N}\times X$ to $X$ where $\mathbb{N}$ is discrete additive, then by Lemma~\ref{lem6.7} the Ellis semigroup of $(f,\mathbb{N},X)$ contains $\textit{id}_X$; i.e., $(t^{-1},X)$ is rigid. Then (2) follows from Lemma~\ref{lem6.9}.

The proof of Proposition~\ref{prop6.10} is therefore completed.
\end{proof}

We note that using Ellis' semigroup (Lemma~\ref{lem6.7}) we have easily concluded Proposition~\ref{prop6.10}. However, if we make no use of this and the $\beta$-compactification of $T$, based on Theorem~\ref{thm5.1} in $\S\ref{sec5}$ and using only topological approaches we can prove it as follows.

\begin{proof}[\textbf{Proof~II of Proposition~\ref{prop6.10}}]
Let $(T,X)$ be distal. We will divide our non-enveloping semigroup proof into relatively independent four steps.

\begin{step}\label{step1}
Every point of $X$ is a.p. for $(T,X)$. Moreover, $(T,X)$ is invertible.
\end{step}

\begin{proof}
The first part of Step~\ref{step1} follows at once from Theorem~\ref{thm5.1}. Now given $t\in T$, since $(t,X)$ is pointwise a.p., then $tX=X$. This shows that $(T,X)$ is invertible.
\end{proof}

Although $(T,X)$ is pointwise a.p. by Step~\ref{step1}, yet because $T$ need not be right-syndetic in $\langle T\rangle$ and $(T,X)$ need not be minimal the following Step~\ref{step2} is non-trivial.

\begin{step}\label{step2}
{\it $(\langle T\rangle,X)$ is pointwise a.p. (cf.~Definition~\ref{sn2.4}).}
\end{step}

\begin{proof}
Let $x\in X$ be any given and write $Y_x=\textrm{cls}_X{Tx}$. Clearly by Step~\ref{step1}, $(T,Y_x)$ is minimal distal so that $\textrm{cls}_X{Ty}=Y_x$ for all $y\in Y_x$. Given $y\in Y_x$ and $t\in T$, since $y$ is a (forwardly) minimal point for $(\pi_t,Y_x)$ by Step~\ref{step1}, there is a net $\{n_k\}$ in $\mathbb{N}$ with $t^{n_k}y\to y$. So $t^{n_k-1}y\to t^{-1}y\in Y_x$, for $t^{n_k-1}y\in Y_x$ and $Y_x$ is closed. This shows $Y_xT\subseteq Y_x$.
Thus $Y_x=\textrm{cls}_XTy\subseteq\textrm{cls}_X{\langle T\rangle y}\subseteq Y_x$ for all $y\in Y_x$. This shows that each $y\in Y_x$ and so $x$ are a.p. for $(\langle T\rangle,X)$.
\end{proof}

\begin{step}\label{step3}
{\it $(T,X\times X)$ is distal and so $(\langle T\rangle,X\times X)$ is pointwise a.p..}
\end{step}

\begin{proof}
It follows easily from definition that $(T,X\times X)$ is distal. Then $(\langle T\rangle,X\times X)$ is pointwise a.p. by Steps~\ref{step1} and \ref{step2}.
\end{proof}

\begin{step}\label{step4}
\textit{Let $(T,Z)$ be a semiflow with any phase semigroup $T$. If $(T,Z\times Z)$ is pointwise a.p., then $(T,Z)$ is distal.}
\end{step}

\begin{proof}
This follows at once from Lemma~\ref{lem2.7}.
\end{proof}

Now, since $(\langle T\rangle,X\times X)$ is pointwise a.p. by Step~\ref{step3}, $(\langle T\rangle,X)$, which is minimal if so is $(T,X)$, is distal by Step~\ref{step4}. Thus $(X,T)$ is distal.

Next, assume $(T,X)$ is minimal distal. Then $(t^{-1},X)$ is pointwise a.p. (forwardly) and so every negatively-invariant closed subset of $X$ is also $\pi$-invariant. This implies the minimality of $(X,T)$.
The proof II of Proposition~\ref{prop6.10} is therefore complete.
\end{proof}

We will continue to consider the minimality of the reflection $(X,T)$ under much more weaker conditions in $\S\ref{sec6.3}$. Moreover, for an amenable phase semigroup, we will show in $\S\ref{sec6.3}$ that $(T,X)$ is distal at some point $x\in X$ if and only if so is $(X,T)$ at the same point $x$ (see Corollary~\ref{cor3.27A}).

The following result is originally due to Ellis~\cite[Theorem~3]{E57} (also see \cite[Theorem~3.3]{Aus}) in the case that $(T,X)$ is a flow.

\begin{cor}\label{cor6.11}
Let $(T,X)$ be a surjective semiflow. Then $(T,X)$ is equicontinuous if and only if $E(X)$ is a group of self-homeomorphisms of $X$.
\end{cor}

\begin{proof}
First from equicontinuity, all $p$ in $E(X)$ are continuous. Then the necessity follows at once from Theorem~\ref{thm2.1} and Lemma~\ref{lem6.7}.
Conversely, if $E(X)$ is a group of homeomorphisms of $X$, then by Ellis' joint continuity theorem (cf.~\cite[Theorem~4.3]{Aus} and also Theorem~\ref{thm9.8} in Appendix), it follows that $E(X)$ and so $T$ acts equicontinuously on $X$.
This proves Corollary~\ref{cor6.11}.
\end{proof}

In Corollary~\ref{cor6.11}, it is essential that $T$ consists of surjections, and not merely a semigroup of continuous maps. Corollary~\ref{cor6.11} may follows from (3) of Corollary~\ref{cor6.8}.

Given any semigroup $T$ of bijections of $X$, $T\cup T^{-1}$ is not necessarily equal to the group $\langle T\rangle$. In addition, if $T$ acts equicontinuously on $X$, then so does $T\cup T^{-1}$ by Proposition~\ref{prop6.1}. However, since $T$ need not be abelian, the equicontinuity of $\langle T\rangle$ cannot be trivially obtained.

Nevertheless Theorem~\ref{thm2.1} together with Lemma~\ref{lem6.7} implies the following important fact, which is just (4) of Theorem~\ref{thm1.13}.

\begin{thm}\label{thm6.12}
Let $G$ be a semigroup of self-homeomorphisms of $X$. Then $G$ is equicontinuous on $X$ if and only if so is $\langle G\rangle$.
\end{thm}

\begin{proof}
It suffices to show the ``only if'' part. Let $G$ is equicontinuous on $X$. By Corollary~\ref{cor6.11}, $E(G,X)$ is a group consist of self-homeomorphisms of $X$. Further $E(G,X)$ acts equicontinuously on $X$. Since $\langle G\rangle\subseteq E(G,X)$, thus $\langle G\rangle$ is equicontonuous on $X$.
\end{proof}

Motivated by Proof (III) of Theorem~\ref{thm2.1}, we can present another self-contained topological proof of Theorem~\ref{thm6.12} without using Lemmas~\ref{lem6.3} and \ref{lem6.7}.

\begin{proof}[\textbf{Proof~II of Theorem~\ref{thm6.12}}]
We only need to show the ``only if'' part; and then assume $G$ is equicontinuous on $X$.
By $C_{\textrm{cpt-op}}(X,X)$ we denote the space $C(X,X)$ of all continuous self-maps of $X$ equipped with the compact-open topology, and let $E$ be the closure of $G$ in $C_{\textrm{cpt-op}}(X,X)$. Then by Ascoli's theorem $E$ is compact and moreover, each $p\in E$ is a surjection of $X$. We will show that $E$ is a group.

First we note that $(f,g)\mapsto fg:=f\circ g$ of $C_{\textrm{cpt-op}}(X,X)\times C_{\textrm{cpt-op}}(X,X)$ to $C_{\textrm{cpt-op}}(X,X)$ is separately continuous. This implies that $EE\subseteq E$ and thus $E$ is a compact semi-topological semigroup. Since each $p\in E$ is surjective, $E$ has the unique idempotent $\textit{id}_X$.

Given any $p\in E$, since $Ep$ is a closed subsemigroup of $E$ so that it contains an idempotent, hence $\textit{id}_X\in Ep$ and so there is some $q\in E$ such that $qp=\textit{id}_X$. This shows that $E$ is a group of self-homeomorphisms of $X$.
Finally, since $G$ and then $E$ acts equicontinuously on $X$, so does $\langle G\rangle$ because of $\langle G\rangle\subseteq E$.
\end{proof}

Finally we notice that whereas Proposition~\ref{prop6.1} may be a consequence of Theorem~\ref{thm6.12}, its direct proof is of independent interest.
\subsection{Furstenberg's structure theorem of distal minimal semiflows}\label{sec3.2A}
Let $T$ be any discrete semigroup with neutral element $e$ and let $\theta$ be some ordinal. Following Furstenberg~\cite{F63} we introduce a basic notion.

\begin{sn}\label{sn6.13}
A \textit{projective system} of minimal semiflows with phase semigroup $T$ is a collection of minimal semiflows $(T,X_\lambda)$ with compact $T_2$ phase spaces $X_\lambda$ indexed by ordinal numbers $\lambda\le\theta$, and a family of epimorphisms, $\pi_\nu^\lambda\colon(T,X_\lambda)\rightarrow(T,X_\nu)$, for $0\le\nu<\lambda\le\theta$, satisfying:
\begin{enumerate}
\item[(1)] If $0\le\nu<\lambda<\eta\le\theta$, then $\pi_\nu^\eta=\pi_\nu^\lambda\circ\pi_\lambda^\eta$.
\item[(2)] If $\mu\le\theta$ is a limit ordinal, then $X_\mu$ is the minimal subset of the Cartesian product semiflow $\left(T,\prod_{\lambda<\mu}X_\lambda\right)$ consisting of all $x=(x_\lambda)_{\lambda<\mu}\in\prod_{\lambda<\mu}X_\lambda$ with $x_\nu=\pi_\nu^\lambda(x_\lambda)$ for all $\nu<\lambda<\mu$ and then for $\lambda<\mu$,
    $\pi_\lambda^\mu\colon X_\mu\rightarrow X_\lambda$ is just the projection map. In this case, we say that $(T,X_\mu)$ is the \textit{inverse limit} of the directed family of minimal semiflows $\{(T,X_\lambda)\,|\,\lambda<\mu\}$.
\end{enumerate}
\end{sn}

Let $(T,X)$ be an invertible semiflow and let $G=\langle T\rangle$ be the discrete group of self homeomorphisms of $X$ generated by $T$ associated to $(T,X)$.

If $(T,Y)$ is another invertible semiflow and if $\pi\colon(T,X)\rightarrow(T,Y)$ is an epimorphism, then there is a natural extension
\begin{equation*}
\pi\colon(G,X)\rightarrow(G,Y)
\end{equation*}
where for all $g=\tau_1\tau_2\dotsm\tau_n\in G, \tau_i\in T\cup T^{-1}$,
\begin{gather*}
gx=\tau_1\tau_2\dotsm\tau_nx\ \forall x\in X\quad \textrm{and}\quad gy=\tau_1\tau_2\dotsm\tau_ny\ \forall y\in Y.
\end{gather*}
Since $t=t_1t_2$ relative to $(T,X)$ implies that $t=t_1t_2$ relative to $(T,Y)$, thus $\pi\colon(G,X)\rightarrow(G,Y)$ is well defined. However, it should be noticed that $G$ is defined by $(T,X)$, not by the factor $(T,Y)$.

Recall that $\pi\colon(T,X)\rightarrow(T,Y)$ is a relatively equicontinuous extension iff for all $\varepsilon\in\mathscr{U}_X$ there is $\delta\in\mathscr{U}_X$ such that whenever $(x,x^\prime)\in\delta$ with $\pi(x)=\pi(x^\prime)$, then $(tx,tx^\prime)\in\varepsilon$ for all $t\in T$ (cf.~Definition~\ref{sn1.20}).

Now based on Definitions~\ref{sn1.20} and \ref{sn6.13}, we are ready to state the Furstenberg structure theorem for minimal distal semiflows as follows:

\begin{thm}[Furstenberg's structure theorem]\label{thm3.14A}
Let $(T,X)$ and $(T,Y)$ be distal minimal semiflows and let $\pi\colon (T,X)\rightarrow(T,Y)$ be an epimorphism. Then there is a projective system of minimal semiflows $\{(T,X_\lambda)\,|\,\lambda\le\theta\}$, for some ordinal $\theta\ge1$, with $X_\theta=X$, $X_0=Y$ such that if $0\le\lambda<\theta$, then $\pi_\lambda^{\lambda+1}\colon(T,X_{\lambda+1})\rightarrow(T,X_\lambda)$ is a relatively equicontinuous extension.
\end{thm}

\begin{proof}
According to Proposition~\ref{prop6.10}, $(\langle T\rangle, X)$ and $(\langle T\rangle, Y)$ are distal minimal flows. We now write $G=\langle T\rangle$ associated to $(T,X)$. Then $\pi\colon (G,X)\rightarrow(G,Y)$ is an epimorphism of distal minimal flows with phase group $G$.

Thus by Furstenberg's structure theorem of distal minimal flows (cf.~\cite{F63} or \cite[Theorem~7.1]{Aus}), it follows that there is a projective system of minimal flows $\{(G,X_\lambda)\,|\,\lambda\le\theta\}$ with $X_\theta=X$, $X_0=Y$ such that if $0\le\lambda<\theta$, then $\pi_\lambda^{\lambda+1}\colon(G,X_{\lambda+1})\rightarrow(G,X_\lambda)$ is a relatively equicontinuous extension. In order to show that
\begin{equation*}
(T,X)=(T,X_\theta)\xrightarrow{}\dotsm\xrightarrow{}(T,X_{\lambda+1})\xrightarrow{\pi_\lambda^{\lambda+1}}(T,X_\lambda)\xrightarrow{}\dotsm \rightarrow(T,X_1)\xrightarrow{\pi_0^1}(T,X_0)=(T,Y)
\end{equation*}
is actually the desired projective system of minimal semiflows with phase semigroup $T$, it is sufficient to prove that $(T,X_\lambda)$, $0<\lambda<\theta$, is a minimal semiflow.

Indeed, given $0<\lambda<\theta$, since
$\pi_\lambda^\theta\colon(G,X_\theta)\rightarrow(G,X_\lambda)$ is an epimorphism, it follows that $\pi_\lambda^\theta\colon(T,X)\rightarrow(T,X_\lambda)$ is also an epimorphism so $(T,X_\lambda)$ is a minimal semiflow for all $\lambda<\theta$. This proves Theorem~\ref{thm3.14A}.
\end{proof}

\begin{cor}\label{cor3.15A}
If $(T,X)$ is a minimal distal semiflow, then it has a non-trivial equicontinuous surjective factor, i.e., there is an epimorphism $\pi\colon(T,X)\rightarrow(T,Y)$ such that $(T,Y)$ is a non-trivial equicontinuous surjective semiflow.
\end{cor}

Of course if $T$ is amenable and $X$ is metric, then the statement of Corollary~\ref{cor3.15A} can follow from Corollary~\ref{cor5.6}.
\subsection{Minimality of reflections}\label{sec6.3}
If $(T,X)$ is a flow, $x\in X$, and $U$ a neighborhood of $x$, then $(N_T(x,U))^{-1}=N_{(X,T)}(x,U)$ is left-syndetic in $T$ by $T=T^{-1}$. So if $x$ is a.p. for $(T,X)$, then it is also a.p. for the reflection $(X,T)$. But if $(T,X)$ is only an invertible semiflow, then $(N_T(x,U))^{-1}$ need not be a left-syndetic subset of $T$ so that $x$ need not be a.p. for $(X,T)$; see 1. of Examples~\ref{ex1.4}. However, we will be concerned with questions or reflection principles as follows:
\begin{quote}
{\it Let $(T,X)$ be invertible. If $(T,X)$ is minimal, is $(X,T)$ minimal too? If $x$ is an a.p. point of $(T,X)$, is it an a.p. point of $(X,T)$?}
\end{quote}
In this subsection, we shall show that this question is in the affirmative if $T$ is an amenable semigroup (cf.~$\S\ref{sec1.1.4}$\,(\textbf{k})) or if $T$ is a right \textit{C}-semigroup (cf.~Definition~\ref{sn3.8B}).

\subsubsection{Abelian phase semigroup}
First of all, whereas the following observation is simple, it might be useful for our later proof of Proposition~\ref{prop6.17}.

\begin{lem}\label{lem6.16}
Let $f\colon X\rightarrow X$ be a homeomorphism and let $x\in X$ be a (forwardly) recurrent point for $f ^{-1}$. Then $f(x)$ belongs to $\mathrm{cls}_X{\{f^{-n}(x)\,|\,n=0,1,2,\dotsc\}}$.
\end{lem}

Motivated by Proposition~\ref{prop6.10} stated in $\S\ref{sec6.1}$, we can easily obtain the following result using Lemma~\ref{lem6.16}.

\begin{prop}\label{prop6.17}
If $(T,X)$ is minimal invertible with $T$ an abelian semigroup, then $(X,T)$ is minimal.
\end{prop}

\begin{proof}
Let $X_0$ be a minimal set of $(X,T)$ and $t\in T$. Let $x_0\in X_0$ be a minimal point for $(t^{-1},X_0)$ with phase semigroup $\mathbb{Z}_+$. As $x_0$ is recurrent for $t^{-1}$, it follows from Lemma~\ref{lem6.16} that $tx_0\in X_0$. Then by commutativity of $T$,
$tT^{-1}x_0=T^{-1}tx_0\subseteq X_0$ so $tX_0\subseteq X_0$. Whence $X_0$ is invariant for $(T,X)$.
This proves Proposition~\ref{prop6.17}.
\end{proof}

In fact, Proposition~\ref{prop6.17} can be differently proved as follows:

\begin{proof}[\textbf{Proof~II of Proposition~\ref{prop6.17}}]
Let $X_0$ be a minimal set of $(X,T)$. Then if $t\in T$, then $X_0\cap tX_0\not=\emptyset$ (since $tt^{-1}x=x$ by Corollary~\ref{cor3.3B}). But since $T$ is abelian, then $tX_0$ is minimal for $(X,T)$ so $tX_0=X_0$. This shows that $X_0=X$.
\end{proof}

The lighting point of Proposition~\ref{prop6.17} is that $T$ is not necessarily a right-syndetic subsemigroup of the group $\langle T\rangle$ of homeomorphisms of $X$ generated by $T$.

Let $N$ be a non-empty closed invariant set of $(X,T)$ and $t\in T$; then for every $x\in N$, its the $\alpha$-limit points set $\alpha_t(x)$ under $(t,X)$ is such that $t^n\alpha_t(x)\subseteq N$ for all $n\in\mathbb{Z}_+$. More generally, we can obtain the following.

\begin{cor}\label{cor3.18A}
Let $(T,X)$ be an invertible semiflow and $N_-$ an invariant closed non-empty subset of its reflection $(X,T)$. Then for every abelian subsemigroup $S\subseteq T$, there exists some point $x\in N_-$ such that $Sx\in N_-$.
\end{cor}

\begin{proof}
Let $S$ be an abelian subsemigroup of $T$. Since $N_-$ is invariant for $(X,T)$, it is invariant for $(X,S)$. Then there is a minimal set $N_0$ for $(X,S)$ with $N_0\subseteq N_-$. By Proposition~\ref{prop6.17}, $N_0$ is a minimal set for $(S,X)$, so $SN_0=N_0\subseteq N_-$. This proves Corollary~\ref{cor3.18A}.
\end{proof}

\subsubsection{Amenable phase semigroup}
Recall that as in $\S\ref{sec1.1.4}$\,(\textbf{k}) a semigroup $T$ is said to be amenable iff every semiflow on a compact $T_2$-space with the phase semigroup $T$ admits an invariant Borel probability measure.

Since each abelian semigroup is an amenable semigroup, then the following theorem covers Proposition~\ref{prop6.17} by different ergodic approaches.

\begin{thm}\label{thm6.19}
Let $(T,X)$ be an invertible semiflow with $T$ an amenable semigroup and $x\in X$. Then $x$ is an a.p. point of $(T,X)$ if and only if $x$ is an a.p. point of $(X,T)$. Moreover, if $x$ is an a.p. point of $(T,X)$, then $\mathrm{cls}_XTx=\mathrm{cls}_XxT$.
\end{thm}

\begin{note}
If ``with $T$ an amenable semigroup'' is replaced by ``admitting an invariant Borel probability measure'', then the statement still holds.
\end{note}

\begin{proof}
Let $X_0$ be a minimal subset of $(T,X)$. Since $T$ is amenable, there is an invariant quasi-regular Borel probability measure $\mu$ for $(T,X_0)$ such that $\textrm{supp}\,(\mu)=X_0$. Then by Lemma~\ref{lem3.6B}, it follows that for each $t\in T$ is a surjection of $X_0$ and so is $t^{-1}$ and then all $t$ restricted to $X_0$ are self-homeomorphisms of $X_0$. This shows that $X_0$ is also a closed invariant subset of $(X,T)$. We will show that $X_0$ is also minimal for $(X,T)$.

To be contrary assume that $X_0$ is not minimal for $(X,T)$; then by Zorn's lemma, there exists a proper non-empty closed subset $Y$ of $X_0$ such that $(Y,T)$ is a minimal semiflow. Since $T$ is amenable, there is an invariant quasi-regular Borel probability measure $\nu$ for $(Y,T)$ such that $\textrm{supp}\,(\nu)=Y$. Then by Lemma~\ref{lem3.6B} again, it follows that for each $t\in T$, $t^{-1}\colon Y\rightarrow Y$ is surjective and so is $t^{-1}$ and then $t$ restricted to $Y$ is a self-homeomorphism of $Y$. This shows that $Y$ is also a closed invariant subset of $(T,X_0)$.
But this contradicts that $(T,X_0)$ is minimal.

By symmetry, we can show that every minimal set of $(X,T)$ is a minimal set of $(T,X)$.
The proof of Theorem~\ref{thm6.19} is therefore complete.
\end{proof}

In view of 1 of Examples~\ref{ex1.4}, the condition that $T$ is amenable is essential for the above proof of Theorem~\ref{thm6.19}. In fact, the key idea is that each $t\in T$ is surjective restricted to every minimal subset. Amenability just guarantees this condition.

Recall that Proposition~\ref{prop6.10} claims that if $(T,X)$ is distal, then so is $(X,T)$. However, from Theorem~\ref{thm6.19} we can obtain the following ``reflection principle of distality'' which asserts that if $x\in X$ is a distal point of $(T,X)$ and if the phase semigroup $T$ is amenable, then $x$ is also a distal point for $(X,T)$. So if $f\colon X\rightarrow X$ is a homeomorphism such that it is forwardly distal at a point $x$, then it is backwardly distal at $x$.

\begin{prop}\label{prop6.20}
Let $(T,X)$ be invertible with $T$ an amenable semigroup and $x\in X$.
If $x$ is a distal point of $(T,X)$, then $x$ is a distal point of $(X,T)$.
\end{prop}

\begin{notes}
\begin{enumerate}
\item If ``with $T$ an amenable semigroup'' is replaced by ``admitting an invariant Borel probability measure'', then the statement still holds.
\item Proposition~\ref{prop6.20} is in fact a corollary of Theorem~\ref{thm5.4}. But we will present an independent proof here.
\end{enumerate}
\end{notes}

\begin{proof}
Let $x\in X$ be distal for $(T,X)$. Then by Theorem~\ref{thm5.1}, $x$ is minimal for $(T,X)$. By Theorem~\ref{thm6.19}, $x$ is a minimal point for $(X,T)$. Let $Z=\textrm{cls}_XxT$ corresponding to $(X,T)$. Clearly $Z=\textrm{cls}_XTx$ by Theorem~\ref{thm6.19} again. We will show that $x$ is not proximal to any $x^\prime\not=x$ in $Z$ in the sense of $(X,T)$.
In fact, if $x^\prime$ is in $Z$, then $x^\prime$ is a minimal point of $(X,T)$. Whence $x^\prime$ is also a minimal point of $(T,X)$ by Theorem~\ref{thm6.19} once more. Then by Theorem~\ref{thm5.1}, $(x,x^\prime)$ is a minimal point for $(T,X\times X)$. This implies by Theorem~\ref{thm6.19} that $(x,x^\prime)$ is a minimal point of $(X\times X,T)$. Thus, if $x$ is proximal to $x^\prime$ for $(X,T)$, then $\textrm{cls}_{X\times X}(x,x^\prime)T$ is contained in the diagonal of $X\times X$ by minimality of $(x,x^\prime)$ under $(X\times X,T)$. Thus $x=x^\prime$.
The proof of Proposition~\ref{prop6.20} is thus complete.
\end{proof}

In preparation for our next equicontinuity consequence of Proposition~\ref{prop3.7B}, we need to recall a notion for our convenience.

\begin{sn}\label{def6.21}
A subsemigroup $S$ of $X^X$ is called a \textit{semi-topological semigroup} if under the topology $\mathfrak{p}$ of pointwise convergence, $(f,g)\mapsto f\circ g$ is separately continuous.
\end{sn}

If $E(X)$ is a topological group with the pointwise topology and if $(T,X)$ is minimal, then $(T,X)$ is equicontinuous (cf.~\cite[Proposition~5.5]{DX}). However, if $E(X)$ is only a topological semigroup but $E(X)\subset C(X,X)$, then $(T,X)$ is still equicontinuous by the following.

\begin{thm}\label{thm6.22}
Let $(T,X)$ be a semiflow with $T$ an amenable semigroup and with a dense set of a.p. points. Then, $(T,X)$ is equicontinuous surjective iff $(T,X)$ is equicontinuous iff $E(T,X)$ is a topological semigroup with $E(T,X)\subset C(X,X)$.
\end{thm}

\begin{proof}
The ``only if'' parts are obvious. Next we show the ``if'' parts of Theorem~\ref{thm6.22}. In fact, we only need prove that if $E(X)\subset C(X,X)$ is a topological semigroup, then $(T,X)$ is equicontinuous and surjective.
For this, we now assume that $E(X)\subset C(X,X)$ is a topological semigroup in the sense of the pointwise topology $\mathfrak{p}$.

Let $\mathbb{I}$ be a minimal left ideal in $E(X)$. Then we can first show that
\begin{description}
\item[(i)] Given any $p\in\mathbb{I}$, $p\mathbb{I}=\mathbb{I}$.
\end{description}

\begin{proof}
Indeed, applying Proposition~\ref{prop3.7B} with $T\times\mathbb{I}\rightarrow\mathbb{I}$, $(t,p)\mapsto tp$, it follows that for every $t\in T$, $t\mathbb{I}=\mathbb{I}$. Then if $T\ni t_n\xrightarrow{\mathfrak{p}}p$ and $q\in\mathbb{I}$,  there are $q_n\in\mathbb{I}$ with $t_nq_n=q$ and $q_n\xrightarrow{\mathfrak{p}}r$ for some $r\in\mathbb{I}$ so that $pr=q$ by the joint continuity of $(f,g)\mapsto f\circ g$. Thus, $p\mathbb{I}=\mathbb{I}$.
\end{proof}

Then by (i) there follows that
\begin{description}
\item[(ii)] $up=p$, for $u\in J(\mathbb{I})$ and $p\in\mathbb{I}$.
\end{description}
\begin{proof}
By (i), $u\mathbb{I}=\mathbb{I}$, so $p=uq$, for some $q\in\mathbb{I}$. Then $up=uuq=uq=p$.
\end{proof}

Next, if $u,v$ are idempotents in $\mathbb{I}$, then by (ii), it follows that $(u,v)u=(u,u)$. (9) of Lemma~\ref{lem6.3} implies that
\begin{description}
\item[(iii)] $u=v$.
\end{description}
Therefore, $\mathbb{I}$ has a unique idempotent $u$ in $\mathbb{I}$. Of course $u\in C(X,X)$. Since $(T,X)$ has a dense set of a.p. points, hence $ux=x$ for each $x\in X$. This implies by (4) of Lemma~\ref{lem6.3} that $\mathbb{I}=E(X)\subset C(X,X)$ is a group. Thus $(T,X)$ is equicontinuous by Corollary~\ref{cor6.11}.
\end{proof}

If $T$ is a topological group, then we can improve the statement of Theorem~\ref{thm6.22} by dropping the amenability condition by a completely different proof as follows:

\begin{itemize}
\item \textit{A flow $(T,X)$ is equicontinuous if and only if $E(T,X)$ is a topological semigroup with $E(T,X)\subset C(X,X)$.}\quad (See Theorem~\ref{thm9.14} in Appendix $\S\ref{sec9}$.)
\end{itemize}

It should be noticed that the `topological semigroup' condition is essential for the above theorem (cf.~Theorem~\ref{thm9.14} in $\S\ref{sec9}$) as shown by the following example.

\begin{exa}\label{exa6.23}
Let $X$ be the one-point compactification of the reals and define a homeomorphism $f\colon X\rightarrow X$ by $x\mapsto x+1$ for all $x\in X$. Then $\infty$ is the unique almost periodic point and $(f,X)$ with phase group $\mathbb{Z}$ is not equicontinuous; but $E(f,X)\subset C(X,X)$ consists of the powers of $f$ together with the constant map $c\colon x\mapsto\infty$. By Theorem~\ref{thm6.6}, $I=\{c\}$ is the unique minimal left ideal in $E(f,X)$.
Moreover, it is easy to see that $E(f,X)$ is a semi-topological semigroup but not a topological semigroup. Indeed, let $t_n=f^n$ and $s_n=f^{1-n}$ for any $n\ge1$. Clearly, $t_n\to c$ and $s_n\to c$ but $f=\lim t_ns_n\not=(\lim t_n)(\lim s_n)=c$.
\end{exa}

On the other hand, let us consider $(f,X)$ in Example~\ref{exa6.23} from the viewpoint of semiflow. It shows that the condition `with a dense set of a.p. points' is essential in Theorem~\ref{thm6.22}.

\begin{exa}\label{exa6.24}
Let $f\colon X\rightarrow X$ be same as in Example~\ref{exa6.23}. But here we now consider $(f,X)$ with phase semigroup $\mathbb{Z}_+$. Clearly $\infty$ is also the unique almost periodic point and $(f,X)$ is not equicontinuous; but $E(f,X)\subset C(X,X)$ consists of the powers $f^n$, $n\ge0$, together with the constant map $c\colon x\mapsto\infty$ of $X$ into itself. By Theorem~\ref{thm6.6}, $I=\{c\}$ is the unique minimal left ideal in $E(f,X)$.
Moreover, it is easy to see that $E(f,X)$ is a topological
semigroup but not a topological group.
\end{exa}

Recall that any subset $A$ of $X$ is called \textit{non-trivial}\index{set!non-trivial} if $A\not=\emptyset$ and moreover $A\not=X$. Then it is easy to verify that
\begin{itemize}
\item If $T$ is a group, then $(T,X)$ is minimal if and only if $X$ does not contain a non-trivial invariant open subset.
\end{itemize}
However, in our semigroup situation, this becomes a non-trivial case.
First of all, we can easily get the following simple fact for an invertible semiflow $(T,X)$.

\begin{lem}\label{lem6.25}
Let $(T,X)$ be an invertible semiflow; then the following two statements hold:
\begin{enumerate}
\item[$(1)$] $W\subset X$ is an invariant open set of $(T,X)$ iff $X\setminus{W}$ is an invariant closed set of $(X,T)$.
\item[$(2)$] $(X,T)$ is minimal iff $TU=X$ for every non-empty open set $U$.
\end{enumerate}
\end{lem}

The following seems to be helpful for considering the minimality of the reflection $(X,T)$ with $T$ a non-abelian semigroup. See \cite[Theorem~1.1.(2)-b]{AAB} for cascades on compact metric spaces.

\begin{thm}\label{thm6.26}
Let $(T,X)$ be invertible. Then $(X,T)$ is minimal if and only if $(T,X)$ does not have a non-trivial invariant open subset of $X$. Hence, $(T,X)$ is minimal if and only if there is no non-trivial open invariant set of $(X,T)$.
\end{thm}

\begin{proof}
Let $(X,T)$ be minimal and assume $U$ is a non-trivial open invariant subset of $(T,X)$. Then $X\setminus{U}$ is invariant non-empty closed for $(X,T)$ by Lemma~\ref{lem6.25} and so $X\setminus{U}=X$ contradicting $U$ non-trivial. Thus $X$ does not contain a non-trivial open invariant subset for $(T,X)$.

Conversely, let $X$ have no non-trivial open invariant subset for $(T,X)$ and assume $(X,T)$ is not minimal. Then we can find a non-trivial closed invariant subset $\Theta$ of $(X,T)$. Then $X\setminus{\Theta}$ is a non-trivial open invariant subset of $(T,X)$ by Lemma~\ref{lem6.25} again. Thus this concludes that $(X,T)$ is a minimal semiflow.
\end{proof}

It is clear that every minimal \textit{flow} admits no non-trivial open invariant set. Now, by Theorem~\ref{thm6.19} and Theorem~\ref{thm6.26}, we can easily obtain the following semigroup-action result.

\begin{cor}\label{cor3.27A}
If $(T,X)$ is minimal invertible with $T$ amenable, then there exists no non-trivial, open, and invariant set for $(T,X)$.
\end{cor}

\begin{proof}
If this were false, then $(X,T)$ would not be minimal by Theorem~\ref{thm6.26}. But this contradicts Theorem~\ref{thm6.19}. This completes the proof of Corollary~\ref{cor3.27A}.
\end{proof}

Another result of Theorem~\ref{thm6.19} is the following theorem, which is a generalization of a classical theorem of Tumarkin \cite[Theorem~V7.13]{NS} from the important case of $T=(\mathbb{R},+)$ to the case of general amenable semigroups.

\begin{thm}\label{thm3.28A}
Let $T$ be an amenable semigroup. If $\varLambda$ is a minimal subset of $(T,X)$ such that $\mathrm{Int}_X\varLambda\not=\emptyset$, then $\varLambda$ is clopen in $X$.
\end{thm}

\begin{proof}
Let $y\in\varLambda$ be an interior point of $X$. Then we can pick some index $\varepsilon\in\mathscr{U}_X$ such that $\varepsilon[y]\subseteq\varLambda$. Then $U:=\bigcup_{t\in T}t\varepsilon[y]$ is an open, invariant, and non-empty subset of $X$ such that $U\subseteq\varLambda$. Thus by Theorem~\ref{thm6.19} (more precisely by Corollary~\ref{cor3.27A}), it follows that $U=\varLambda$. This proves Theorem~\ref{thm3.28A}.
\end{proof}

Let $n$ be a positive integer. From Urysohn's theorem the dimension of a compact subset of an $n$-dimensional manifold which has no interior points does not exceed $n-1$ (cf.~\cite[Lemma~2.14]{GH}). Hence we have the following

\begin{cor}[{Hilmy~\cite[Theorem~7.16]{NS} for $T=\mathbb{R}$ and \cite[Theorem~2.15]{GH} for $T$ in groups}]\label{cor3.29A}
Let $(T,M^n)$ be an invertible semiflow on an $n$-dimensional manifold $M^n, n\ge1$, such that $T$ is an amenable semigroup. If $A$ is a compact minimal subset with $A\not=M^n$, then $\mathrm{Int}_{M^n}A=\emptyset$ and $\dim A\le n-1$.
\end{cor}

\begin{proof}
If $\mathrm{Int}_{M^n}A=\emptyset$, then by Urysohn's theorem $\dim A\le n-1$. Now assume $\mathrm{Int}_{M^n}A\not=\emptyset$; then by Theorem~\ref{thm3.28A}, it follows that $A$ is clopen non-trivial in $M^n$. This is a contradiction.
\end{proof}

Let $(G,X)$ be a flow with phase group $G$ and $T$ a normal right-syndetic subgroup of $G$. Then it is a well-known fact that
\begin{itemize}
\item An $x\in X$ is an a.p. point of $(G,X)$ if and only if $x$ is an a.p. point of $(T,X)$ (cf., e.g., \cite[Proposition~2.8]{E69} and \cite[Theorem~1.13]{Aus}).
\end{itemize}

By Theorem~\ref{thm3.28A} we can obtain the following same flavor result using amenability instead of the normality of $T$.

\begin{cor}\label{cor3.30A}
Let $(G,X)$ be an invertible semiflow on a compact connected $T_2$-space $X$ and $T$ a discrete right-syndetic amenable subsemigroup of $G$. Then $(G,X)$ is minimal iff so is $(T,X)$.
\end{cor}

\begin{proof}
We only show the ``only if'' part. Let $Y$ be a minimal subset of $(T,X)$. Let $G=K^{-1}T$ for some subset $K=\{k_1,\dotsc,k_n\}$ of $G$. Then
\begin{equation*}
X=\textrm{cls}_X{GY}=\textrm{cls}_X{K^{-1}TY}={\bigcup}_{k\in K}k^{-1}\textrm{cls}_X{TY}={\bigcup}_{k\in K}k^{-1}Y.
\end{equation*}
Thus $Y$ has non-empty interior. This implies by Theorem~\ref{thm3.28A} that $Y$ is clopen so that $Y=X$.
\end{proof}

It should be noted that if $T$ is not discrete right-syndetic, then the statement of Corollary~\ref{cor3.30A} need not be correct. For example, let $\pi\colon\mathbb{R}\times\mathbb{T}\rightarrow \mathbb{T}$ be periodic of period $1$ on the unit circle; then $\mathbb{Z}$ is right-syndetic in $\mathbb{R}$ under the usual topology but $(\pi,\mathbb{Z},\mathbb{T})$ is not minimal.

\subsubsection{\textit{C}-semigroup and an open question}\label{sec6.3.3}
We do not know if the amenability condition in Theorem~\ref{thm6.19} may be replaced by the one that $(X,T)$ is homogeneous; that is, there is a group $G$ of homeomorphisms of $X$ such that $(G,X)$ is minimal with $t^{-1}g=gt^{-1}$ for all $t\in T$ and $g\in G$. More generally, the following questions would be interesting:
\begin{quote}
\begin{enumerate}
\item {\it If $x\in X$ is a minimal point of $(T,X)$, whether or not $x$ is a minimal point of $(X,T)$, where $T$ is a non-amenable semigroup.}
\item {\it When $(T,X)$ is minimal invertible with $T$ non-amenable, is $(X,T)$ minimal?}
\end{enumerate}
\end{quote}

In view of Examples~\ref{ex1.4}, the general solution to Question~1 above is NO. However the answer to Question~2 above is in the affirmative if the phase semigroup $T$ is a \textit{C}-semigroup (cf.~Definition~\ref{sn3.8B}), as we proceed to show.

\begin{thm}\label{thm6.31}
Let $(T,X)$ be an invertible semiflow where $T$ is not necessarily discrete. Then:
\begin{enumerate}
\item[$(1)$] If $T$ is a left \textit{C}-semigroup and $(X,T)$ is minimal, then $(T,X)$ is minimal.

\item[$(2)$] If $T$ is a right \textit{C}-semigroup and $(T,X)$ is minimal, then $(X,T)$ is minimal.
\end{enumerate}
\end{thm}

\begin{proof}
First of all, note that if $T$ is a compact topological semigroup, then the statements are evidently true. Indeed, let $(T,X)$ be minimal and then for all $x,y\in X$, $Ty=\textrm{cls}_X{Ty}=X$ and so $ty=x$ for some $t\in T$. This implies that for all $x,y\in X$, $y=t^{-1}x=xt$ for some $t\in T$ and thus $\textrm{cls}_X{xT}=xT=X$ for every $x\in X$. Hence $(X,T)$ is minimal. Analogously, $(T,X)$ is minimal if so is $(X,T)$.

We now then suppose that $T$ is a non-compact semigroup. Since minimality is independent of the topology of the phase semigroup $T$, we assume $T$ is an infinite discrete semigroup without loss of generality. (Note that \textit{C}-semigroup relies on topology of $T$, but the general case can be analogously proved.)

(1) Let $(X,T)$ be minimal with $T$ a left \textit{C}-semigroup. We now proceed to show that $(T,X)$ is minimal. To this end, for every $x\in X$, define the \textit{$\omega$-limit set} of $x$ with respect to $(T,X)$ as follows:
\begin{equation*}
\omega_T(x)={\bigcap}_{F\in\mathscr{F}}\textrm{cls}_XF^cx
\end{equation*}
where $\mathscr{F}$ is the collection of finite subsets of $T$ and $F^c$ is the complement of $F$ in $T$. Clearly, $\omega_T(x)$ is closed non-empty by the ``finite intersection property'' (noting that $T$ is non-compact and $X$ is compact by hypothesis) and $\omega_T(x)\subseteq\textrm{cls}_XTx$.

We will show that $\omega_T(x)$ is an invariant set of $(X,T)$. For this, let $y\in\omega_T(x)$ and $s\in T$ be arbitrarily given. Let $F\in\mathscr{F}$ be arbitrary. Since $K:=sF\cup\textrm{cls}_T(T\setminus sT)$ is finite and $sF^c\supseteq K^c$ (for $K^c\subseteq sT=s(F\cup F^c)$), then $y\in\textrm{cls}_X sF^cx$ so there is a net $\{t_n\}$ in $F^c$ such that
$st_nx\to y$ and $t_nx\to z$.
Thus $sz=y$ and $z\in\textrm{cls}_XF^cx$. This shows that $ys=s^{-1}y\in\omega_T(x)$. Thus $\omega_T(x)T\subseteq\omega_T(x)$, i.e., $\omega_T(x)$ is an invariant closed set of $(X,T)$.

However, since $(X,T)$ is minimal by hypothesis, then $\omega_T(x)=X$ for all $x\in X$. Therefore, $\textrm{cls}_XTx=X$ for all $x\in X$ and further $(T,X)$ is minimal.

(2) By symmetry and using the $\alpha$-limit set $\alpha_T(x)$ of $x$, we can easily show that $(X,T)$ is minimal if so is $(T,X)$, whenever $T$ is a right \textit{C}-semigroup.

This thus proves Theorem~\ref{thm6.31}.
\end{proof}

Since $T$ is only a topological semigroup and $L_{s^{-1}}\colon T\rightarrow T$, $t\mapsto s^{-1}t$ need not be well defined, we are not sure that $\omega_T(x)$ is invariant for $(T,X)$ in the above proof of Theorem~\ref{thm6.31}.

\begin{cor}\label{cor6.32}
Let $(T,X)$ be minimal invertible such that $Tt$ is co-finite for all $t\in T$. Then $(X,T)$ is minimal.
\end{cor}

\section{Non-sensitivity of invertible semiflows and w.a.p. $\mathbb{Z}$-flows}\label{sec7}
In this section, we will give simple applications of our Reflection principles I, II and III to chaotic dynamics and $\mathbb{Z}$-flows here.
\subsection{Non-sensitivity of invertible semiflows}
First, we need to introduce and recall some basic notions for self-closeness.

\begin{sn}\label{sn7.1}
Let $(T,X)$ be a semiflow with phase semigroup $T$ with phase map $(t,x)\mapsto tx$. Then:
\begin{enumerate}
\item $(T,X)$ is \textit{sensitive} in case there exists an $\varepsilon\in\mathscr{U}_X$ such that for all $x\in X$ and $\delta\in\mathscr{U}_X$, one can find some $y\in\delta[x]$ and some $t\in T$ with
    $t(x,y)\not\in\varepsilon$, or equivalently, $T(\delta[x],x)\not\subseteq\varepsilon$ (cf.~\cite{G03,KM,DT}).
\item If $(T,X)$ is not sensitive, then it is called \textit{non-sensitive}.

\item Given $\varepsilon\in\mathscr{U}_X$ and $x\in X$, we say $x\in\mathrm{Equi}_\varepsilon(T,X)$ if one can find some $\delta\in\mathscr{U}_X$ such that $T(\delta[x],x)\subseteq\varepsilon$.

\item We say $x\in\mathrm{Tran}^-(T,X)$ if and only if $\mathrm{cls}_XT^{-1}x=X$. Similarly we could define $\mathrm{Tran}^+(T,X)$. If $\mathrm{Tran}^+(T,X)\not=\emptyset$, then $(T,X)$ is said to be \textit{point-transitive} (cf.~\cite{KM, DT}).
\end{enumerate}
By Definitions~(\textbf{a}) and (\textbf{b}) in $\S\ref{sec1.1.1}$, it is easy to check the following statements:
\begin{itemize}
\item $\mathrm{Equi}\,(T,X)=\bigcap_{\varepsilon\in\mathscr{U}_X}\mathrm{Equi}_\varepsilon(T,X)$.
\item If $(T,X)$ is \textit{expansive}, i.e., $\exists\,\varepsilon\in\mathscr{U}_X$ s.t. $x,y\in X$ with $x\not=y$ implies $\exists\, t\in T$ with $t(x,y)\not\in\varepsilon$; then it is sensitive.
\item $(T,X)$ is non-sensitive iff $\mathrm{Equi}_\varepsilon(T,X)\not=\emptyset\ \forall \varepsilon\in\mathscr{U}_X$. Thus, if $\mathrm{Equi}\,(T,X)\not=\emptyset$ then $(T,X)$ is non-sensitive.
\end{itemize}
\end{sn}

It is well known that
\begin{quote}
If $(T,X)$ is a \textit{flow} with $\mathrm{Equi}\,(T,X)\not=\emptyset$ and $\mathrm{Tran}\,(T,X)\not=\emptyset$, then it holds that $\mathrm{Equi}\,(T,X)=\mathrm{Tran}\,(T,X)$ (cf.~\cite[Proposition~1.35]{G03} for the case that $X$ is a compact metric space).
\end{quote}

Moreover, the following more general result holds for semiflows with each $t\in T$ an open self-map of $X$.

\begin{lem}\label{lem7.2}
Let $(T,X)$ be a semiflow with each $t\in T$ an open self-map of $X$. If $(T,X)$ is non-sensitive and $\mathrm{Tran}^-(T,X)\not=\emptyset$, then $\mathrm{Tran}^-(T,X)\subseteq\mathrm{Equi}\,(T,X)$.
\end{lem}

\begin{proof}
Let $x_0\in\mathrm{Tran}^-(T,X)$ and $x\in\mathrm{Equi}_\varepsilon(T,X)$ both be any given points. We then need to verify $x_0\in\mathrm{Equi}_\varepsilon(T,X)$. For this, let $\eta\in\mathscr{U}_X$ such that if $y,z\in\eta[x]$ then $t(y,z)\in\varepsilon$ for all $t\in T$. Since $x_0$ is negatively transitive, there is an $s\in T$ such that $s^{-1}x_0\cap\eta[x]\not=\emptyset$. Then by the openness of $s$, there exists a $\delta\in\mathscr{U}_X$ such that $\delta[x_0]\subseteq s(\eta[x])$. Now for $y,z\in\delta[x_0]$ and $t\in T$,
$t(y,z)=ts(y^\prime,z^\prime)\in\varepsilon$ for some $y^\prime,z^\prime\in\eta[x]$ with $sy^\prime=y, sz^\prime=z$.
This shows that $x_0\in\mathrm{Equi}_\varepsilon(T,X)$ for all $\varepsilon\in\mathscr{U}_X$.
\end{proof}

\begin{prop}\label{prop7.3}
Let $(T,X)$ be a minimal non-sensitive semiflow with $T$ not necessarily discrete. If $(T,X)$ is invertible with $T$ an amenable semigroup or if $T$ is a right \textit{C}-semigroup, then $(T,X)$ is equicontinuous.
\end{prop}

\begin{proof}
(1) Let $(T,X)$ be invertible with $T$ an amenable semigroup. By Theorem~\ref{thm6.19}, $\textrm{cls}_XxT=X$ for all $x\in X$. Thus, $\mathrm{Tran}^{-}(T,X)=X$. Then by Lemma~\ref{lem7.2}, $\mathrm{Equi}\,(T,X)=X$ and so $(T,X)$ is equicontinuous by Lemma~\ref{lem1.6}.

(2) Assume $T$ is a right \textit{C}-semigroup. Given $\varepsilon\in\mathscr{U}_X$, there are $x_0\in X$ and $\delta^\prime\in\mathscr{U}_X$ such that $T(\delta^\prime[x_0],x_0)\subset\frac{\varepsilon}{3}$ by non-sensitivity. Now since $(T,X)$ is minimal, for every $x\in X$ there are $s\in T$ and $\delta\subseteq\delta^\prime$ such that $s(\delta[x])\subseteq\delta^\prime[x_0]$. In addition, since $T\setminus Ts$ is relatively compact in $T$, we can take an $\eta\in\mathscr{U}_X$ with $\eta\subseteq \delta$ so small that $t(\eta[x],x)\subset\frac{\varepsilon}{3}$ for all $t\in T\setminus Ts$. Thus $T(\eta[x],x)\subseteq\varepsilon$. Since $\varepsilon$ is arbitrary, $\mathrm{Equi}\,(T,X)=X$ and so $(T,X)$ is equicontinuous by Lemma~\ref{lem1.6}.
\end{proof}

\begin{cor}\label{cor7.4}
Let $(T,X)$ be a minimal semiflow with $T$ not necessarily discrete. Suppose that (1) $T$ is a right \textit{C}-semigroup or (2) $(T,X)$ is invertible with $T$ amenable.
Then $(T,X)$ is either sensitive or equicontinuous.
\end{cor}
\begin{proof}
If $(T,X)$ is sensitive, then it evidently not equicontinuous. Now if $(T,X)$ is non-sensitive, then it is equicontinuous by Proposition~\ref{prop7.3}.
\end{proof}

Since $\mathbb{Z}_+$ is a right \textit{C}-semigroup, hence the case (1) of Corollary~\ref{cor7.4} is a generalization of the Auslander-Yorke dichotomy theorem~\cite{AY}.

\begin{cor}\label{cor7.5}
Let $(T,X)$ be a minimal semiflow with $T$ not necessarily discrete such that (1) $T$ is a right \textit{C}-semigroup or (2) $(T,X)$ is invertible with $T$ amenable. Then $(T,X)$ is equicontinuous if and only if $\mathrm{Equi}\,(T,X)\not=\emptyset$.
\end{cor}

\begin{proof}
If $\mathrm{Equi}\,(T,X)\not=\emptyset$, then $(T,X)$ is non-sensitive and so it is equicontinuous by Proposition~\ref{prop7.3}. This proves Corollary~\ref{cor7.5}.
\end{proof}

The following corollary is a reflection principle on sensitivity of invertible semiflows in amenable semigroups or \textit{C}-semigroups.

\begin{cor}\label{cor7.6}
Let $(T,X)$ be a minimal invertible semiflow with $T$ not necessarily discrete such that $T$ is either a \textit{C}-semigroup or an amenable semigroup. Then $(T,X)$ is sensitive if and only if so is $(X,T)$.
\end{cor}

\begin{proof}
Assume $(T,X)$ is sensitive. If $(X,T)$ were not sensitive, then it would be non-sensitive minimal by Theorems~\ref{thm6.19} and \ref{thm6.31} and so equicontinuous by Proposition~\ref{prop7.3}. Moreover, by Reflection principle I (Proposition~\ref{prop6.1}), it follows that $(T,X)$ would be equicontinuous. This is a contradiction. Conversely, if $(X,T)$ is sensitive, then we could similarly prove that $(T,X)$ is sensitive.
\end{proof}

Then by Theorem~\ref{thm2.1} and Corollary~\ref{cor7.4}, we can easily obtain the following.

\begin{cor}\label{cor7.7}
Let $(T,X)$ be minimal surjective with $T$ a right \textit{C}-semigroup not necessarily discrete.
If there is some $t\in T$ non-invertible, then $(T,X)$ is sensitive.
\end{cor}

\begin{proof}
If $(T,X)$ were non-sensitive, then by Corollary~\ref{cor7.4} $(T,X)$ would be equicontinuous and so distal by Theorem~\ref{thm2.1}. This is a contradiction.
\end{proof}

Now by using distality instead of amenability and \textit{C}-semigroup, we can obtain the following dichotomy.

\begin{prop}\label{prop7.8}
Let $(T,X)$ be a minimal distal semiflow. Then $(T,X)$ is equicontinuous iff $(T,X)$ is non-sensitive.
\end{prop}

\begin{proof}
Let $(T,X)$ be non-sensitive. By Lemma~\ref{lem7.2}, $\mathrm{Tran}^-(T,X)\subseteq\mathrm{Equi}\,(T,X)$. Further by Reflection principle I, $\mathrm{Tran}^-(T,X)=X$ so $(T,X)$ is equicontinuous. The other side implication is evident. Thus the proof is complete.
\end{proof}

This shows that the dynamics of any non-equicontinuous minimal distal semiflow could not be predictable.
\subsection{Weakly almost periodic $\mathbb{Z}$-flows on zero-dimensional spaces}\label{sec7.2}
Recall that in \cite{E59, EEN} a flow $(T,X)$ is called weakly almost periodic (named weakly equicontinuous by $\S\ref{sec1.1.5}$\,(\textbf{p})) if for all $f\in C(X)$, $fT$ is relatively compact in $C(X)$ provided with the pointwise topology.

However, following Gottschalk \cite{G,GH} the notion of weakly almost periodic has completely different explanation as follows.

\begin{defn}[\cite{G,GH}]\label{d7.9}
We say that $(T,X)$ is \textit{weakly almost periodic} (w.a.p.) provided that if $\alpha\in\mathscr{U}_X$, then there exists a compact subset $K$ of $T$ such that $Ktx\cap\alpha[x]\not=\emptyset$ for all $x\in X$ and all $t\in T$.
\end{defn}

It follows easily from the definition that every factor of a w.a.p. semiflow is also w.a.p.. 
When $(T,X)$ is a.p. (or equivalently equicontinuous surjective, cf.~Definition~\ref{u8.1}), then it is w.a.p. obviously. That the converse is not true is shown by the fact that every minimal semiflow is w.a.p. (\cite{G,GH}).

\begin{defn}[\cite{G,GH}]\label{d7.10}
Let $f\colon X\rightarrow X$ be a homeomorphism. It naturally induces the flow $\pi\colon \mathbb{Z}\times X\rightarrow X$ with phase map $(t,x)\mapsto f^t(x)$ and semiflow $\pi_+\colon \mathbb{Z}_+\times X\rightarrow X$ and its reflection $\pi_-\colon X\times\mathbb{Z}_+\rightarrow X$ with phase map $(x,t)\mapsto f^{-t}(x)$. Then:
\begin{enumerate}
\item $(f,X)$ is called \textit{positively recurrent} at $x\in X$ if there is a net $t_n\to\infty$ such that $f^{t_n}(x)\to x$; or equivalently, $x\in\overline{\pi([1,\infty),x)}$. That is to say, $\pi_+\colon \mathbb{Z}_+\times X\rightarrow X$ is recurrent at $x$.

\item $(f,X)$ is called \textit{negatively recurrent} at $x\in X$ if there is a net $t_n\to\infty$ such that $f^{-t_n}(x)\to x$; or equivalently, $x\in\overline{\pi((\infty,-1],x)}$. That is to say, $\pi_-\colon X\times \mathbb{Z}_+\rightarrow X$ is recurrent at $x$.
\item We say that $(f,X)$ is \textit{recurrent at} $x$ if it is both positively and negatively recurrent at $x$. When each point of $X$ is recurrent, we say $(f,X)$ is \textit{pointwise recurrent}.
\end{enumerate}

Now we consider the reflection question: \textit{If $(f,X)$ is pointwise positively recurrent, is it pointwise negatively recurrent? If $(f,X)$ is positively w.a.p., is it negatively w.a.p.?}
\end{defn}

When the phase space $X$ is a compact zero-dimensional space like a closed subset of the classical symbolic space, we can then present a confirmative solution to this by using our Reflection principle~II and Gottschalk's theorem.

\begin{lem}[{\cite[Theorem~6]{G}}]\label{lem7.11}
Let $f\colon X\rightarrow X$ be a homeomorphism on a compact zero-dimensional space $X$. If $(f,X)$ is pointwise positively recurrent, then $(f,X)$ is positively w.a.p. (i.e. $\pi_+\colon\mathbb{Z}_+\times X\rightarrow X$ is w.a.p. in the sense of Definition~\ref{d7.9}).
\end{lem}

From Gottschalk's proof (cf.~\cite[p.~766]{G}), readers can see that the structure of $\mathbb{Z}_+$ plays a crucial role and there is no expectation for the w.a.p. of the flow $\pi\colon\mathbb{Z}\times X\rightarrow X$.
We can, however, obtain the following results:

\begin{thm}\label{thm7.12}
Let $f\colon X\rightarrow X$ be a homeomorphism of a compact zero-dimensional space $X$. Then the following are pairwise equivalent:
\begin{enumerate}
\item[$(1)$] $\pi_+\colon \mathbb{Z}_+\times X\rightarrow X$ is pointwise recurrent.
\item[$(2)$] $\pi_-\colon X\times\mathbb{Z}_+\rightarrow X$ is pointwise recurrent.
\item[$(3)$] $\pi_+\colon\mathbb{Z}_+\times X\rightarrow X$ is w.a.p..
\item[$(4)$] $\pi_-\colon X\times\mathbb{Z}_+\rightarrow X$ is w.a.p..
\item[$(5)$] $(f,X)$, or equivalently $\pi\colon \mathbb{Z}\times X\rightarrow X$, is w.a.p..
\item[$(6)$] $(f,X)$, or equivalently $\pi\colon \mathbb{Z}\times X\rightarrow X$, is pointwise recurrent.
\end{enumerate}
Here ``w.a.p.'' is in the sense of Definition~\ref{d7.9}. 
\end{thm}

\begin{proof}
We first note that w.a.p. implies pointwise a.p. and pointwise a.p. implies pointwise recurrent for any cascade system. Thus $(1)\Rightarrow(3)$ by Lemma~\ref{lem7.11}; and $(3)\Rightarrow(2)$ by Reflection principle~II. $(2)\Rightarrow(4)$ by Lemma~\ref{lem7.11} and $(4)\Rightarrow(1)$ by Reflection principle~II again. Therefore, we have concluded that $(1)\Leftrightarrow(2)\Leftrightarrow(3)\Leftrightarrow(4)$. From this, $(1)\Leftrightarrow(6)$ follows easily.

Finally we proceed to the proof of $(1)\Leftrightarrow(5)$. Let (1) hold; and then (3) and (4) both hold. Therefore, for all $\alpha\in\mathscr{U}_X$, there are two finite subsets $K_+$ and $K_-$ of $\mathbb{Z}_+$ such that for all $x\in X$,
$$
f^{K_++t}(x)\cap\alpha[x]\not=\emptyset\quad \textrm{and}\quad f^{-K_--t}(x)\cap\alpha[x]\not=\emptyset\quad \forall t\in \mathbb{Z}_+.
$$
Now let $K=K_+\cup(-K_-)$, which is a finite subset of $\mathbb{Z}$ such that $f^{K+t}(x)\cap\alpha[x]\not=\emptyset$ for all $x\in X$ and $t\in\mathbb{Z}$. This shows that $(f,X)$ is w.a.p..

Conversely, suppose $(f,X)$ is w.a.p. and let $x\in X$. Let $U$ be a neighborhood of $x$. Then $N_T(x,U)$ where $T=\mathbb{Z}$ is relatively dense in $\mathbb{Z}$. Whence $\mathbb{Z}_+\cap N_T(x,U)$ is relatively dense in $\mathbb{Z}_+$. This thus shows that every point of $X$ is a.p. and further recurrent for $\pi_+\colon\mathbb{Z}_+\times X\rightarrow X$. Thus (1) holds.
The proof is therefore completed.
\end{proof}
\section{Almost periodic flow and w.a.p. functions}\label{uap}
Let $T$ be a topological semigroup with $e\in T$, not necessarily discrete, and $X$ a non-empty compact $T_2$ space. In this section, we will consider ``a.p.'' flows, ``weakly equicontinuous'' semiflows and ``point-distal'' flows on $X$.

We begin with recall and introduce some basic notions. $(T,X)$ is equicontinuous ($\S\ref{sec1.1.1}$\,(\textbf{a})) if and only if given $\varepsilon\in\mathscr{U}_X$ there exists $\delta\in\mathscr{U}_X$ with $T\delta\subseteq\varepsilon$; it is weakly equicontinuous ($\S\ref{sec1.1.5}$\,(\textbf{p})) if and only if $E(X)\subseteq C(X,X)$.

\begin{defn}\label{u8.1}
$(T,X)$ is called \textit{almost periodic} (a.p.) if given $\varepsilon\in\mathscr{U}_X$, there exists a right-syndetic set $A$ in $T$ such that $Ax\subseteq\varepsilon[x]$ for all $x\in X$.
\end{defn}

Here $A$ is uniformly for all point $x$ of $X$ so that we can find a compact subset $K$ of $T$ such that $Ktx\cap\varepsilon[x]\not=\emptyset$ for all $x\in X$ and $t\in T$. Thus an a.p. semiflow is w.a.p. (cf.~Definition~\ref{d7.9}). 

An a.p. $(T,X)$ is pointwise a.p. whence $X=\bigcup_{x\in X}\overline{Tx}$ is a continuous partition of $X$ into minimal sets (\cite[Theorem~5]{G}). We will show an a.p. semiflow is equicontinuous.

\begin{defn}\label{u8.2}
\begin{enumerate}
\item A subset $A$ of $T$ is called \textit{bilaterally syndetic} in $T$ if there are compact subsets $K_R$ and $K_L$ of $T$ such that $K_Rt\cap A\not=\emptyset$ and $tK_L\cap A\not=\emptyset$ for all $t\in T$.
\item If $A$ is a bilaterally syndetic subset of $T$ in Definition~\ref{u8.1}, then $(T,X)$ is called a \textit{bilaterally a.p. semiflow}.
\end{enumerate}

Clearly a bilaterally syndetic set is left- and right-syndetic (cf.~Definition~\ref{sn2.4}). When $T$ is not abelian, a right-syndetic set is generally not bilaterally syndetic even if $T$ is a topological group.
\end{defn}

For a semiflow, using our Theorem~\ref{thm2.1} as an important tool, Dai and Xiao in \cite{DX} have proved the equivalence of `almost periodic' and `equicontinuous surjective' which shows that whether or not $(T,X)$ is a.p. does not depend upon the topology on the phase semigroup $T$. Here we will present another proof by using Reflection Principles together with Ellis' enveloping semigroup.

\begin{thm}\label{u8.3}
Let $(T,X)$ be a semiflow with phase semigroup $T$. Then the following statements are pairwise equivalent:
\begin{enumerate}
\item[$(1)$] $(T,X)$ is an a.p. semiflow.
\item[$(2)$] $(T,X)$ is equicontinuous surjective.
\item[$(3)$] $E(T,X)$ is a group of homeomorphisms.

\item[$(4)$] Given $\alpha\in\mathscr{U}_X$ there is a bilaterally discretely syndetic subset $A$ of $T$ such that $Ax\subseteq\alpha[x]$ for all $x\in X$.

\item[$(5)$] $(T,X)$ is discretely a.p..
\end{enumerate}
If $T$ is a group here, then each of $(1), (2), (3), (4)$, and $(5)$ is equivalent to the following
\begin{enumerate}
\item[$(6)$] Given $\alpha\in\mathscr{U}_X$ there is a symmetric discretely syndetic subset $A$ of $T$ such that $Ax\subseteq\alpha[x]$ for all $x\in X$.
\end{enumerate}
\begin{note}
We notice that $(1)\Leftrightarrow(2)\Leftrightarrow(3)$ is exactly \cite[Proposition~4.4]{E69} for $T$ in groups. Here $A$ is symmetric iff $A=A^{-1}$.
\end{note}
\end{thm}

\begin{proof}
(1) $\Rightarrow$ (2): Assume $(T,X)$ is a.p.; then given any $\varepsilon\in\mathscr{U}_X$ there are a $\delta\in\mathscr{U}_X$ and a right-syndetic subset $A$ of $T$ such that if $(x,y)\not\in\varepsilon$ then $a(x,y)\not\in\delta$ for all $a\in A$. Let $K$ be a compact subset of $T$ such that $T=K^{-1}A$. We can assert that there is some $\alpha\in\mathscr{U}_X$ such that if $(x,y)\not\in\delta$ then $\alpha\cap t^{-1}(x,y)=\emptyset$ for all $t\in K$.

Indeed, otherwise, for all $\eta\in\mathscr{U}_X$ there are $(x_\eta,y_\eta)\not\in\delta$ and $t_\eta\in K$ such that $\eta\cap t_\eta^{-1}(x_\eta,y_\eta)\not=\emptyset$ so that $(x_\eta,y_\eta)\in t_\eta\eta$. Thus $(x_\eta,y_\eta)\in K\eta$. However, since $K\varDelta_X\subseteq\varDelta_X$ and $K\times X\times X\rightarrow X\times X$ is continuous, hence as $\eta$ is small sufficiently, $K\eta\subset\delta$ and thus $(x_\eta,y_\eta)\not\in\delta$ contradicts $(x_\eta,y_\eta)\in K\eta$. This proves our assertion.

Therefore, for all $k\in K$ and $a\in A$, if $(x,y)\not\in\varepsilon$ then $\alpha\cap k^{-1}a(x,y)=\emptyset$. Thus if $(x,y)\not\in\varepsilon$ then $t(x,y)\not\in\alpha$ for all $t\in T$. That is, $(T,X)$ is uniformly distal (cf.~Definition~\ref{def1.9}). Whence by Theorem~\ref{thm1.14}, $(T,X)$ is equicontinuous surjective.

(2) $\Rightarrow$ (3): Let $(T,X)$ be surjective equicontinuous. Then by Theorem~\ref{thm1.13} and Lemma~\ref{lem6.7}, $E(T,X)$ is a group in $C(X,X)$.

(3) $\Rightarrow$ (4): Let $E(T,X)$ be a group in $C(X,X)$ and $\alpha\in\mathscr{U}_X$. Then $E(T,X)$ is a compact topological group by Ellis' joint continuity theorem.\footnote{\textbf{Ellis' Joint Continuity Theorem}\,(\cite[Theorem~1]{E57})\textbf{.}\ {\it Let $G$ be a locally compact $T_2$ space with a group structure for which $(s,t)\mapsto st$ is separately continuous, and suppose $T$ acts on a compact $T_2$ space $X$ in a separately continuous manner. Then the action of $T$ on $X$ is jointly continuous.}

The readers may find in Appendix an independent proof of Ellis' theorem; see Theorem~\ref{thm9.8} in Appendix for its semiflow versions.}
So $(p,x)\mapsto p(x)$ of $E(T,X)\times X$ to $X$ is continuous by Ellis' joint continuity theorem again. Hence there exists a neighborhood $N$ of the identity $e$ of $E(T,X)$ such that
$Nx\subseteq\alpha[x]$ for all $x\in X$.

Since $\overline{T}=\overline{T^{-1}}=E(T,X)$ and $E(T,X)$ is compact, there exist two finite subsets $K_R$ and $K_L$ of $T$ such that $K_R^{-1}N=E(T,X)=NK_L^{-1}$. Let $A=N\cap T$ and $K=K_R\cup K_L$. If $k^{-1}p=t\in T$ or $pk^{-1}=t\in T$ for some $k\in K$ and $p\in N$, then $p=kt\in T$ or $p=tk\in T$. Thus $T=K^{-1}A=AK^{-1}$ and $Ax\subseteq\alpha[x]$ for all $x\in X$. This shows that $(T,X)$ is a discretely bilaterally a.p. semiflow.

(4) $\Rightarrow$ (5) $\Rightarrow$ (1) are obvious. Thus we have concluded that $(1)\Leftrightarrow(2)\Leftrightarrow(3)\Leftrightarrow(4)\Leftrightarrow(5)$.

Next, suppose $T$ is a topological group and then $(6)\Rightarrow(1)$ is obvious. We now proceed to the proof of $(3)\Rightarrow(6)$.
Indeed, if we take $N$ a `symmetric' neighborhood of $e$ in the above proof of $(3)\Rightarrow(4)$, then $A=A^{-1}$ whence $A$ is symmetric discretely syndetic such that $Ax\subseteq\alpha[x]$.

The proof of Theorem~\ref{u8.3} is completed.
\end{proof}

It should be noted that the compactness of $X$ is essential for Theorem~\ref{u8.3}. For example, the $C^0$-flow $\mathbb{R}\times X\rightarrow X,\ (t,x)\mapsto t+x$, where $X=\mathbb{R}$ with the usual topology, is equicontinuous but it is not a.p. with no a.p. point at all.

Moreover, if starting from Theorem~\ref{u8.3} as Proof (II) of Theorem~\ref{thm2.1}, we can easily obtain Theorem~\ref{thm2.1}. But the proof of Theorem~\ref{u8.3} is itself based on Theorem~\ref{thm2.1} in \cite{DX} and here.

Theorem~\ref{u8.3} implies the following important fact (cf.~\cite[Proposition~4.7]{E69} and \cite[Corollary~2.6]{Aus} for $T$ in groups and \cite[Proposition~3.6]{DX} for the general case by different approaches).

\begin{cor}[{cf.~\cite[Proposition~4.7]{E69} for $T$ in groups}]\label{u8.4}
Every factor of an equicontinuous surjective semiflow is always equicontinuous surjective.
\end{cor}

\begin{proof}
Let $(T,Y)$ be a factor of an equicontinuous surjective semiflow $(T,X)$ via an epimorphism $\pi\colon X\rightarrow Y$. We proceed to show $(T,Y)$ is a.p.. Given $\varepsilon\in\mathscr{U}_Y$, there is $\delta\in\mathscr{U}_X$ with $\pi\delta\subset\varepsilon$. Since $(T,X)$ is a.p. by Theorem~\ref{u8.3}, there is a right-syndetic subset $A$ of $T$ such that $Ax\subseteq\delta[x]$ for all $x\in X$. Now for each $y\in Y$, having chosen $x\in\pi^{-1}(y)$, $Ay=A\pi x=\pi Ax\subseteq\pi\delta[x]\subset\varepsilon[y]$. Thus $(T,Y)$ is a.p. and so it is equicontinuous surjective.
\end{proof}

\begin{cor}[{cf.~\cite[Proposition~4.8]{E69} for $T$ in groups}]\label{u8.5}
\begin{enumerate}
\item[$1)$] Let $(T,X)$ be a.p. and let $M$ be a closed invariant subset of $X$. Then $(T,M)$ is an a.p. subsemiflow.
\item[$2)$] Let $(T,X_i), i\in I$, be a family of semiflows and $(T,X)=(T,\prod_iX_i)$. Then $(T,X)$ is a.p. iff $(T,X_i)$ is a.p. for all $i\in I$.
\end{enumerate}
\end{cor}

\begin{proof}
1) Firstly $(T,M)$ is equicontinuous. Secondly by Theorem~\ref{thm2.1}, $(T,M)$ is distal so that $(T,M)$ is surjective by (2) of Theorem~\ref{thm1.13}. Whence $(T,M)$ is a.p. by Theorem~\ref{u8.3}.

2) Let $(T,X)$ be a.p.; then $(T,X_i)$ is a.p. by Corollary~\ref{u8.4}. Conversely, if $(T,X_i), i\in I$, is a.p., then by Theorem~\ref{u8.3} it follows that $(T,X)$ is equicontinuous surjective. So $(T,X)$ is a.p. by Theorem~\ref{u8.3} again.
\end{proof}

Although the notion `syndetic' is closely related to the topology of the phase semigroup $T$, yet Theorem~\ref{u8.3} shows that `a.p.' does not depend on it, since distality and equicontinuity do not depend on it.
In fact we can obtain more from Theorem~\ref{u8.3}.

\begin{cor}[{cf.~\cite[Remark~4.32]{GH} for $T$ in groups}]\label{u8.6}
Let $(T,X)$ be a surjective semiflow. Then $(T,X)$ is a.p. if and only if for every $\alpha\in\mathscr{U}_X$ there exists a finite subset $K$ of $T$ such that to each $t\in T$ there corresponds $k\in K$ with $tx\in\alpha[kx]$ for all $x\in X$.
\end{cor}

\begin{proof}
Assume $(T,X)$ is a.p. and let $\alpha\in\mathscr{U}_X$. Then by Theorem~\ref{u8.3}, $(T,X)$ is equicontinuous. So by Lemma~\ref{lem1.6}, there is a finite set $K\subset T$ such that to each $t\in T$ there is some $k\in K$ with $tx\in\alpha[kx]$ for all $x\in X$.

Conversely, suppose for every $\alpha\in\mathscr{U}_X$ there exists a finite subset $K$ of $T$ such that to each $t\in T$ there corresponds $k\in K$ with $tx\in\frac{\alpha}{3}[kx]$ for all $x\in X$. Since $K$ is finite and $X$ is compact, there is $\delta\in\mathscr{U}_X$ such that $K\delta\subseteq\frac{\alpha}{3}$. This implies that if $(x,y)\in\delta$ then $(tx,ty)\in\alpha$ for all $t\in T$. Thus $(T,X)$ is a.p. by Theorem~\ref{u8.3} again.
\end{proof}

If here $k$ depends upon $x$ in the foregoing corollary, it is exactly a characterization of a.p. points (cf.~\cite[Theorem~4.1]{CD} and Remark~\ref{rem3.14B}).

\medskip
Let $C(X)$ be the algebra of real-valued continuous functions on $X$; then based on a semiflow $(T,X)$, $T$ can act from the right on $C(X)$ by the means $(f,t)\mapsto f_t=ft$. This induces the canonical semiflow $C(X)\times T\rightarrow C(X)$
which is jointly continuous w.r.t. the pointwise topology of $C(X)$.
Here $ft(x)=f(tx)$ for all $x\in X$ and write $fT:=\{f_t\,|\,t\in T\}$.

\begin{defn}[\cite{Eb,E59}]\label{u8.7}
Let $f\in C(X)$, $F\subseteq C(X)$ and $x_0\in X$ be any given. Then
\begin{enumerate}
\item $f$ is an \textit{almost periodic function} (a.p.f.) on $X$ w.r.t. $(T,X)$ iff $fT$ is relatively compact in $C(X)$ provided with the uniform topology.
\item $f$ is a \textit{weakly almost periodic function} (w.a.p.f.) on $X$ w.r.t. $(T,X)$ iff $fT$ is relatively compact in $C(X)$ provided with the pointwise topology.
\item We shall say that \textit{$x_0$ is separated by} $F$ or \textit{$F$ separates $x_0$} if for every $x\in X$ such that $x_0\not=x$ there exists an $f\in F$ with $f(x_0)\not=f(x)$.
\end{enumerate}
\end{defn}

Similar to \cite[Proposition~4.15]{E69} for $T$ in groups, with the aid of Theorem~\ref{thm2.1} (precisely Theorem~\ref{u8.3}), we can characterize the a.p., equicontinuity, and weak equicontinuity of $(T,X)$ in terms of $C(X)$.

\begin{prop}\label{u8.8}
Let $(T,X)$ be a semiflow with phase semigroup $T$. Then:
\begin{enumerate}
\item $(T,X)$ is a.p. if and only if $(T,X)$ is surjective and every $f\in C(X)$ is an a.p.f. w.r.t. $(T,X)$.
\item $(T,X)$ is equicontinuous if and only if every $f\in C(X)$ is an a.p.f. w.r.t. $(T,X)$.
\item $(T,X)$ is weakly equicontinuous if and only if every $f\in C(X)$ is w.a.p. w.r.t. $(T,X)$.
\end{enumerate}
\begin{note}
The above 2. and 3. are in fact contained in Ellis' \cite[(3) of Lemma~1]{E59} and \cite[Lemma~4]{E59}, respectively. We will reprove them here from a modern perspective. In view of 3. the ``weakly equicontinuous'' is also called ``weakly almost periodic'' in \cite{E59, EN}.
\end{note}
\end{prop}

\begin{proof}
1. Let $(T,X)$ be a.p. and $f\in C(X)$. Then by Theorem~\ref{u8.3}, $fT$ and its uniform closure are equicontinuous and $(T,X)$ is surjective. Hence $fT$ is relatively compact in $C(X)$ under the supremum norm by Ascoli's theorem (cf.~\cite[Theorem~7.18 in p.~234]{Kel}) or by Lemma~\ref{lem1.6}.

Conversely assume $fT$ is relatively compact in $C(X)$ under the uniform topology, for all $f\in C(X)$; and let $\alpha\in\mathscr{U}_X$ be any given. Then there exist $\varepsilon>0$ and $F$ a finite subset of $C(X)$ such that if $|f(x)-f(y)|<\varepsilon$ for all $f\in F$, then $(x,y)\in\alpha$ (cf.~\cite[p.~47]{Aus}). By Ascoli's theorem $fT$, for each $f\in F$, is equicontinuous. Thus since $F$ is finite, there exists an index $\delta\in\mathscr{U}_X$ with
\begin{gather*}
|f_t(x)-f_t(y)|<\varepsilon\ \forall t\in T,\ f\in F\ \textrm{and}\ (x,y)\in\delta.
\end{gather*}
Hence $T\delta\subseteq\alpha$ and $(T,X)$ is equicontinuous and so a.p. by Theorem~\ref{u8.3}.

2. The same argument implies the second assertion of Proposition~\ref{u8.8} and we thus omit the details here.

3. Let $(T,X)$ be weakly equicontinuous (i.e. $E(X)\subset C(X,X)$ by $\S\ref{sec1.1.5}\,(\textbf{p})$). Let $f\in C(X)$ and define $\phi(p)=fp$ for all $p\in E(X)$. If $x_n\to x$ in $X$, then $\lim\phi(p)(x_n)=f(p(x))=\phi(p)(x)$. Thus $\phi(p)\in C(X)$. Now we claim that the map $\phi\colon E(X)\rightarrow C(X)$ is continuous under the pointwise topologies. To see this let $p_n\to p$ in $E(X)$; then
$\lim\phi(p_n)(x)=\lim f(p_n(x))=f(p(x))=\phi(p)(x)$.
Thus the image of $\phi$ is compact and it contains the set $fT=\phi(T)$. Thus each $f\in C(X)$ is w.a.p. w.r.t. $(T,X)$.

Conversely, let every $f\in C(X)$ be w.a.p. w.r.t. $(T,X)$ and $p\in E(X)$. To show $p\in C(X,X)$, it is enough to verify that for any $f\in C(X)$, $fp\colon X\rightarrow\mathbb{R}$ is continuous. For this, let $t_n\to p$ in $E(X)$; then $f{t_n}\to fp$ in pointwise topology. Then by Definition~\ref{u8.7}, there is no loss of generality in assuming that $f{t_n}\to g$, for some $g\in C(X)$, in the sense of the pointwise topology. Thus $fp=g$ is continuous and then $E(X)\subset C(X,X)$.

The proof of Proposition~\ref{u8.8} is thus completed.
\end{proof}

\begin{cor}\label{u8.9}
\begin{enumerate}
\item[$(1)$] Every factor of an equicontinuous semiflow is always equicontinuous.
\item[$(2)$] Every factor of a weakly equicontinuous semiflow is weakly equicontinuous.
\item[$(3)$] Every subsemiflow of a weakly equicontinuous semiflow is weakly equicontinuous.
\end{enumerate}
\end{cor}

\begin{proof}
(1) Let $\pi\colon X\rightarrow Y$ be a factor map where $(T,X)$ equicontinuous. Then $\pi^*\colon f\mapsto f\circ\pi$ of $C(Y)$ into $C(X)$ is continuous 1-1 under the uniform topology. Let $f\in C(Y)$. Then $\pi^*(f)T$ is relatively compact in $C(X)$ by 2. of Proposition~\ref{u8.8}. So by $\pi^*(f)T=\pi^*(fT)$, it follows that $fT$ is relatively compact in $C(Y)$. Whence $(T,Y)$ is equicontinuous by 2. of Proposition~\ref{u8.8} again.

(2) If $E(X)\subset C(X,X)$, then it is easy to check that $E(Y)\subset C(Y,Y)$. Thus $(T,Y)$ weakly equicontinuous if so is $(T,X)$.

(3) Let $(T,X)$ be weakly equicontinuous and $Y$ is closed invariant subset of $X$. Let $p\in E(Y)$ with $t_n\to p$ in $E(Y)$. Then $t_ny\to p(y)$ for all $y\in Y$. Since $E(X)$ is compact, we can take a subnet $\{t_i\}$ from $\{t_n\}$ such that $t_i\to q\in E(X)\subseteq C(X,X)$. Then $t_ix\to q(x)$ for all $x\in X$. This shows that $p=q|Y$ whence $p\in C(Y,Y)$. Thus $(T,Y)$ is weakly equicontinuous.

This proves Corollary~\ref{u8.9}.
\end{proof}

It turns out that Corollary~\ref{u8.9} implies Corollary~\ref{u8.4} because every factor of a surjective semiflow is surjective. The following is another consequence of Theorem~\ref{thm1.13} and Proposition~\ref{u8.8} using the notion of weakly equicontinuous.

\begin{thm}\label{u8.10}
Let $(G,X)$ be a minimal flow with phase group $G$ and $T$ a subsemigroup of $G$ generating $G$, i.e., $G=\langle T\rangle$. If $(T,X)$ is distal and if there exists an $x_0\in X$ such that the w.a.p. functions on $X$ w.r.t. $(T,X)$ separate it, then $(G,X)$ is an equicontinuous flow.
\end{thm}

\begin{proof}
At first we note that $(G,X)$ is minimal distal by (3) of Theorem~\ref{thm1.13}.
Let $F_{wap}(X)$ be the set of w.a.p. functions on $X$ w.r.t. $(T,X)$. Define a relation $R$ on $X$ as follows:
$$
R=\{(x,x^\prime)\,|\,f(x)=f(x^\prime)\ \forall f\in F_{wap}(X)\}.
$$
It is easy to check that $R$ is a closed $T$-invariant equivalence relation on $X$. Let $\pi\colon X\rightarrow X/R$ be the canonical projection. Then $\pi\colon (T,X)\rightarrow(T,X/R)$ is a factor map such that $\pi$ is 1-1 at $x_0$ (i.e. $\pi^{-1}(\pi(x_0))$ is a singleton set). By the definition of $R$, we can see that the mapping
$$
\tilde{{}}\colon F_{wap}(X)\ni f\mapsto\tilde{f}\in C(X/R),\quad \textrm{where }\tilde{f}([x]_R)=f(x)\ \forall [x]_R\in X/R,
$$
is well defined and continuous. Thus $\tilde{f}$, for each $f\in F_{wap}(X)$, is a w.a.p.f. on $X/R$ w.r.t. $(T,X/R)$. Since $F_{wap}(X/R)$ separates all of the points of $X/R$, $F_{wap}(X/R)=C(X/R)$ so that $(T,X/R)$ is such that $E(X/R)\subset C(X/R,X/R)$ by Proposition~\ref{u8.8}. Since $(T,X/R)$ is distal, so $E(X/R)$ is group of homeomorphisms of $X/R$ and further $(T,X/R)$ is equicontinuous surjective.
Thus by (4) of Theorem~\ref{thm1.13}, $(G,X/R)$ is an equicontinuous flow.

In addition, $\pi\colon(G,X)\rightarrow(G,X/R)$ is also a factor map. Because $\pi$ is 1-1 at $x_0$, $(G,X)$ is ``locally almost periodic'' at $x_0$, that is, for all neighborhood $U$ of $x_0$ there are a right-syndetic subset $A$ of $G$ and a neighborhood $V$ of $x_0$ such that $AV\subseteq U$. By minimality of $(G,X)$, $(G,X)$ is locally almost periodic at every point of $X$.

Whence $(G,X)$ is equicontinuous (cf., e.g., \cite[Corollary~5.28]{E69} that asserts: \textit{A flow is equicontinuous iff it is distal and locally almost periodic}). The proof is thus completed.
\end{proof}

Since an a.p.f. on $X$ w.r.t. $(T,X)$ must be a w.a.p.f. on $X$ w.r.t. $(T,X)$, by Theorem~\ref{u8.10} we can easily obtain a generalization of a theorem of Ellis.

\begin{cor}[{cf.~\cite[Theorem~5 for the case $G=T$]{E59}}]\label{u8.11}
Let $(G,X)$ be a minimal flow and $T$ a subsemigroup of $G$ generating $G$. If $(T,X)$ is distal and if there exists an $x_0\in X$ such that the a.p. functions on $X$ w.r.t. $(T,X)$ separate it, then $(G,X)$ is an equicontinuous flow.
\end{cor}

The following theorem presents two sufficient and necessary conditions for point-distal flows to have non-trivial equicontinuous factors, which is possibly useful for Veech's question (cf.~Question~\ref{q5.7} in $\S\ref{sec5}$).

\begin{thm}\label{u8.12}
Let $(T,X)$ be a point-distal flow with $X$ a non-trivial compact $T_2$ space. Then the following three statements are pairwise equivalent.
\begin{enumerate}
\item[$(1)$] $(T,X)$ has a non-trivial equicontinuous factor.
\item[$(2)$] There exists a non-constant a.p.f. on $X$ w.r.t. $(T,X)$.
\item[$(3)$] There exists a non-constant w.a.p.f. on $X$ w.r.t. $(T,X)$.
\end{enumerate}
\begin{note}
When $(T,X)$ is merely a surjective semiflow, this theorem also holds by the same proof.
\end{note}
\end{thm}

\begin{proof}
(1)$\Rightarrow$(2): Let $\pi\colon(T,X)\rightarrow(T,Y)$ be an extension of a non-trivial equicontinuous flow $(T,Y)$. Then $\pi^*\colon C(Y)\rightarrow C(X)$, defined by $f\mapsto f^*=f\circ\pi$, is continuous 1-1 such that $f^*t=(ft)^*$ for all $t\in T$ and $f\in C(Y)$. Let $f\in C(Y)$ be non-constant. Since $fT$ is relatively compact in $C(Y)$ under the uniform topology, $(fT)^*=f^*T$ is also relatively compact in $C(X)$ with the uniform topology. Whence $f^*\in C(X)$ is a non-constant a.p. function w.r.t. $(T,X)$.

(2)$\Rightarrow$(3) is trivial for an a.p.f. on $X$ is w.a.p. w.r.t. $(T,X)$.

(3)$\Rightarrow$(1): Let there exist a non-constant w.a.p.f. on $X$ w.r.t. $(T,X)$. Let $F_{wap}(X)$ be the set of w.a.p. functions w.r.t. $(T,X)$. Define
$$
R=\{(x,x^\prime)\in X\times X\,|\,f(x)=f(x^\prime)\ \forall f\in F_{wap}(X)\}.
$$
Then $R\not=X\times X$ is a closed invariant equivalence relation and let $Y=X/R$. By 3. of Proposition~\ref{u8.8}, $(T,Y)$ is weakly equicontinuous (i.e. $E(Y)\subset C(Y,Y)$). Since $(T,Y)$ is a factor of $(T,X)$, hence $(T,Y)$ is point-distal. Thus by Corollary~\ref{cor6.8}, it follows that $(T,Y)$ is equicontinuous. Therefore $(T,X)$ has a non-trivial equicontinuous factor.

The proof of Theorem~\ref{u8.12} is therefore completed.
\end{proof}
\section{Appendix: revisit to Robert Ellis' joint continuity theorems}\label{sec9}
Let $T$ be a multiplicative topological group or semigroup and $X$ a compact $T_2$-space, and let $T\times X\rightarrow X,\ (t,x)\mapsto tx$ be a separately continuous flow or semiflow. Then the problem to find conditions on $T$ so that $(t,x)\mapsto tx$ jointly continuous goes back, at least, to Baire (1899); see, e.g., \cite{Bou, For, E57, Tro, Nam, Aus}. In this Appendix, we will revisit Robert Ellis' joint continuity theorems and generalize some of them to locally compact Hausdorff semi-topological semigroups based on Isaac Namioka's theorem.

\begin{no*}
In this appendix ``locally compact Hausdorff space'' will be abbreviated as ``l.c.$T_2$-space''.
\begin{quote}
Every l.c.$T_2$-space is of the second category and every closed or open subset of an l.c.$T_2$-space as a subspace itself is an l.c.$T_2$-space.
\end{quote}
\end{no*}

Differently with Ellis' original proof~\cite{E57} in which the group structure of $T$ plays an essential role, we will employ mainly the following basic joint continuity theorem due to Isaac~Namioka 1974~\cite[Theorem~1.2]{Nam} as our tool.

\begin{lem}[{Namioka~\cite{Nam}}]\label{lemB.1}
Let $G$ be an l.c.$T_2$-space or a separable Baire space and $X$ a compact $T_2$-space, and let $(Z,d)$ be a metric space. If a map $f\colon G\times X\rightarrow Z$ is unilaterally continuous, then there exists a dense $G_\delta$-set $R$ in $G$, such that $f$ is jointly continuous at each point of $R\times X$.
\end{lem}

Some alternate proofs and generalizations of Namioka's theorem have been given by, for examples, \cite{Tro79, Ken, Hel, Chr} and \cite[Lemma~1.36]{G03}.

Since $f\colon G\times X\rightarrow Z$ is unilaterally continuous, then $E_f\colon G\rightarrow C_{\mathfrak{p}}(X,Z)$ given by $g\mapsto f(g,\centerdot)$ is continuous, where $C_{\mathfrak{p}}(X,Z)$ is the space of all continuous functions from $X$ to $Z$ with the pointwise topology $\mathfrak{p}$. Thus Lemma~\ref{lemB.1} is a corollary of \cite[Lemma~4.2]{Aus} due to Troallic~\cite{Tro79}.

Under the setup of Lemma~\ref{lemB.1}, $C_u(X,Z)$ denotes the uniform space $C(X,Z)$ where the topology of uniform convergence on $C(X,Z)$ is induced by the standard supremum norm:
\begin{gather*}
\|\phi-\psi\|={\sup}_{x\in X}d(\phi(x),\psi(x))\quad \forall \phi,\psi\in C(X,Z).
\end{gather*}

It is very convenient to reformulate Lemma~\ref{lemB.1} in terms of functions spaces as follows. In our later application of this lemma, $Z$ will be the unit interval $I=[0,1]$ with the usual Euclidean metric.

\begin{lem}\label{lemB.2}
Let $G$ be an l.c.$T_2$-space and $X$ a compact $T_2$-space, and let $(Z,d)$ be a metric space. If $f\colon G\times X\rightarrow Z$ is unilaterally continuous, then there exists a dense $G_\delta$-set $R$ in $G$, such that the induced map
$F\colon G\rightarrow C_u(X,Z),\ t\mapsto f(t,\centerdot)\ \forall t\in G$,
is continuous at each point of $R$.
\end{lem}

\begin{proof}
This is just a consequence of \cite[Theorem~2.2]{Nam}; however we present its proof here for reader's convenience. Let $R$ be a dense $G_\delta$-set in $G$ given by Lemma~\ref{lemB.1}. Then for any $\tau\in R$, $f\colon G\times X\rightarrow Z$ is continuous at each point of $\{\tau\}\times X$. Since $X$ is compact, we see that $F$ is continuous at $\tau$ in the sense of the topology of uniform convergence on $C(X,Z)$. Indeed, given any $\varepsilon>0$, for any $x\in X$, there are open neighborhoods $U_x$ of $\tau$ in $G$ and $V_x$ of $x$ in $X$ such that
\begin{gather*}
d(f(\tau,y), f(t,y))<\varepsilon\quad \forall t\in U_x\textrm{ and }y\in V_x.
\end{gather*}
Choosing $x_1,\dotsc,x_n\in X$ so that $X=V_{x_1}\cup\dotsm\cup V_{x_n}$ and letting $U=\bigcap_{i=1}^n U_{x_i}$, it follows that
\begin{gather*}
\|F(\tau)-F(t)\|={\sup}_{x\in X}d(f(\tau,x),f(t,x))<\varepsilon\quad \forall t\in U.
\end{gather*}
This concludes the desired.
\end{proof}

A topological space $Y$ is called \textit{completely regular} iff for each member $y$ of $Y$ and each neighborhood $U$ of $y$ there is a continuous function $\alpha$ on $Y$ to the closed unit interval $I$ such that $\alpha(y)=0$ and $\alpha$ is identically $1$ on $Y\setminus U$. It is clear that the family $C(Y,I)$ of all continuous functions on a completely regular space $Y$ to the unit interval $I$ distinguishes points and closed sets in the sense that for closed subset $A$ of $Y$ and each point $y\in Y\setminus A$ there is an $\alpha\in C(Y,I)$ such that $\alpha(y)$ does not belong to the closure of $\alpha(A)$.

If $X$ is a completely regular $T_1$-space, then by the classical Embedding Lemma (cf.~\cite[Chapter~4]{Kel}) $X$ is homeomorphic to a subspace of the cube $Q=I^{C(X,I)}$. Therefore we can easily obtain the following

\begin{lem}\label{lemB.3}
Let $X$ be a completely regular $T_1$-space and $W$ a topological space. Then a map $f\colon W\rightarrow X$ is continuous at a point $w_0\in W$ if and only if $\alpha\circ f\colon W\rightarrow I$ is continuous at the point $w_0$ for each $\alpha\in C(X,I)$.
\end{lem}

With this lemma at hands, we do not need here to strengthen Lemma~\ref{lemB.2} by uniform space instead of a metric space $Z$ as \cite{Tro79, Law84} there.

Recall that $(T,X)$ is weakly equicontinuous iff $E(X)\subseteq C(X,X)$. Comparing with \cite[Corollary~6]{AAB98}, an interesting point of the following is that $X$ is not necessarily to be a metric space.

\begin{thm}\label{thm9.4}
Let $(T,X)$ be a weakly equicontinuous flow and $\mathrm{Tran}\,(T,X)=\{x\,|\,\textrm{cls}_XTx=X\}$. Then $\mathrm{Tran}\,(T,X)\subseteq\mathrm{Equi}\,(T,X)$.
\end{thm}

\begin{note}
Therefore by Lemma~\ref{lem1.6}, if $(T,X)$ is a minimal flow, then $(T,X)$ is equicontinuous if and only if it is weakly equicontinuous (cf.~\cite[Theorem~4.6]{Aus}).
\end{note}

\begin{proof}
Let $\alpha\in C(X,I)$ where $I=[0,1]$ and $\pi\colon X\times E(X)\rightarrow X$ by $(x,p)\mapsto p(x)$ which is unilaterally continuous. Then by Lemma~\ref{lemB.1}, there is an $x_0\in X$ such that $\alpha\circ\pi\colon X\times E(X)\rightarrow I$ is jointly continuous at each point of $\{x_0\}\times E(X)$. Thus by compactness of $E(X)$, for $\varepsilon>0$ there is $\delta_0\in\mathscr{U}_X$ such that $|\alpha\circ p(y)-\alpha\circ p(x_0)|<\varepsilon/3$ for all $y\in\delta_0[x_0]$ and $p\in E(X)$. Now let $x\in\mathrm{Tran}\,(T,X)$; there are $t\in T$ and $\delta\in\mathscr{U}_X$ such that $t(\delta[x])\subseteq\delta_0[x_0]$. Using $E(X)t=E(X)$, we can conclude that
$|\alpha\circ p(y)-\alpha\circ p(x)|<\varepsilon$ for all $y\in\delta[x]$ and $p\in E(X)$. This shows by Lemma~\ref{lemB.3} that $\pi\colon X\times E(X)\rightarrow X$ is jointly continuous at each point of $\mathrm{Tran}\,(T,X)\times E(X)$. Thus $(T,X)$ is equicontinuous at each point of $\mathrm{Tran}\,(T,X)$.
\end{proof}

\begin{sn}\label{def9.5}
Given any semigroup $T$ and non-empty set $X$, $T\times X\rightarrow X, (t,x)\mapsto tx$ is called an \textit{algebraic semiflow}\index{semiflow!algebraic} if $ex=x$ and $(st)x=s(tx)$ for all $x\in X$ and $s,t\in T$.

For example, let $\mathbb{I}$ be a minimal left ideal in $\beta T$, $u$ an idempotent in $\mathbb{I}$, and $G=u\mathbb{I}$ which is a group with identity $u$ by (4) of Lemma~\ref{lem6.3}. Then $G\times\mathbb{I}\rightarrow\mathbb{I},\ (g,p)\mapsto gp$, is an algebraic flow where $g\colon p\mapsto gp$ does not need to be continuous.
\end{sn}

Following the classical work of Ellis~\cite{E57} we now introduce a notion we will need in our later arguments.

\begin{sn}\label{def9.6}
Let $T$ be a semigroup with a topology $\mathfrak{S}$ and $X$ a compact $T_2$-space. Then an algebraic semiflow $(T,X)$ is called an \textit{Ellis semiflow} if
\begin{enumerate}
\item[(a)] $T$ is an l.c.$T_2$-space and a right-topological semigroup under $\mathfrak{S}$;

\item[(b)]  $(t,x)\mapsto tx$ of $T\times X$ to $X$ is unilaterally continuous.
\end{enumerate}
\end{sn}

In view of Definition~\ref{def9.6} we now introduce ``admissible'' time.

\begin{sn}[\cite{Law84}]\label{def9.7}
We say that an Ellis semiflow $(T,X)$ is \textit{admissible at an element $\tau\in T$} if
\begin{enumerate}
\item[(c)] $\mathrm{Int}_T\textrm{cls}_T{\tau\{t\in T\,|\,t\textrm{ is a surjection of }X\}}\not=\emptyset$.
\end{enumerate}
We shall say $(T,X)$ is \textit{admissible} if it is admissible at each element of $T$.

\begin{note}
In Ellis' setup that $T$ is a group, every Ellis flow of $T$ is admissible.
\end{note}
\end{sn}

Clearly each $T_2$ non-compact \textit{C}-semigroup $T$ is such that $T\setminus tT$ relatively compact and so $\mathrm{Int}_T\textrm{cls}_T{tT}\not=\emptyset$ for each $t\in T$ like the additive semigroups $\mathbb{N}$ and $\mathbb{R}_+$ (cf.~\cite{KM}).

By adapting ``transport'' technique of Troallic~\cite{Tro83, HT}, Lawson \cite{Law84} generalized Ellis' celebrated joint continuity theorem~\cite[Theorem~1]{E57} (also cf.~\cite[Theorem~4.7]{Aus}) as follows, for which we will present another concise proof based on Namioka~\cite{Nam}.

\begin{thm}[{Ellis-Lawson joint continuity theorem; cf.~\cite[Theorem~5.2]{Law84}}]\label{thm9.8}
Let $(T,X)$ be an Ellis semiflow. If it is admissible at an element $\tau\in T$ for $(T,X)$, then $(t,x)\mapsto tx$ of $T\times X$ to $X$ is jointly continuous at each point of $\{\tau\}\times X$.

\begin{note}
Therefore, if $(T,X)$ is an Ellis flow, then $(t,x)\mapsto tx$ of $T\times X$ to $X$ is jointly continuous~\cite[Theorem~1]{E57}.
\end{note}
\end{thm}

\begin{proof}
For simplicity, write $M=\{t\in T\,|\,t\textrm{ is a surjection of }X\}$.
In view of Lemma~\ref{lemB.3} with $W=T\times X$, to prove Theorem~\ref{thm9.8} it is sufficient to show that for any $\theta\in C(X,I)$, the induced function $\vartheta\colon (t,x)\mapsto\theta(tx)$ of $T\times X$ to $I$ is jointly continuous at each point of $\{\tau\}\times X$.
For that, by Lemma~\ref{lemB.2} with $G=T$ and $Z=I$, it follows that there exists a residual subset $R$ of $T$ such that at each point of $R$, the map $\varTheta$ induced by $\vartheta$,
\begin{gather*}
\varTheta\colon T\rightarrow C(X,I); \quad t\mapsto\vartheta(t,\centerdot)\ \forall t\in T,
\end{gather*}
is continuous under the topology of uniform convergence on $C(X,I)$. Next, we will prove that $\varTheta$ is continuous at $\tau$ under the topology of uniform convergence on $C(X,I)$.

Indeed, let $\tau$ be an arbitrary admissible element of $T$ and let $\{t_\gamma\,|\,\gamma\in\Gamma\}$ be a net in $T$ with $t_\gamma\to\tau$ under the topology of $\mathfrak{S}$. We need to show that $\|\varTheta(t_\gamma)-\varTheta(\tau)\|\to0$.

By condition (c), $R\cap\textrm{cls}_T{\tau M}\not=\emptyset$; and so it follows that we can choose an $a\in R$ with $\tau a_j\to a$ for some net $\{a_j\,|\, j\in J\}$ in $M$. Then $t_\gamma a_j\to\tau a_j$ for any $j\in J$ by condition (b).
Now given any $\varepsilon>0$, there exists a neighborhood $U$ of $a$ in $T$ such that $\|\varTheta(a)-\varTheta(t)\|<\varepsilon$ for each $t\in U$, because $\varTheta$ is continuous at the point $a\in R$.

Therefore, there exist two indices $j_0\in J$ and $\gamma_0\in\Gamma$ such that $\|\varTheta(t_\gamma a_j)-\varTheta(\tau a_j)\|<2\varepsilon$ if $j>j_0$ and $\gamma>\gamma_0$. Since $a_j\colon x\mapsto a_jx, j\in J$, is a surjection of $X$, then
\begin{align*}
\|\varTheta(t_\gamma)-\varTheta(\tau)\|&={\sup}_{x\in X}|\vartheta(t_\gamma x)-\vartheta(\tau x)|
={\sup}_{x\in X}|\vartheta(t_\gamma(a_jx))-\vartheta(\tau(a_jx))|\\
&=\|\varTheta(t_\gamma a_j)-\varTheta(\tau a_j)\|\\
&<2\varepsilon
\end{align*}
as $j>j_0$ in the directed index set $J$.
Thus $\|\varTheta(t_\gamma)-\varTheta(\tau)\|\to 0$ for $\varepsilon>0$ is arbitrary; and so $\varTheta$ is continuous at the point $\tau$ from $(T,\mathfrak{S})$ to $(C(X,I),\|\cdot\|)$.

This, of course, implies that $\vartheta\colon (t,x)\mapsto\theta(tx)$ of $T\times X$ to $I$ is jointly continuous at each point of $\{\tau\}\times X$. The proof of Theorem~\ref{thm9.8} is thus completed.
\end{proof}

Note that the group structure of $T$ plays a role in Namioka's proof of Ellis' joint continuity theorem (\cite[Theorem~1]{E57} and \cite[Theorem~3.1]{Nam}). From Theorem~\ref{thm9.8}, we can easily obtain the following four corollaries and Ellis' joint continuity theorem.

As the first simple application of Theorem~\ref{thm9.8}, we can obtain an affirmative answer to the following open question:
\begin{quote}
{\it Let $S$ be a compact $T_2$ semi-topological semigroup with a dense algebraic subgroup $G$. Suppose a net $g_\alpha^{}\to g$ in $G$. Does $g_\alpha^{-1}$ converges to $g^{-1}$ in $G$?} (See \cite[Question~10.3]{Law84}.)
\end{quote}

\begin{cor}\label{cor9.9}
Let $S$ be a compact $T_2$ semi-topological semigroup with a dense algebraic subgroup $G$. Then $G$ is a topological subgroup of $S$.
\end{cor}

\begin{proof}
Let $T=S, X=S$ and define $T\times X\rightarrow X$ by $(t,x)\mapsto tx$ and $X\times T\rightarrow X$ by $(x,t)\mapsto xt$. Since $G$ is a subgroup and dense in $S$, it follows that $\textrm{cls}_T{gG}=T=\textrm{cls}_T{Gg}$ for all $g\in G$. Thus $g\colon x\mapsto gx$ and $g\colon x\mapsto xg$ are surjections of $X$ for each $g\in G$ and further $T$ is admissible at each element $g\in G$. Then by Theorem~\ref{thm9.8}, $(t,x)\mapsto tx$ is continuous on $G\times X$ and $(x,t)\mapsto xt$ is continuous on $X\times G$. Now let $g_\alpha^{}\to x$ in $G$ and let $g_\alpha^{-1}\to y$ in $S$; then by the continuity, $xy=e=yx$. Whence $y=x^{-1}$. This concludes the proof of Corollary~\ref{cor9.9}.
\end{proof}

The interesting point of Corollary~\ref{cor9.9} is that $G$ as a subspace of $S$ is not necessarily locally compact so Ellis' theorem (cf.~Theorem~\ref{thm9.15} below) plays no role here.

\begin{cor}\label{cor9.10}
Let $T$ be a semigroup of continuous self-surjections of a compact $T_2$-space $X$; and let $\mathfrak{S}$ be a topology on $T$ such that $(T,X)$ is admissible. Then $(t,x)\mapsto tx$ of $T\times X$ into $X$ is jointly continuous.
\end{cor}

Given any integer $d\ge1$, the following corollary seems to be non-trivial because it is beyond Ellis' joint continuity theorem.

\begin{cor}\label{cor9.11}
Let $\mathbb{R}_+^d\times X\rightarrow X, (t,x)\mapsto tx$ be a separately continuous semiflow, where $(\mathbb{R}_+^d,+)$ is under the usual Euclidean topology.
If $X$ is minimal, then $(t,x)\mapsto tx$ is jointly continuous on $\mathbb{R}_+^d\times X$.
\end{cor}

\begin{proof}
Write $T=\mathbb{R}_+^d$, which is an additive abelian semigroup. First, under the discrete topology of $T$, $(T,X)$ becomes a minimal semiflow. Then by Corollary~\ref{cor3.3B}, it follows that for each $t\in T$, $x\mapsto tx$ is a continuous surjection of $X$. Therefore, under the Euclidean topology of $\mathbb{R}_+^d$, the following conditions are satisfied:
\begin{itemize}
\item[(a)] $T$ is a locally compact $T_2$-space; and $(t,x)\mapsto tx$ is separately continuous of $T\times X$ to $X$.
\item[(b)] The right translation $R_s\colon t\mapsto t+s$ of $T$ to itself is continuous, for each $s\in T$.
\item[(c)] $\textrm{Int}_{T}\textrm{cls}_T(\tau+\{t\,|\,\pi_t\textrm{ is a surjection of }X\})\not=\emptyset$, for each $\tau\in T$.
\end{itemize}
Then by Lawson's theorem (cf.~\cite[Theorem~5.2]{Law84} and also see Theorem~\ref{thm9.8}), $(t,x)\mapsto tx$ is jointly continuous on $T\times X$. This completes the proof of Corollary~\ref{cor9.11}.
\end{proof}

This corollary may be applied to two interesting cases. First, let $\mathbb{R}_+^d\times X\rightarrow X$ be a semiflow; then it is well known that the induced Ellis semiflow $\mathbb{R}_+^d\times E(X)\rightarrow E(X)$ is only separately continuous, not necessarily jointly continuous. However, for any minimal left ideal $\mathbb{I}$ of $E(X)$, $\mathbb{R}_+^d\times\mathbb{I}\rightarrow\mathbb{I}$ is a jointly continuous semiflow by Corollary~\ref{cor9.11}. Particularly, if $(\pi,\mathbb{R}_+^d,X)$ is distal, then $E(X)$ itself is a minimal left ideal in $E(X)$ (by Lemma~\ref{lem6.7}) so that $(\pi_*,\mathbb{R}_+^d,E(X))$ is a semiflow with the phase semigroup $\mathbb{R}_+^d$ under the usual topology.

Secondly, let $\beta\mathbb{R}_+^d$ be the Stone-\v{C}ech compactification of $\mathbb{R}_+^d$. Then $\beta\mathbb{R}_+^d$ is a compact Hausdorff right-topological semigroup in a natural manner and there is a natural separately continuous semiflow $\mathbb{R}_+^d\times\beta\mathbb{R}_+^d\rightarrow\beta\mathbb{R}_+^d$. Therefore, for any minimal left ideal $\mathbb{I}$ of $\beta\mathbb{R}_+^d$, $\mathbb{R}_+^d\times\mathbb{I}\rightarrow\mathbb{I}$ is a jointly continuous semiflow by Corollary~\ref{cor9.11}.

Let $C_{\mathfrak{p}}(X,X)$ denote the Hausdorff space $C(X,X)$ equipped with the topology $\mathfrak{p}$ of pointwise convergence. Clearly, $C_{\mathfrak{p}}(X,X)$ is a semi-topological semigroup, since the maps $R_g\colon f\mapsto f\circ g$ and $L_g\colon f\mapsto g\circ f$ of $C_{\mathfrak{p}}(X,X)$ to itself are continuous for each $g\in C_{\mathfrak{p}}(X,X)$.
Then for any semigroup $G$ of homeomorphisms on $X$, by an argument similar to the proof of \cite[Proposition~8.3]{Fur}, we can see that the closure $\mathrm{cls}_{C_{\mathfrak{p}}(X,X)}^{}G$ of $G$ in $C_{\mathfrak{p}}(X,X)$ is a subsemigroup of $C_{\mathfrak{p}}(X,X)$.

The following corollary is a generalization of \cite[Lemma~3]{E57} using different approach. There Ellis is for compact metric phase space $X$.

\begin{cor}\label{cor9.12}
Let $G$ be a group of self-homeomorphisms of a compact $T_2$-space $X$; and let $T=\mathrm{cls}_{C_{\mathfrak{p}}(X,X)}^{}G$. If $T$ is an l.c. subset of $C_{\mathfrak{p}}(X,X)$, then $(g,x)\mapsto gx$ of $G\times X$ to $X$ is jointly continuous, where $G$ is regarded as a subspace of $C_{\mathfrak{p}}(X,X)$.
\end{cor}

\begin{note}
If $G$ itself is a compact subset of $C_{\mathfrak{p}}(X,X)$, then $(G,X,\pi)$ is equicontinuous (cf.~\cite[Theorem~4.3]{Aus} and Theorem~\ref{cor6.11} before).
\end{note}

\begin{proof}
We consider $T\times X\rightarrow X$ defined by the evaluation map $(t,x)\mapsto tx$. As $\textrm{cls}_T{gG}=T$ for each $g\in G$, $T$ is admissible at each element of $G$. Thus Corollary~\ref{cor9.12} follows at once from Theorem~\ref{thm9.8}.
\end{proof}

We shall say that for a group $G$, an action $G\times X\rightarrow X$ is \textit{effective} if whenever $g\not=e$ for $g\in G$ then $gx\not=x$ for some $x\in X$. This is only a minor technical condition. If the action is not effective, let $F=\{t\in G\,|\,tx=x\ \forall x\in X\}$. Then $F$ is a closed (since $X$ is $T_2$) normal subgroup of $T$. The quotient group $G/F$ acts on $X$ by $(Ft)x=tx$, and this action is clearly effective. Therefore, we can assume that the action of $G$ on $X$ is effective.

Another consequence of Theorem~\ref{thm9.8} is the following

\begin{cor}\label{cor9.13}
Let $G\times X\rightarrow X$ be an effective flow with compact $T_2$ phase space $X$ and discrete phase group $G$. If $G$ is abelian, then $G$ is a topological subgroup of the enveloping semigroup $E(X)$ in the pointwise topology.
\end{cor}

\begin{proof}
Since $G$ effectively acts on $X$, we may see $G\subseteq E(X)$. Let $\Pi\colon E(X)\times G\rightarrow E(X)$ be defined by $\Pi\colon (f,g)\mapsto f\circ g$, which is separately continuous in the pointwise topology by noting that $f\circ g=g\circ f$ for any $f\in E(X)$ and $g\in G$ (cf.~\cite[(1) of Lemma~3.4]{Aus}). Clearly, $\Pi$ is effective. Write $E=E(X)$. Let $T=\mathrm{cls}_{C_{\mathfrak{p}}(E,E)}G$ where we have identified $G$ with $\{\Pi_g\,|\,g\in G\}$ such that $G\subset C_{\mathfrak{p}}(E,E)$.

On the other hand, it is well-known fact that the Ellis semigroup of $(E,G)$ is such that $E(E,G)\approx E$ (cf.~\cite[p.~55]{Aus}). Thus, for any $\xi\in E(E,G)\subseteq E^E$, $\xi\colon f\mapsto f\circ \xi$ of $E$ to $E$ is continuous in the pointwise topology, i.e., $\xi\in C_{\mathfrak{p}}(E,E)$. So, $T=E(E,G)$ is a compact Hausdorff subset of $C_{\mathfrak{p}}(E,E)$.

Therefore by Corollary~\ref{cor9.12}, it follows that $\Pi\colon E\times G\rightarrow E$ is jointly continuous in the pointwise topologies and so $G$ is a paratopological group in the pointwise topology. Moreover, if $g_n^{}\to g$ in $G$ and $g_n^{-1}\to f\in E$ with the involved pointwise topologies, then $g_n^{-1}\circ g_n=e$ and $g_n^{-1}\circ g_n\to f\circ g$ for $\Pi$ is continuous. Whence $f\circ g=e$ and then $f=g^{-1}$ since $G$ is a group. This implies that $G$ is a topological subgroup of $E$ in the pointwise topology.

The proof of Corollary~\ref{cor9.13} is thus completed.
\end{proof}

Notice that under the situation of Corollary~\ref{cor9.13}, while $G$ is abelian, $E(X)$ is not necessarily abelian (cf.~\cite[p.~55]{Aus}); otherwise, $E(X)$ becomes a compact $T_2$ semi-topological semigroup and then the conclusion of Corollary~\ref{cor9.13} follows at once from Corollary~\ref{cor9.9}.

The following is a slight generalization of a theorem of Ellis, in which the only new ingredient is condition (1) $\Rightarrow$ (3).

\begin{thm}\label{thm9.14}
Let $G$ be a group of self-homeomorphisms of a compact $T_2$-space $X$; and let $T=\mathrm{cls}_{C_{\mathfrak{p}}(X,X)}G$. Then the following conditions are pairwise equivalent.
\begin{enumerate}
\item[$(1)$] $T$ is a compact $T_2$-topological subsemigroup of $C_{\mathfrak{p}}(X,X)$.
\item[$(2)$] $T$ is a compact $T_2$-topological subgroup of $C_{\mathfrak{p}}(X,X)$.
\item[$(3)$] $G$ is equicontinuous on $X$.
\end{enumerate}
\end{thm}

\begin{noteB}
Example~\ref{exa6.24} shows that the statement of Theorem~\ref{thm9.14} is not true if $G$ is a \textit{semigroup} of homeomorphisms of $X$ in place of $G$ being a group.
\end{noteB}
\begin{noteB}
It is comparable with \cite[Theorems~3.3 and 4.4]{Aus}. Condition $(2)\Leftrightarrow(3)$ is just Ellis'~\cite[Theorem~3]{E57}. Here our proof is completely independent of Ellis~\cite{E57} and it is more concise than his one.
\end{noteB}

\begin{proof}
Condition $(1)\Rightarrow(3)$. Let $T$ be a compact $T_2$-topological subsemigroup of $C_{\mathfrak{p}}(X,X)$. Then $(f,g)\mapsto f\circ g$ of $T\times T$ to $T$ is continuous in the topology $\mathfrak{S}$ of pointwise convergence on $T$ inherited from $C_{\mathfrak{p}}(X,X)$. We will prove that $T\times X\rightarrow X, (t,x)\mapsto tx$ is jointly continuous.
According to Theorem~\ref{thm9.8}, it suffices to show that $(T,\mathfrak{S})$ is admissible. Obviously we only need to check condition (c). Indeed, since $G$ is a group consisting of homeomorphisms on $X$, hence $\textrm{cls}_T{gG}=T$ for all $g\in G$. Now for any $\tau\in T\setminus G$ and $t\in T$, take nets $\{\tau_i\}\subset G, \{t_i\}\subset T$ with $\tau_i\to\tau$ and $\tau_i\circ t_i=t$. By choosing a subnet of $\{t_i\}$ in the compact $T$ if necessary, we may assume $t_i\to f\in T$. Thus, $\tau\circ f=t$ and then $\tau T=T$ for each $\tau\in T$. Thus $(T,\mathfrak{S})$ is admissible.
Furthermore, $\pi$ is continuous on $T\times X$ and so $T$ is equicontinuous on $X$ since $T\times X$ is compact.

Condition $(3)\Rightarrow(2)$. Since $G$ is equicontinuous, hence $G$ is distal on $X$ and further $T$ is a compact $T_2$-space with a group structure (\cite[Theorem~1]{E58}).\footnote{The proof that $T$ has the group structure is somewhat involved in \cite{E57, E58}. Here is an easy direct argument. First by the equicontinuity of $G$, it follows that under the uniform topology $T$ is a compact $T_2$ semi-topological semigroup. Now for any $\xi\in T$ and net $\{t_n\}$ in $G$ with $t_n\to\xi$ uniformly, let $t_n^{-1}\to\eta$ uniformly. Then $e=t_nt_n^{-1}\to\xi\eta, e=t_n^{-1}t_n\to\eta\xi$ and thus $\xi^{-1}=\eta$. This shows that $T$ is a compact $T_2$ semi-topological group.}
Thus by Theorem~\ref{thm9.8}, it follows that the map $(u,v)\mapsto u\circ v$ of $T\times T$ to $T$ is continuous. Now let $\{t_i\}\subset T$ be a net with $t_i\to t$. If $t_i^{-1}\to r$, then $t_it_i^{-1}=e$ implies that $r=t^{-1}$. Thus $t_i^{-1}\to t^{-1}$.
Therefore $T$ is a compact group relative to the space $C_{\mathfrak{p}}(X,X)$.

Condition $(2)\Rightarrow(1)$. This is trivial by definitions.

The proof of Theorem~\ref{thm9.14} is thus completed.
\end{proof}

Finally we will simply reprove another classical theorem of Ellis using our Theorem~\ref{thm9.8} above as follows.

\begin{thm}[{\cite[Theorem~2]{E57}}]\label{thm9.15}
Let $G$ be an l.c.$T_2$-space with a group structure such that $(x,y)\mapsto xy$ of $G\times G$ to $G$ is separately continuous. Then $G$ is a topological group.
\end{thm}

\begin{proof}
Let $X$ be the one-point compactification of $G$ with point at infinity $\infty$. Then $G$ may be thought of as a subset of $C_{\mathfrak{p}}(X,X)$ by setting $g\infty=\infty$ and $\infty g=\infty$ for all $g\in G$. By Theorem~\ref{thm9.8}, it follows that
\begin{gather*}
G\times X\rightarrow X,\ (g,x)\mapsto gx\quad \textrm{and}\quad X\times G\rightarrow X,\ (x,g)\mapsto xg
\end{gather*}
are jointly continuous. Thus, $(x,y)\mapsto xy$ of $G\times G$ to $G$ is continuous.

Now let $g\in G$ and $\{g_\gamma\}$ a net in $G$ with $g_\gamma\to g$ in $G$. Since $X$ is compact, we may assume $g_\gamma^{-1}\to h$ in $X$. Thus by $g_\gamma g_\gamma^{-1}=e=g_\gamma^{-1}g_\gamma$, we see that $gh=e=hg$ and $h=g^{-1}\in G$. Therefore, the map $g\mapsto g^{-1}$ of $G$ to $G$ is continuous. The proof is completed.
\end{proof}

Comparing our independent self-closed proof of Theorem~\ref{thm9.15} with Ellis' presented in \cite{E57}, here we need not use \cite[Exercise~17]{Bou} which is not accessible for many readers. The proof that inversion is continuous is somewhat more involved in the available literature (see, e.g., \cite{E57a,E57} and \cite[p.~63]{Aus}).
Theorem~\ref{thm9.15} is comparable with \cite[Theorem~2]{M36} and \cite[Lemma, p.~982]{G48} where $G$ is a Polish space.

Following Definition~\ref{sn3.8B} a topological semigroup $T$ is a left \textit{C}-semigroup if and only if $T\setminus{sT}$ is relatively compact in $T$ for every $s\in T$.

\begin{thm}\label{thm9.16}
Let $X$ be a compact $T_2$-space and let $T$ be a non-compact l.c.$T_2$-topological semigroup consisting of self-surjections of $X$. If $T$ is a left \textit{C}-semigroup and $(t,x)\mapsto tx$ of $T\times X$ onto $X$ is separately continuous, then $(T,X)$ is a semiflow $($i.e. $(t,x)\mapsto tx$ is jointly continuous$)$.
\end{thm}

\begin{proof}
The conditions (a) and (b) of Definition~\ref{def9.6} evidently hold. Since $T$ is not compact, hence $T\setminus(\textrm{cls}_T(T\setminus{\tau T}))$ is a non-empty open set. Thus $\tau T$ has a non-empty interior. This implies condition (c) of Definition~\ref{def9.7}. Then our statement follows at once from Theorem~\ref{thm9.8}.
\end{proof}

Note that under the usual topology, $(\mathbb{R}_+^d,+)$, for $d\ge2$, is not a left \textit{C}-semigroup. Thus Corollary~\ref{cor9.11} has different flavor with Theorem~\ref{thm9.16}.

\begin{cor}\label{cor9.17}
Let $T$ be a non-compact l.c. \textit{C}-semigroup and $X$ a compact $T_2$-space. Suppose that $(T,X)$ is an invertible semiflow. If $t_n\to t$ in $T$ implies that $t_n^{-1}x\to t^{-1}x$ for all $x\in X$, then the reflection $(X,T)$ is a semiflow.
\end{cor}

\begin{proof}
By hypothesis, $(x,t)\mapsto xt=t^{-1}x$ is separately continuous. Then $(X,T)$ is a semiflow by Theorem~\ref{thm9.16}.
\end{proof}

\section*{\textbf{Acknowledgments}}
This project was supported by National Natural Science Foundation of China (Grant Nos. 11431012 and 11271183) and PAPD of Jiangsu Higher Education Institutions.



\begin{thebibliography}{10}
\addcontentsline{toc}{section}{References}
\bibitem{Aki}
   \newblock {E.~Akin},
   \newblock {\it On chain continuity}.
   \newblock {Discret. Contin. Dyn. Syst. \textbf{2} (1996), 111--120}.

\bibitem{AA}
   \newblock {E.~Akin and J.~Auslander},
   \newblock {\it Almost periodic sets and subactions in topological dynamics}.
   \newblock {Proc. Amer. Math. Soc. \textbf{131} (2003), 3059--3062}.

\bibitem{AAB}
   \newblock {E.~Akin, J.~Auslander and K.~Berg},
   \newblock {\it When is a transitive map chaotic?, in: Conference in Ergodic Theory and Probability},
   \newblock {Columbus, OH, 1993, in: Ohio State Univ. Math. Res. Inst. Publ. Vol.~\textbf{5}, de Gruyter, Berlin, 1996, pp.~25--40}.

\bibitem{AAB98}
   \newblock {E.~Akin, J.~Auslander and K.~Berg},
   \newblock {\it Almost equicontinuity and the enveloping semigroup},
   \newblock {Contemporary Math. \textbf{215} (1998), 75--81}.

\bibitem{Aus}
   \newblock {J.~Auslander},
   \newblock {\it Minimal Flows and Their Extensions}.
   \newblock {North-Holland Math. Studies Vol. \textbf{153}. North-Holland, Amsterdam, 1988}.

\bibitem{ACD}
   \newblock {J.~Auslander, K.~Cao and X.~Dai},
   \newblock {\it Universal object is coalescent in a category which is closed under taking ordinal-inverse limits},
   \newblock {Preprint, 2018}.

\bibitem{AF}
   \newblock {J.~Auslander and H.~Furstenberg},
   \newblock {\it Product recurrence and distal points}.
   \newblock {Trans. Amer. Math. Soc. \textbf{343} (1994), 221--232}.

\bibitem{AM}
   \newblock {J.~Auslander and N.~Markley},
   \newblock {\it Locally almost periodic minimal flows}.
   \newblock {J. Difference Equ. Appl. \textbf{15} (2009), 97--109}.
\bibitem{AY}
   \newblock {J.~Auslander and J.~Yorke},
   \newblock {\it Interval maps, factors of maps and chaos}.
   \newblock {T\^{o}hoku Math. J. \textbf{32} (1980), 177--188}.
\bibitem{Bou}
   \newblock {N.~Bourbaki},
   \newblock {\it Topologie G\'{e}n\'{e}rale, Chapitre IX}.
   \newblock {Actualit\'{e}s Sci. Ind. No. \textbf{1045}, Hermann, Paris, 1948}.

\bibitem{CD}
   \newblock {B.~Chen and X.~Dai},
   \newblock {\it On uniformly recurrent motions of topological semigroup actions}.
   \newblock {Discret. Contin. Dyn. Syst. \textbf{36} (2016), 2931--2944}.

\bibitem{Chr}
   \newblock {J.\,P.\,R.~Christensen},
   \newblock {\it Joint continuity of separately continuous functions}.
   \newblock {Proc. Amer. Math. Soc. \textbf{82} (1981), 455--461}.

\bibitem{Ch62}
   \newblock {H.~Chu},
   \newblock {\it On universal minimal transformation groups}.
   \newblock {Illinois J. Math. \textbf{6} (1962), 317--326}.

\bibitem{C63}
   \newblock {J.\,P.~Clay},
   \newblock {\it Proximity relations in transformation groups}.
   \newblock {Trans. Amer. Math. Soc. \textbf{108} (1963), 88--96}.

\bibitem{C63-D}
   \newblock {J.\,P.~Clay},
   \newblock {\it Variations on equicontinuity}.
   \newblock {Duke Math. J. \textbf{30} (1963), 423--431}.

\bibitem{Dai}
   \newblock {X.~Dai},
   \newblock {\it Notes on almost periodic point of transformation group with any phase space}.
   \newblock {Preprint 2018}.

\bibitem{DG}
   \newblock {X.~Dai and E.~Glasner},
   \newblock {\it On universal minimal proximal flows of topological groups}.
   \newblock {Proc. Amer. Math. Soc. (in press)}.

\bibitem{DL}
   \newblock {X.~Dai and H.~Liang},
   \newblock {\it On Galvin's theorem for compact Hausdorff right-topological semigroups with dense topological centers}.
   \newblock {Sci. China Math. \textbf{60} (2017), 2421--2428}.

\bibitem{DT}
   \newblock {X.~Dai and X.~Tang},
   \newblock {\it Devaney chaos, Li-Yorke chaos, and multi-dimensional Li-Yorke chaos for topological dynamics}.
   \newblock {J. Differential Equations \textbf{263} (2017), 5521--5553}.

\bibitem{DX}
   \newblock {X.~Dai and Z.~Xiao},
   \newblock {\it Equicontinuity, uniform almost periodicity, and regionally proximal relation for topological semiflows}.
   \newblock {Topology Appl. \textbf{231} (2017), 35--49}.

\bibitem{Day}
   \newblock {M.~Day},
   \newblock {\it Amenable semigroups}.
   \newblock {Illinois J. Math. \textbf{1} (1957), 509--544}.

\bibitem{Dix}
   \newblock {J.~Dixmier},
   \newblock {\it Les moyennes invariantes dans les semi-groupes et leures applications}.
   \newblock {Acta Sci. Math. Szeged. \textbf{12} (1950), 213--227}.

\bibitem{Eb}
   \newblock {F.~Eberlein},
   \newblock {\it Abstract ergodic theorems and weak almost periodic functions}.
   \newblock {Trans. Amer. Math. Soc. \textbf{67} (1949), 217--240}.

\bibitem{EE}
   \newblock {D.~Ellis and R.~Ellis},
   \newblock {\it Automorphisms and Equivalence Relations in Topological Dynamics}.
   \newblock {London Math. Soc. Lecture Note Ser. \textbf{412}, Cambridge Univ. Press, 2014}.

\bibitem{EEN}
   \newblock {D.~Ellis, R.~Ellis and M.~Nerurkar},
   \newblock {\it The topological dynamics of semigroup actions}.
   \newblock {Trans. Amer. Math. Soc. \textbf{353} (2001), 1279--1320}.

\bibitem{E57a}
   \newblock {R.~Ellis},
   \newblock {\it A note on the continuity of the inverse}.
   \newblock {Proc. Amer. Math. Soc. \textbf{8} (1957), 372--373}.

\bibitem{E57}
   \newblock {R.~Ellis},
   \newblock {\it Locally compact transformation groups}.
   \newblock {Duke Math. J. \textbf{24} (1957), 119--126}.

\bibitem{E58}
   \newblock {R.~Ellis},
   \newblock {\it Distal transformation groups}.
   \newblock {Pacific J. Math. \textbf{8} (1958), 401--405}.

\bibitem{E59}
   \newblock {R.~Ellis},
   \newblock {\it Equicontinuity and almost periodic functions}.
   \newblock {Proc. Amer. Math. Soc. \textbf{10} (1959), 637--643}.

\bibitem{E60}
   \newblock {R.~Ellis},
   \newblock {\it Universal minimal sets}.
   \newblock {Proc. Amer. Math. Soc. \textbf{11} (1960), 540--543}.

\bibitem{E69}
   \newblock {R.~Ellis},
   \newblock {\it Lectures on topological dynamics}.
   \newblock {Benjamin, New York, 1969}.

\bibitem{EN}
   \newblock {R.~Ellis and M.~Nerurkar},
   \newblock {\it Weakly almost periodic flows}.
   \newblock {Trans. Amer. Math. Soc. \textbf{313} (1989), 103--119}.

\bibitem{For}
   \newblock {M.\,K.~Fort, Jr},
   \newblock {\it Category theorems}.
   \newblock {Fund. Math. \textbf{42} (1955), 276--288}.

\bibitem{F63}
   \newblock {H.~Furstenberg},
   \newblock {\it The structure of distal flows}.
   \newblock {Amer. J. Math. \textbf{85} (1963), 477--515}.

\bibitem{Fur}
   \newblock {H.~Furstenberg},
   \newblock {\it Recurrence in Ergodic Theory and Combinatorial Number Theory}.
   \newblock {Princeton University Press, Princeton, New Jersey, 1981}.

\bibitem{G76}
   \newblock {E.~Glasner},
   \newblock {\it Proximal Flows}.
   \newblock {Lecture Notes in Math. \textbf{517}, Springer-Verlag, 1976}.

\bibitem{G03}
   \newblock {E.~Glasner},
   \newblock {\it Ergodic Theory via Joinings}.
   \newblock {Math. Surveys and Monographs Vol. \textbf{101}, Amerian Math. Soc., 2003}.

\bibitem{GM}
   \newblock {S.~Glasner and D.~Maon},
   \newblock {\it Rigidity in topological dynamics}.
   \newblock {Ergod. Th. $\&$ Dyn. Syst. \textbf{9} (1989), 309--320}.

\bibitem{G}
   \newblock {W.\,H.~Gottschalk},
   \newblock {\it Almost periodic points with respect to transformation semi-groups}.
   \newblock {Annals of Math. \textbf{47} (1946), 762--766}.

\bibitem{G48}
   \newblock {W.\,H.~Gottschalk},
   \newblock {\it Transitivity and equicontinuity}.
   \newblock {Bull. Amer. Math. Soc. \textbf{54} (1948), 982--984}.

\bibitem{GH}
   \newblock {W.\,H.~Gottschalk and G.\,A.~Hedlund},
   \newblock {\it Topological Dynamics}.
   \newblock {Amer. Math. Soc. Coll. Publ. Vol. \textbf{36}, Amer. Math. Soc., Providence, R.I., 1955}.

\bibitem{HT}
   \newblock {G.~Hansel and J.-P.~Troallic},
   \newblock {\it Points de continuite a gauche d'une action de semigroupe}.
   \newblock {Semigroup Forum \textbf{26} (1983), 205--214}.

\bibitem{Hel}
   \newblock {D.~Helmer},
   \newblock {\it Continuity of semigroup actions}.
   \newblock {Semigroup Forum \textbf{23} (1981), 153--188}.


\bibitem{Kel}
   \newblock {J.\,L.~Kelley},
   \newblock {\it General Topology}.
   \newblock {GTM \textbf{27}, Springer-Verlag, New York Berlin Heidelberg, 1955}.

\bibitem{Ken}
   \newblock {P.~Kenderov},
   \newblock {\it Dense strong continuity of pointwise mappings}.
   \newblock {Pacific J. Math. \textbf{89} (1980), 111--130}.

\bibitem{KM}
   \newblock {E.~Kontorovich and M.~Megrelishvili},
   \newblock {\it A note on sensitivity of semigroup actions}.
   \newblock {Semigroup Forum \textbf{76} (2008), 133--141}.

\bibitem{Law84}
   \newblock {J.\,D.~Lawson},
   \newblock {\it Points of continuity for semigroup actions}.
   \newblock {Trans. Amer. Math. Soc. \textbf{284} (1984), 183--202}.

\bibitem{M36}
   \newblock {D.~Montgomery},
   \newblock {\it Continuity in topological groups}.
   \newblock {Bull. Amer. Math. Soc. \textbf{42} (1936), 879--882}.

\bibitem{Nam}
   \newblock {I.~Namioka},
   \newblock {\it Separate continuity and joint continuity}.
   \newblock {Pacific J. Math. \textbf{51} (1974), 515--531}.

\bibitem{NS}
   \newblock {V.\,V.~Nemytskii and V.\,V.~Stepanov},
   \newblock {\it Qualitative Theory of Differential Equations}.
   \newblock {Princeton University Press, Princeton, New Jersey, 1960}.

\bibitem{SS74}
   \newblock {R.\,J.~Sacker and G.\,R.~Sell},
   \newblock {\it Finite extension of minimal transformations groups}.
   \newblock {Trans. Amer. Math. Soc. \textbf{190} (1974), 325--334}.

\bibitem{SS}
   \newblock {R.\,J.~Sacker and G.\,R.~Sell},
   \newblock {\it Lifting properties in skew-product flows with applications to differential equations}.
   \newblock {Mem. Amer. Math. Soc. \textbf{11} (1977), no. 190, iv+67 pp}.

\bibitem{SSY}
   \newblock {G.\,R.~Sell, W.~Shen and Y.~Yi},
   \newblock {\it Topological dynamics and differential equations}.
   \newblock {Contemp. Math. \textbf{215} (1998), 279--298}.

\bibitem{SY}
   \newblock {W.~Shen and Y.~Yi},
   \newblock {\it Almost automorphic and almost periodic dynamics in skew-product semiflows}.
   \newblock {Mem. Amer. Math. Soc. \textbf{136} (1998), no. 647, x+93 pp}.

\bibitem{Tro79}
   \newblock {J.-P.~Troallic},
   \newblock {\it Espaces fonctionnels et th\'{e}or\`{e}mes de I.~Namioka}.
   \newblock {Bull. Soc. Math. France \textbf{107} (1979), 127--137}.

\bibitem{Tro83}
   \newblock {J.-P.~Troallic},
   \newblock {\it Semigroupes semitopologiques et presqueperiodicite}.
   \newblock {(Proc. Conf. Semigroups, Oberwolfach 1981), Lecture Notes in Math. \textbf{998}, Springer-Verlag, Berlin New York, 1983, pp.~239--251}.

\bibitem{Tro}
   \newblock {S.\,L.~Troyanski},
   \newblock {\it On locally uniformly convex and differentiable norms in certain nonseparable Banach spaces}.
   \newblock {Studia Math. \textbf{37} (1971), 173--180}.

\bibitem{V70}
   \newblock {W.\,A.~Veech},
   \newblock {\it Point-distal flows}.
   \newblock {Amer. J. Math. \textbf{92} (1970), 205--242}.
\bibitem{V77}
   \newblock {W.\,A.~Veech},
   \newblock {\it Topological dynamics}.
   \newblock {Bull. Amer. Math. Soc. \textbf{83} (1977), 775--830}.
\end{thebibliography}
\end{document}